\documentclass[12pt, reqno]{amsart}
\usepackage[margin=2cm]{geometry}
\usepackage{subfig}
\usepackage{bm}

\usepackage[colorlinks=true,linkcolor=blue,citecolor=blue,urlcolor=blue,citebordercolor={0 0 1},urlbordercolor={0 0 1},linkbordercolor={0 0 1}]{hyperref} 

\usepackage{amsmath,amssymb,amstext,verbatim,amsthm,mathrsfs,tikz-cd}
\usepackage[all]{xy}

\theoremstyle{plain}
\newtheorem{thm}{Theorem}[section]
\newtheorem{theorem}[thm]{Theorem}
\newtheorem{cor}[thm]{Corollary}
\newtheorem{lemma}[thm]{Lemma}
\newtheorem{prop}[thm]{Proposition}

\newtheorem{definition}[thm]{Definition}

\newtheorem{remark}[thm]{Remark}

\theoremstyle{remark}
\newtheorem{example}[thm]{Example}

\makeatletter
\def\makeCal#1{\expandafter\newcommand\csname c#1\endcsname{\mathcal{#1}}}
\def\makeBB#1{\expandafter\newcommand\csname b#1\endcsname{\mathbb{#1}}}
\def\makeFrak#1{\expandafter\newcommand\csname f#1\endcsname{\mathfrak{#1}}}
\count@=0 \loop \advance\count@ 1 \edef\y{\@Alph\count@}
\expandafter\makeCal\y \expandafter\makeBB\y \expandafter\makeFrak\y
\ifnum\count@<26 \repeat

\newcommand {\Hom}{\operatorname{Hom}}

\newcommand {\Aut}{\operatorname{Aut}}

\newcommand{\tensor}{\otimes}

\newcommand{\Pic}{\operatorname{Pic}}

\renewcommand{\leq}{\leqslant}

\newcommand{\Spec}{\operatorname{Spec}}
\newcommand{\Ext}{\operatorname{Ext}}

\newcommand{\Coh}{\operatorname{Coh}}

\newcommand{\D}{D}

\newcommand{\bs}{\mathbf{s}}

\begin{document}

\title{Geometric helices on del Pezzo surfaces from tilting}
\date{}
\author{Pierrick Bousseau}
\address{The Mathematical Institute, University of Oxford, Oxford, OX2 6GG, UK}
\email{pierrick.bousseau@maths.ox.ac.uk}

\begin{abstract}
    We prove that all geometric helices in the derived category of coherent sheaves on a del Pezzo surface are  related  by a sequence of elementary operations: rotation, shifting, orthogonal reordering, tensoring by a line bundle, and tilting.  As a consequence, any two non-commutative crepant resolutions of the affine cone over a del Pezzo surface are related by mutations. The proof relies on a geometric interpretation of tilting operations as cluster transformations acting on toric models of a log Calabi--Yau surface mirror to the del Pezzo surface. 
\end{abstract}

\maketitle

\setcounter{tocdepth}{1}

%%%% These 2 lines to print the extended TOC
%\setcounter{tocdepth}{4}
%\setcounter{secnumdepth}{4}
%%%%
\tableofcontents

%%%%%%%%%%%%%%%%%%%%%%%%%%%%%%%%%%%%%%%%%%%%%%%
%%%%%%%%%%%%%%%%%%%%%%%%%%%%%%%%%%%%%%%%%%%%%%%
%%%%%%%%%%%%%%%%%%%%%%%%%%%%%%%%%%%%%%%%%%%%%%%
%%%%%%%%%%%%%%%%%%%%%%%%%%%%%%%%%%%%%%%%%%%%%%%

\section{Introduction}

\subsection{Motivation and main result}
Let $Z$ be a del Pezzo surface over $\bC$ and $D(Z)$ denote its bounded derived category of coherent sheaves. Let $n$ be the rank of the Grothendieck group $K(Z)$ of $Z$, so that $n=\rho+2$, where $\rho$ is the Picard rank of $Z$.  
A \emph{geometric helix} on $Z$, in the sense of  \cite{BS}, is a sequence of objects $\bH= (E_i)_{i \in \bZ}$ in $\D(Z)$ satisfying the following conditions:
\begin{itemize}
\item[(i)] for each $i\in \bZ$, the corresponding thread $(E_{i+1}, \ldots, E_{i+n})$ is a full exceptional collection,
\item[(ii)] $E_{i+n}=E_i\tensor \omega_Z^{-1}$ for all $i\in \bZ$,
\item[(iii)] for all $i<j$, $\Hom^k(E_i,E_j)=0$ unless $k=0$. 
\end{itemize}
A thread $\bE=(E_{i+1}, \dots, E_{i+n})$ in a geometric helix is called a \emph{very strong} exceptional collection.

Geometric helices and very strong exceptional collections are of particular interest because they provide an algebraic description of the derived category $D(X)$ of the non-compact Calabi--Yau 3-fold $X$ given by the total space of the canonical line bundle of $Z$. Indeed, let $\bE=(E_1,\dots, E_n)$ be a very strong exceptional collection on $Z$, and let $\pi: X \rightarrow Z$ be the natural projection. 
By \cite[Theorems 1.6-1.7]{BS}, setting $T:= \bigoplus_{i=1}^n \pi^\star E_i$, the algebra $B(\bE):=\mathrm{End}(T)$ is a noetherian 
$CY_3$ algebra isomorphic to the path algebra of a non-degenerate quiver with potential $(Q,W)$. Moreover, there is an equivalence 
\[ \Phi_\bE: D(B(\bE)) \xrightarrow{\sim} D(X)\,,\] 
where $D(B(\bE))$ is the bounded derived category of finitely-generated $B(\bE)$-modules -- see also \cite[\S 4.2]{segal} and 
\cite[\S 4.2]{deTh_vdB}. Descriptions of $D(X)$ in terms of quivers with potential play an important role in both enumerative geometry \cite{arguz2026mock, ABquivers, BMP, bridgeland2024invariant} and theoretical physics \cite{aspinwall2, aspinwall, closset, clossetu, Her, Her_Karp}. 

Geometric helices on del Pezzo surfaces always exist, but they are far from unique. 
Starting from a geometric helix 
$\bH$, one can construct many others by applying standard operations in the derived category, such as rotation, shifting, orthogonal reordering, and tensoring by a line bundle; see \S\ref{sec_geometric_tilting} for details. These operations do not change the associated algebra $B(\bE)$ or the quiver with potential $(Q, W)$. More interesting are the \emph{tilting operations} introduced by Bridgeland–Stern \cite{BS}, building on earlier work of Herzog \cite{Her} in the theoretical physics literature. Tilting operations act non-trivially on the associated algebra  $B(\bE)$ and induce combinatorial mutations of the quiver with potential \cite{DWZ}.
The main result of this paper, stated as Theorem \ref{thm_main} in the main body, shows that tilting operations, together with the basic operations above, suffice to connect any two geometric helices. In particular, any two quivers with potential associated with geometric helices on a del Pezzo surface are related by mutations.

\begin{theorem} \label{thm_main_intro}
    Let $Z$ be a del Pezzo surface. Any two geometric helices on $Z$ are related by a sequence of the following operations:
     rotation, shifting, orthogonal reordering, tensoring by a line bundle, and tilting.
\end{theorem}

Theorem \ref{thm_main_intro} also has a natural application in the context of non-commutative crepant resolutions. Recall that, following Van den Bergh \cite{bergh2002non}, a non-commutative crepant resolution of a Gorenstein algebra $R$ is a $R$-algebra $\Lambda =\mathrm{End}_R(M)$, where $M$ is a finitely generated reflexive $R$-module, such that $\Lambda$ is Cohen-Macaulay as $R$-module and is of finite global dimension. A general theory of mutations of non-commutative crepant resolutions has been developed in \cite{hara2026stability, IyamaReiten, IyamaWemyss, IWtits, vdB_ICM}.
For a del Pezzo surface $Z$, consider the Gorenstein ring $R_Z:= \bigoplus_{k \in \bZ_{\geq 0}} H^0(Z, \omega_Z^{-k})$. Geometrically, $\Spec R_Z$ is the affine Gorenstein canonical 3-fold singularity obtained from $X$ by contracting the zero section $Z$ to a point. According to
\cite[Proposition 7.2]{bergh2002non}, the algebra $B(\bE)$ is a non-commutative crepant resolution of $R_Z$  for every very strong exceptional collection $\bE$ on $Z$. Moreover, tilting operations of the helix generated by $\bE$ correspond to mutations of non-commutative crepant resolutions. Recently, Nordskova proved \cite[Theorem II.3.1]{nord} -- see also \cite[Theorem 1.1]{NvdB} --
that any non-commutative crepant resolution of the completion $\widehat{R}_Z$ of $R_Z$ is Morita equivalent to an algebra $B(\bE)$ for some very strong exceptional collection $\bE$ on $Z$.
Consequently, Theorem \ref{thm_main_intro} implies the following corollary:

\begin{cor} \label{cor_intro}
Let $Z$ be a del Pezzo surface. Any two non-commutative crepant resolutions of the completion $\widehat{R}_Z$ of $R_Z= \bigoplus_{k \in \bZ_{\geq 0}} H^0(Z, \omega_Z^{-k})$ are related by a sequence of mutations.
\end{cor}

Corollary \ref{cor_intro} was conjectured by Nordskova in \cite[Conjecture II.1.18]{nord}, who also proved it in the cases $Z=\bP^2$ and $Z=\bP^1 \times \bP^1$. Even in these special cases, our proof of Theorem \ref{thm_main_intro}
and Corollary \ref{cor_intro} differs from the approach of \cite{nord}\footnote{While this paper was being completed, the author was informed that Nordskova and Van den Bergh independently obtained a different proof of Corollary \ref{cor_intro}, now available in \cite{NvdB}.}. More generally, it is conjectured in \cite[Conjecture 1.3-Theorem 1.4]{hara2026stability} that any two non-commutative crepant resolutions of a 3-dimensional complete local Gorenstein isolated singularity are related by mutations. This has been proven by Iyama-Wemyss for terminal singularities in \cite{IWtits}. Corollary \ref{cor_intro} establishes the conjecture for a class of non-terminal singularities.

\subsection{Structure of the proof}
Let $\bE=(E_1, \dots, E_n)$ be a very strong exceptional collection on $Z$. To prove Theorem \ref{thm_main_intro} we consider the dual exceptional collection $\bF=(F_n, \dots, F_1)$ and the corresponding basis $\bs(\bE)=\{[F_i]\}_{1 \leq i \leq n}$ of $K(Z)$. Tilting operations on the helix generated by $\bE$ induce seed mutations on $\bs(\bE)$, viewed as a seed in the sense of cluster algebras. We prove that the seeds $\bs(\bE)$ are of $q$-Painlev\'e type in the sense of \cite{Mizuno}. In particular, $\bs(\bE)$ determines a lattice
polygon in $\bR^2$, referred to as a T-polygon 
-- see \S\ref{sec_q_painleve}. 

A deep combinatorial result of Kasprzyk--Nill--Prince \cite{KNP} classifies T-polygons up to mutations and action of $GL(2,\bZ)$. A simpler and more conceptual proof of this result was later obtained by Lutz \cite{lutz} using a factorization result in the birational geometry of log Calabi--Yau surfaces \cite[Proposition 3.27]{hacking_keating}--\cite[Theorem 1]{blanc}. The relevant log Calabi--Yau surfaces are constructed from the seeds $\bs(\bE)$ via an explicit blow-up construction of toric surfaces, following the work of Gross--Hacking--Keel \cite{GHKbir} on the birational geometry of cluster algebras. These log Calabi--Yau surfaces are classically related to del Pezzo surfaces by mirror symmetry \cite{mirror_orbifold, AKO, corti2023cluster, GGLP, gugiatti2025mirrors, lutz}.
The factorization result \cite[Proposition 3.27]{hacking_keating}--\cite[Theorem 1]{blanc} shows that different presentations of a log Calabi--Yau surface as blow-ups of toric varieties are related by a sequence of elementary cluster mutations.
In Theorem \ref{thm_q_P}, based on the geometric proof of \cite{lutz}, we refine the classification of T-polygons by allowing only linear transformations in $SL(2,\bZ)$ instead of all $GL(2,\bZ)$.

Using the classification of T-polygons up to mutations, we reduce the problem to the following question. Let
$\bE=(E_1, \dots, E_n)$ and $\bE'=(E_1', \dots, E_n')$
be two very strong exceptional collections with the same ranks $r(E_i)=r(E_i')$ and the same anticanonical degrees $d(E_i)=d(E_i')$ for all $1 \leq i \leq n$. Then, we must show that $\bE$ and $\bE'$ are related by a sequence of elementary transformations. 

To do this, we show that the seed $\bs(\bE)$ and $\bs(\bE')$ are related by an element of the orthogonal group of $Z$, defined as the group of isometries of $K(Z)$ with respect to the Euler form that preserve both the rank and the anticanonical degree. In Proposition \ref{prop_weyl_affine_gen}, we show that $O(Z)$ admits a semi-direct decomposition  $O(Z)=W(Z)\ltimes \Pic^0(Z)$, where $W(Z)$ is the Weyl group generated by the reflections in the $(-2)$-curve classes, and $\Pic^0(Z)$ is the group of line bundles orthogonal to the canonical divisor. 

To conclude the proof, we establish the existence of an action of the Weyl group $W(Z)$, and more generally of an affine Weyl group $\widehat{W}(Z)$, on very strong exceptional collections by tilting operations and orthogonal reordering. This result is obtained as a special case of a more general result on the geometry of log Calabi--Yau surfaces, Theorem \ref{thm_weyl}, which identifies a combinatorial Weyl group defined in terms of seed mutations with a geometric Weyl group defined in terms of $(-2)$-curve classes on log Calabi--Yau surfaces. 
The factorization result \cite[Proposition 3.27]{hacking_keating}--\cite[Theorem 1]{blanc} again plays a crucial role in the argument.

The structure of the paper is as follows.
In \S \ref{sec_seeds_mutations}, we review the correspondence between seed mutations and the geometry of log Calabi--Yau surfaces, and prove
in Theorem \ref{thm_weyl} the general comparison between combinatorial and geometric Weyl groups.
In \S \ref{sec_weyl_del_pezzo}, we study the Weyl group, the affine Weyl group, and the orthogonal group of a del Pezzo surface $Z$, all acting by isometries on $K(Z)$. In particular, we establish the semi-direct decomposition $O(Z)=W(Z)\ltimes \Pic^0(Z)$
 in Proposition \ref{prop_weyl_affine_gen}.
 In \S \ref{sec_very_strong}, we review geometric helices, very strong exceptional collections, and tilting operations. Finally, in
\S \ref{sec_final}, we prove that the seeds $\bs(\bE)$ associated to very strong exceptional collections are of $q$-Painlev\'e type, and we conclude the proof of the main result, Theorem \ref{thm_main}.

\subsection{Related works}

\subsubsection{Mutations of exceptional collections} For historical reasons, the word mutation is used in two distinct contexts: for exceptional collections and for quivers. These two notions are not compatible in general. As discussed above, tilting of very strong exceptional collections on del Pezzo surfaces corresponds to quiver mutations. 
As reviewed in \S \ref{sec_geometric_tilting}, tilting operations can be interpreted as particular sequences of mutations of exceptional collections that preserve the very strong condition. In contrast, individual mutations of exceptional collections do not generally preserve this property.  Kuleshov--Orlov \cite[Theorem 7.7]{KO} proved that any two full exceptional collections of sheaves on a del Pezzo surface are related by a sequence of mutations -- see also \cite[\S 6]{GorKul} for an exposition of the proof.
This differs from Theorem \ref{thm_main_intro}, which concerns connecting very strong exceptional collections via tilting operations. Moreover, our proof of Theorem \ref{thm_main_intro} does not rely on \cite[Theorem 7.7]{KO} -- see Remark \ref{rem_KN}.

\subsubsection{Hille--Perling polygons and Hacking bundles} 
In this paper, we associate to each very strong exceptional collection on a del Pezzo surface a seed of 
$q$-Painlevé type, and hence a T-polygon in 
$\bR^2$. An alternative construction, due to Hille--Perling \cite{HP,perling}, also assigns a T-polygon to a very strong exceptional collection.
The two constructions are not obviously the same.  In particular, the T-polygons arising from the Hille--Perling construction are defined only up to the action of $GL(2,\mathbb{Z})$, whereas for our purposes it is essential that our T-polygon is defined without this ambiguity. Nevertheless, we expect that the two constructions produce T-polygons that agree up to the action of $GL(2,\mathbb{Z})$. Furthermore, it follows from \cite{KNP,lutz} that T-polygons modulo 
$GL(2,\bZ)$ are in one-to-one correspondence with 
$\bQ$-Gorenstein degenerations of del Pezzo surfaces to toric Fano surfaces. This suggests a natural connection between our results and Hacking’s construction of exceptional vector bundles via degenerations to toric surfaces \cite{hacking, hacking2, tevelev2022categorical, urzua2025wahl}. Finally, T-polygons and their mutations also play a significant role in the symplectic geometry of del Pezzo surfaces and their mirrors
\cite{evans_smith, hacking_keating, hacking_keating_2, vianna0, vianna}.
We leave the exploration of these questions for future work.

\subsubsection{Seiberg duality in physics}
Quivers with potential describing the derived category $D(Z)$ of a del Pezzo surface $Z$ appear naturally in theoretical physics. For instance, they describe the four-dimensional $\mathcal{N}=1$ superconformal field theories living on a stack of D3 branes placed at the tip of the cone over the del Pezzo surface, which are related via the AdS/CFT correspondence to IIB string theory on $\mathrm{AdS}_5 \times X_5$, where $X_5$ is a five-dimensional Sasaki-Einstein manifold \cite{Her, Her_Karp}. In another direction, these quivers with potential define supersymmetric quantum mechanics that describe BPS particles in the four-dimensional $\mathcal{N}=2$ theories obtained from M-theory on $\bR^4 \times S^1 \times X$  \cite{closset, clossetu}.
Tilting operations on geometric helices, that is, quiver mutations, correspond to Seiberg dualities for the four dimensional $\mathcal{N}=1$ theories or for the supersymmetric quantum mechanics. Therefore, Theorem \ref{thm_main_intro} can be reformulated in these contexts as stating that any two of these physical theories can be related by a sequence of Seiberg dualities.

%%%%%%%%%%%%%%%%%%%%%%%%%%%%%%%%%%%%%%%%%%%%%%%
%%%%%%%%%%%%%%%%%%%%%%%%%%%%%%%%%%%%%%%%%%%%%%%
%%%%%%%%%%%%%%%%%%%%%%%%%%%%%%%%%%%%%%%%%%%%%%%
%%%%%%%%%%%%%%%%%%%%%%%%%%%%%%%%%%%%%%%%%%%%%%%

\subsection*{Acknowledgements} The question considered in this work was brought to my attention during discussions with Tom Bridgeland and Luca Giovenzana on a separate project. I am grateful to Tom Bridgeland for many helpful exchanges, as well as valuable comments and corrections. I also thank Dominic Joyce, Anya Nordskova, and Richard Thomas for feedback and corrections on a first version of this paper.
This work was supported by a Sloan Research Fellowship from the Alfred P. Sloan Foundation. For the purpose of open access, the author has applied a CC BY public copyright licence to any author accepted manuscript arising from this submission.

\section{Seeds, mutations, and log Calabi--Yau surfaces}
\label{sec_seeds_mutations}

In this section, we review the notions of seeds and seed mutations in a finite-rank free abelian group endowed with a skew-symmetric form of rank two. We then review, following \cite{GHKbir}, the geometric interpretation of seed mutations as cluster mutations acting on toric models of log Calabi–Yau surfaces. The main result, Theorem~\ref{thm_weyl}, identifies a Weyl group defined combinatorially in terms of seed mutations with a Weyl group defined geometrically using log Calabi–Yau surfaces. Finally, we review the classification of seeds of $q$-Painlev\'e type following \cite{KNP, lutz, Mizuno}.

\subsection{Seeds, cyclically ordered seeds, and mutations}
\label{sec_seeds}

Throughout \S \ref{sec_seeds_mutations}, we fix a free abelian group $N$ of finite rank $n$ together with a group homomorphism
\[ \psi: N \longrightarrow \bZ^2 \,\]
with finite cokernel. We denote by $K=\mathrm{Ker}(\psi)$ its kernel.
Let $\langle-,-\rangle=\det(-,-)$ be the standard determinant pairing on $\bZ^2$. By slight abuse of notation, we use the same symbol for the induced integral skew-symmetric bilinear form on $N$ obtained by pullback along $\psi$.

In this paper, a \emph{seed} in $N$ is an unordered basis $\bs=\{e_i\}_{i\in I}$ of $N$ such that $\psi(e_i) \in \bZ^2$ is a non-zero primitive vector for all $i\in I$. A \emph{cyclic ordering} of a seed $\bs=\{e_i\}_{i\in I}$ is an ordering $I \simeq \{1, \dots, n\}$ such that the vectors $\psi(e_i)$ are counterclockwise cyclically oriented in $\bR^2$: after making the standard identification $\bR^2=\bC$, there exist determinations $\theta_i$ of the arguments of $\psi(e_i) \in \bC$ such that
\begin{equation} \label{eq_cyclic}
\theta_1 \leq \theta_2 \leq \dots \leq \theta_{n-1} \leq \theta_n \leq \theta_1 + 2 \pi\,.\end{equation}
A \emph{cyclically ordered} seed $\tilde{\bs}=(e_i)_{1 \leq i \leq n}$ consists of a seed together with a choice of cyclic ordering. Given a cyclically ordered seed $\tilde{\bs}$, we write $\bs$ for the underlying seed obtained by forgetting the cyclic ordering.

Let $\bs=\{e_i\}_{i\in I}$ be a seed in $N$. For every  
 $j \in I$ and $\epsilon \in \{\pm\}$, the \emph{seed mutation} $\mu_j^\epsilon(\bs)$ of $\bs$ is the seed $\mu_j^\epsilon(\bs)=\{e_i'\}_{i \in I}$ given by 
\begin{equation} \label{eq_seed_mutation}
e_i'= \begin{cases}
-e_j \,\,\,\,\,\text{if} \,\, i=j \\ e_i+[\epsilon\langle{e_i,e_j\rangle}]_+ e_j\,\,\,\ \text{if}\,\,\, i\neq j \,,
\end{cases}\end{equation}
where $[a]_+:= \max(0,a)$ for all $a \in \bR$. Two seeds $\bs$ and $\bs'$ are \emph{mutation equivalent} if there exist sequences $j_1, \dots, j_k$ and $\epsilon_1, \dots, \epsilon_k \in \{\pm\}$ such that $\bs' = (\mu_{j_k}^{\epsilon_k} \circ \cdots \circ \mu_{j_1}^{\epsilon_1}) (\bs)$. For every seed $\bs=\{e_i\}_{i\in I}$ in $N$, $j \in I$ and $\epsilon \in \{\pm\}$, we have $\mu_j^{-\epsilon}(\mu_j^\epsilon(\bs)) = \bs$, and so mutation equivalence defines an equivalence relation on the set of seeds. 
For every seed $\bs$ in $N$, $j \in I$ and $\epsilon \in \{\pm\}$, the seed $\bs'=\{e_i'\}_{i \in I}:= \mu_j^\epsilon(\mu_j^\epsilon(\bs))$ is obtained by applying to $\bs$ the linear transformation $T_{e_j}^\epsilon$ defined by
\begin{equation} \label{eq_spherical}
 T_{e_j}^\epsilon(e_i) = e_i+ \epsilon \langle e_i, e_j\rangle e_j \,.
\end{equation}
In particular, two seeds $\bs$ and $\bs'$ related by the linear transformation 
$ T_{e_j}^\epsilon$ are mutation equivalent.

\begin{remark} \label{remark_seed_isom}
In \cite[Definition 2.3]{Mizuno}, two seeds are called mutation equivalent if they are related by a sequence of seed mutations and seed isomorphisms, where a seed isomorphism is a linear automorphism of $N$ preserving $\langle-,-\rangle$ up to a sign. For the present paper, it is essential that seed isomorphisms are not included in the definition of mutation equivalence. This discrepancy in terminology is taken into account when we refer to results of \cite{Mizuno} in \S \ref{sec_q_painleve} below. 
\end{remark}

\subsection{Seeds and log Calabi--Yau surfaces} \label{sec_log_CY}

In this section we review the relationship between seed mutations and the birational geometry of log Calabi–Yau surfaces, following \cite{GHKbir}. In particular, we recall how seeds give rise to log Calabi–Yau surfaces via toric models and how seed mutations correspond to cluster birational transformations between these models.

In this paper, a \emph{log Calabi--Yau surface} is a pair $(Y,D)$ consisting of a smooth projective surface over $\bC$ together with a non-empty singular normal crossings anticanonical divisor $D$. We refer to \cite{friedman2015geometry, GHKbir, GHK} for background on the geometry of log Calabi--Yau surfaces. A morphism $\pi:(\widetilde{Y},\widetilde{D}) \rightarrow (Y,D)$ between log Calabi--Yau surfaces is called a \emph{corner blow-up} (resp.\! an \emph{interior blow-up}) if $\pi$ is the blow-up at a singular (resp.\! smooth) point of $D$, and $\widetilde{D}$ is the preimage (resp.\! the strict transform) of $D$. 
A \emph{toric model} for $(Y,D)$ is a morphism $\pi: (Y,D) \rightarrow (\overline{Y}, \overline{D})$ where $\pi$ is a composition of interior blow-ups, and $(\overline{Y}, \overline{D})$ is a smooth projective toric surface with its toric boundary divisor. 
The toric model is \emph{framed} if we additionally fix an identification $\overline{Y} \setminus \overline{D} \simeq (\bC^\star)^2$.
By \cite[Lemma 1.3]{GHK}, after a sequence of corner blow-ups any log Calabi–Yau surface admits a toric model.

We next recall the construction of log Calabi--Yau surfaces from seeds in $N$ following \cite{GHKbir}.
Let $\bs=\{e_i\}_{i\in I}$ be a seed in $N$. Choose a smooth complete fan $\Sigma$ in 
$\bR^2$ containing the rays $\bR_{\geq 0} \psi(e_i)$ for every $i\in I$. Let $\overline{Y}$ be the smooth projective toric surface with fan $\Sigma$, and let $\overline{D} \subset \overline{Y}$ be the toric boundary divisor.
Denote by $\overline{D}_i$ the irreducible toric divisor corresponding to the ray $\bR_{\geq 0}\psi(e_i)$. The inclusion 
$\Sigma \subset \bR^2$ induces an identification $\overline{Y} \setminus \overline{D} \simeq (\bC^\star)^2$.
For each $i \in I$, choose a point $p_i \in \overline{D}_i$ such that $p_i$ is a smooth point of $\overline{D}$, and assume that $p_i \neq p_j$ if $i \neq j$. Let $Y$ be the blow-up of $\overline{Y}$ at the points $p_i$, and let $D$ be the strict transform of $\overline{D}$. 
The pair $(Y,D)$ is a log Calabi--Yau surface.
The log Calabi--Yau surface $(Y,D)$, together with the blow-up morphism $\pi: (Y,D) \rightarrow (\overline{Y}, \overline{D})$ and the identification $\overline{Y} \setminus \overline{D} \simeq (\bC^{\star})^2$, is called a \emph{framed toric model} for the seed $\bs$. Framed toric models for $\bs$ are unique up to corner blow-ups, corresponding to refinements of $\Sigma$, and deformation equivalence, corresponding to varying the points $p_i$ on the toric divisors $\overline{D}_i$.

By \cite[Proposition 4.1]{GHK}, if the collection of points $(p_i)_{i \in I}$ is chosen in the complement of the union of countably many proper Zariski closed subsets of $\prod_{i\in I} \overline{D}_i$, then the log Calabi--Yau surface is \emph{generic} in the sense that $Y \setminus D$ does not contain any smooth rational curve with self-intersection $(-2)$. The framed toric model is called generic in this case.  

We now describe the geometric interpretation of seed mutations.
Let $\pi: (Y,D) \rightarrow (\overline{Y},\overline{D})$ be a generic framed toric model for a seed $\bs$ in $N$. Let $j \in I$ and $\epsilon \in \{\pm\}$, and consider the mutated seed $\bs':= \mu_j^\epsilon(\bs)$ as in \eqref{eq_seed_mutation}. By \cite[\S 5]{GHKbir}, after possibly performing corner blow-ups on $(Y,D)$, there exists a generic framed toric model for the seed $\bs'$ of the form $\pi': (Y,D) \rightarrow (\overline{Y}',\overline{D}')$ for the same surface $(Y,D)$, 
such that the induced birational map
\[ (\bC^\star)^2=\overline{Y} \setminus \overline{D} \dashrightarrow (\bC^\star)^2 = \overline{Y}' \setminus \overline{D}'\]
is an elementary cluster transformation -- see \cite[\S 5]{GHKbir} for details. 
Concretely, after subdividing $\Sigma$ if necessary, we may assume that the ray $-\bR_{\geq 0} \psi(e_j)$ also belongs to $\Sigma$. The projection $\bR^2 \rightarrow \bR^2/\bR \psi(e_j)$ then induces a toric morphism $\overline{Y} \rightarrow \bP^1$ which is generically a $\bP^1$-fibration. 
The morphism $\pi': (Y,D) \rightarrow (\overline{Y}',\overline{D}')$ is obtained by contracting the $(-1)$-curve given by the strict transform in $Y$ of the $\bP^1$-fiber passing through $p_j$. We say that the framed 
toric models $\pi$ and $\pi'$ are related by \emph{cluster mutation}.

A fundamental result of \cite{hacking_keating} (see also \cite{blanc}) shows that any birational transformation arising from a change in the presentation of a log Calabi–Yau surface as a blow-up of a toric surface factors into cluster mutations. Since this statement plays a central role in the present paper, we record it for convenience.

\begin{theorem}\label{thm_factorization}(\cite[Proposition 3.27]{hacking_keating}, see also \cite[Theorem 1]{blanc}).
    Let $\pi: (Y,D) \rightarrow (\overline{Y},\overline{D})$ be a generic  framed toric model for a seed $\bs$ in $N$. Let $\pi': (Y,D) \rightarrow (\overline{Y}', \overline{D}')$ be a toric model for $(Y,D)$ which is the composition of interior blow-ups at distinct points. Then, there exists a seed $\bs'$ obtained from $\bs$ by a sequence $\mu$ of seed mutations, and an isomorphism 
    $\overline{Y}' \setminus \overline{D}' \simeq (\bC^\star)^2$ making $\pi'$ into a framed toric model for $\bs'$ which is related to $\pi$ by a sequence of cluster mutations induced by $\mu$.
\end{theorem}

\subsection{The intersection form}
\label{sec_intersection}

Let $\mathscr{S}$ be a mutation equivalence class of seeds in $N$.
In this section we review the construction of a canonical integral symmetric bilinear form
$(-,-)_{\mathscr{S}}$
on $K=\mathrm{Ker}(\psi)$ called the intersection form.
The construction follows \cite[\S 5]{GHKbir} and \cite[\S 3.1]{Mizuno} and uses the geometry of log Calabi--Yau surfaces.

Let $\bs=\{e_i\}_{i\in I} \in \mathscr{S}$ and $\pi: (Y,D) \rightarrow (\overline{Y},\overline{D})$ be a framed toric model for $\bs$ as in \S\ref{sec_log_CY}, with exceptional curves $E_i$ indexed by $i\in I$.
Let $NS(Y)$ denote the N\'eron-Severi group of divisors in $Y$, endowed with the intersection form $(\alpha , \beta) \mapsto \alpha \cdot \beta$.
Consider the subgroup
\begin{equation} \label{eq_lambda}
\Lambda_{(Y,D)}:=\{\alpha \in NS(Y)\,|\, \alpha \cdot D_\rho =0 \,\,\text{for every irreducible component} \,\,D_\rho\,\, \text{of}\,\, D\} \,\end{equation}
of divisor classes orthogonal to the irreducible components of $D$.

According to \cite[Theorem 5.5]{GHKbir} and \cite[Proposition 3.8]{Mizuno}, there is an isomorphism
\begin{equation} \label{eq_isom}
     \iota_\pi: K \xrightarrow{\sim} \Lambda_{(Y,D)} \,,
\end{equation}
defined as follows.
Let $a \in K$, that is  $a= \sum_{i\in I} a_i e_i \in  N$
such that $\sum_{i\in I}a_i \psi(e_i)=0$.
Then, 
\begin{equation} \label{eq_iota}
\iota_\pi(a):= \pi^\star C_a -\sum_{i\in I} a_i E_i\,,\end{equation} 
where $C_a \in NS(\overline{Y})$ is the unique curve class on the toric surface $\overline{Y}$ such that
\begin{equation} \label{eq_CD}
C_a \cdot\overline{D}_\rho=\sum_{\substack{i \in I \\ \overline{D}_i=\overline{D}_\rho}} a_i \,,\end{equation}
for every ray $\rho$ of $\Sigma$, where $\overline{D}_\rho$ denotes the corresponding toric divisor of $\overline{Y}$.  
The inverse map $\iota_\pi^{-1}$ is given by 
\begin{align} \label{eq_iota_inverse}
    \iota_\pi^{-1}: \Lambda_{(Y,D)} &\longrightarrow K \subset N \\ \nonumber
    \alpha &\longmapsto \sum_{i\in I} (\alpha \cdot E_i)e_i  \,.
\end{align}

Using the isomorphism $\iota_\pi$, we define the \emph{intersection form} $(-,-)_{\mathscr{S}}$ on $K$ as the restriction to $\Lambda_{(Y,D)}$ of the intersection form on $NS(Y)$. By \cite[Theorem 5.6]{GHKbir} (see also \cite[Lemmas 3.10-3.15]{Mizuno}), this bilinear form $(-,-)_{\mathscr{S}}$ is independent of the choices of the framed toric model $\pi$ and of the seed $\bs \in \mathscr{S}$. Hence, it depends only on the mutation equivalence class $\mathscr{S}$. 

We now give a purely combinatorial description of this intersection form.
Let $I \simeq \{1,\dots, n\}$ be a cyclic ordering of the seed $\bs=\{e_i\}_{i\in I}$, and let $\tilde{\bs}=(e_i)_{1 \leq i \leq n}$ as in 
\eqref{eq_cyclic}.
Define the bilinear form 
\[ \chi_{\tilde{\bs}}: N \times N \longrightarrow \bZ\]
by
\begin{equation} \label{eq_bilinear}
\chi_{\tilde{\bs}}(e_i, e_j):=
\begin{cases}
1 \,\,\,\text{if} \,\,\, i=j \\
\langle e_i,e_j \rangle\,\,\text{if}\,\,\, i>j\\
0 \,\,\,\, \text{if}\,\,\, i <j \,.
\end{cases}
\end{equation}
Since $\chi_{\tilde{\bs}}(e_i,e_j)-\chi_{\tilde{\bs}}(e_j,e_i)= \langle e_i,e_j\rangle$, 
and $K$ is the kernel of $\langle-,-\rangle$, the restriction $\chi_{\tilde{\bs}}(-,-)|_K$ is a symmetric bilinear form on $K$. In Lemma \ref{lem_brackets} below, we show that the intersection form $(-,-)_{\mathscr{S}}$
coincides with $-\chi_{\tilde{\bs}}(-,-)|_K$.
We first prove an elementary result in the geometry of toric surfaces which will be used in the proof of  Lemma \ref{lem_brackets}.

\begin{lemma} \label{lem_C2}
Let $\overline{Y}$ be a smooth projective toric surface. Let $\rho_1, \dots, \rho_m$ be the rays of the fan of $\overline{Y}$ cyclically ordered in $\bR^2$, with primitive generators  $u_{\rho_1}, \dots, u_{\rho_m} \in \bZ^2$, and corresponding toric divisors $\overline{D}_{\rho_1}, \dots, \overline{D}_{\rho_m}$. Then, for every class $C$ in the N\'eron-Severi group $NS(\overline{Y})$ of $\overline{Y}$, we have 
\begin{equation} \label{eq_C2}
C  \cdot C = \sum_{1 \leq i <j \leq m}\alpha_{\rho_i} \alpha_{\rho_j} \langle u_{\rho_i}, u_{\rho_j} \rangle \,,\end{equation}
where $\alpha_{\rho_i}:= C \cdot \overline{D}_{\rho_i}$ for all $1 \leq i \leq m$.
\end{lemma}

\begin{proof}
Since $\overline{Y}$ is projective, the cone of ample divisors is open in $NS(\overline{Y})$. Therefore, it suffices to prove \eqref{eq_C2} for ample classes $C$. 
By \cite[\S 5.3, Corollary, p.111]{Fultontoric},
the self-intersection $C\cdot C$ of an ample class $C$ is equal to the area of a lattice polygon whose sides have outer normal vectors $u_{\rho_i}$
and integral lengths $\alpha_{\rho_i}:= C \cdot \overline{D}_{\rho_i}$.
Since the vectors $u_{\rho_i}$ are cyclically ordered in $\bR^2$, this area is given by $\sum_{1 \leq i<j\leq n} \alpha_{\rho_i} a_{\rho_j} \langle u_{\rho_i}, u_{\rho_j}\rangle$, which
establishes \eqref{eq_C2}.
\end{proof}

\begin{lemma} \label{lem_brackets}
Let $\tilde{\bs}$ be a cyclically ordered seed with underlying seed $\bs\in \mathscr{S}$. 
Then, for every $\alpha, \beta \in K$, we have 
\[ (\alpha, \beta)_{\mathscr{S}} = - \chi_{\tilde{\bs}}(\alpha, \beta)\,.\]
\end{lemma}

\begin{proof}
Let $\tilde{\bs}=(e_i)_{1 \leq i\leq n}$, and set $v_i:= \psi(e_i)$ for all $1 \leq i \leq n$. 
Since $(-,-)_{\mathscr{S}}$ and $\chi_{\tilde{\bs}}(-,-)|_K$ are symmetric bilinear forms on $K$, it suffices to show that $(a,a)_{\mathscr{S}} = - \chi_{\tilde{\bs}} (a,a)$ for every 
$a =\sum_{i=1}^n a_i e_i \in K$, that is, such that $\sum_{i=1}^n a_iv_i=0$. 
By \eqref{eq_bilinear}, we have 
\begin{equation} \label{eq_chi}
\chi_{\tilde{\bs}}(a,a)= \sum_{i=1}^n a_i^2 + \sum_{1 \leq j <i \leq n } a_i a_j \langle v_i,v_j\rangle
= \sum_{i=1}^n a_i^2 - \sum_{1 \leq i<j \leq n} a_i a_j \langle v_i,v_j\rangle \,.\end{equation}
Let $\pi: (Y,D) \rightarrow (\overline{Y},\overline{D})$ be a framed toric model for $\bs$ as in \S\ref{sec_log_CY}, with exceptional curves $E_i$ indexed by $1 \leq i \leq n$.
Since $E_i^2=-1$ for all 
$1 \leq i \leq n$, it follows from \eqref{eq_iota} that
\[(a,a)_{\mathscr{S}} = -\sum_{i=1}^n a_i^2 + C_a \cdot C_a \,, \]
where $C_a \cdot C_a$ is the self-intersection of the curve class $C_a$ on the toric surface $\overline{Y}$.
Since the vectors $v_i$ are cyclically ordered in $\bR^2$
by \eqref{eq_cyclic}, it follows from Lemma \ref{lem_C2} applied to $C_a$ and \eqref{eq_CD} that 
\[ C_a \cdot C_a = \sum_{1\leq i<j\leq n} a_i a_j \langle v_i, v_j\rangle \,,\]
which concludes the proof by comparison with \eqref{eq_chi}.
\end{proof}

\begin{remark}
    A cyclically ordered seed $\tilde{\bs}=(e_i)_{1 \leq i \leq n}$ in $N$ determines a quiver with linear ordering, with $n$ vertices indexed by $1 \leq i \leq n$, and $\langle e_i,e_j \rangle$ arrows from the vertex $i$ to the vertex $j$ for all $1 \leq i<j \leq n$. The matrix of the bilinear form $\chi_{\tilde{\bs}}$ is the ``unipotent companion matrix" of this quiver with linear ordering, in the sense of  \cite{fomin2024cyclically}.
 \end{remark}

\subsection{Weyl groups}
\label{sec_weyl}
The main result of this section, Theorem \ref{thm_weyl}, establishes the equality between two Weyl groups: one defined combinatorially in terms of seed mutations, and the other defined geometrically via the associated log Calabi--Yau surface. Throughout this section, we fix a mutation equivalence class $\mathscr{S}$ of seeds in $N$.
For every seed $\bs =\{e_i\}_{i\in I}\in \mathscr{S}$, and distinct elements $e_j$ and $e_k$ of $\bs$ such that $\psi(e_j)=\psi(e_k)$,  
consider the unique linear map 
\begin{equation} \label{eq_t_jk}
t_{jk}^{\bs}: N \rightarrow N\end{equation} 
such that $t_{jk}^{\bs}(e_i)=e_i$ for all $i \notin \{j,k\}$, $t_{jk}^{\bs}(e_j)=e_k$, and $t_{jk}^{\bs}(e_k)=e_j$. In other words, $t_{jk}^{\bs}$ simply permutes the elements $e_j$ and $e_k$ of the basis $\bs=\{e_i\}_{i\in I}$ of $N$.

\begin{definition} \label{def_weyl_S}
The Weyl group of $\mathscr{S}$ is the subgroup $W_{\mathscr{S}}$ of automorphisms of $N$ generated by the linear transformations $t_{jk}^{\bs}$ for every seed $\bs =\{e_i\}_{i\in I}\in \mathscr{S}$, and distinct elements $e_j$ and $e_k$ of $\bs$ such that $\psi(e_j)=\psi(e_k)$.
\end{definition}

We give below 
in Lemma \ref{lem_W_S} an alternative description of $W_{\mathscr{S}}$ as a group of isometries of $K=\mathrm{Ker}(\psi)$
endowed with 
the intersection form 
$(-,-)_{\mathscr{S}}$ introduced in \S\ref{sec_intersection}.

\begin{definition} \label{def_root}
    A \emph{root} of $\mathscr{S}$ is an element $\alpha \in N$ for which there exists a seed $\bs=\{e_i\}_{i\in I} \in \mathscr{S}$, and two distinct elements $e_j$ and $e_k$ of $\bs$ such that 
    $\psi(e_j)=\psi(e_k)$ and $\alpha=e_j-e_k$. We denote by $\Phi_{\mathscr{S}}$ the set of roots of $\mathscr{S}$.
\end{definition}

Every root $\alpha = e_j-e_k$ of $\mathscr{S}$ satisfies $\psi(\alpha)=\psi(e_j)-\psi(e_k)=0$, that is, $\alpha \in K=\mathrm{Ker}(\psi)$, and so $\Phi_{\mathscr{S}} \subset K$.

\begin{lemma} \label{lem_root}
Every root $\alpha \in \Phi_{\mathscr{S}}$ satisfies 
$(\alpha, \alpha)_{\mathscr{S}}=-2$.
\end{lemma}

\begin{proof}
By Definition \ref{def_root}, there exists a seed $\bs=\{e_i\}_{i\in I}$ and
distinct elements $e_j$ and $e_k$ of $\bs$ such that $\psi(e_j)=\psi(e_k)$ and $\alpha=e_j-e_k$.
To calculate $(\alpha, \alpha)_{\mathscr{S}}$, consider a framed toric model 
$\pi: (Y,D) \rightarrow (\overline{Y},\overline{D})$ for $\bs$ as in \S\ref{sec_log_CY}, with exceptional curves $E_i$ indexed by $i\in I$, and the isomorphism $\iota_\pi: K \xrightarrow{\sim} \Lambda_{(Y,D)}$ given in \eqref{eq_isom}. Then, we have $\iota_\pi (\alpha)=\iota_\pi(e_j-e_k)=E_k-E_j$. 
Since $E_j^2=E_k^2=-1$ and $E_j \cdot E_k=0$, we have $(E_j-E_k)^2=-2$, and so the result follows.
\end{proof}

For every $\alpha \in K$ with $(\alpha, \alpha)_{\mathscr{S}}=-2$, consider the isometry of $(K, (-,-)_{\mathscr{S}})$ given by the reflection with respect to $\alpha$:
\begin{align*}
    s_\alpha : K &\longrightarrow K \\
    \beta &\longmapsto \beta + (\beta, \alpha)_{\mathscr{S}} \,\alpha \,.
\end{align*}

\begin{lemma} \label{lem_tijk}
Consider a seed $\bs =\{e_i\}_{i\in I}\in \mathscr{S}$, and distinct elements $e_j$ and $e_k$ of $\bs$ such that $\psi(e_j)=\psi(e_k)$.
Then, for every cyclic ordering $\tilde{\bs}$ of $\bs$, denoting $\alpha=e_j-e_k$, we have 
\[ t_{jk}^\bs(\beta) = \beta -\chi_{\tilde{\bs}}(\beta, \alpha) \alpha\]
for every $\beta \in N$. In particular, the automorphisms $t_{jk}^\bs$ preserve $K \subset N$ and $t_{jk}^\bs|_K = s_\alpha$.
\end{lemma}

\begin{proof}
The formula for $t_{jk}^\bs$ follows by a direct computation using the definition of $\chi_{\tilde{\bs}}$ given in \eqref{eq_bilinear}.
The statement regarding the restriction to $K$ then follows since $\chi_{\tilde{\bs}}(\beta, \alpha)=-(\beta, \alpha)_{\mathscr{S}}$ for every $\alpha, \beta \in K$ by Lemma \ref{lem_brackets}.
\end{proof}

By Lemma \ref{lem_tijk}, $W_{\mathscr{S}}$ preserves $K \subset N$, and so we obtain by restriction a group homomorphism 
\[ \rho: W_{\mathscr{S}} \rightarrow \mathrm{Aut}(K)\]
sending the generators $t_{jk}^\bs$
to the reflections $s_\alpha$ with $\alpha=e_j-e_k$, where $\bs=(e_i)_{i\in I}$.

\begin{lemma} \label{lem_W_S}
The restriction map
$\rho: W_{\mathscr{S}} \rightarrow \mathrm{Aut}(K)$
is injective and defines an isomorphism between $W_{\mathscr{S}}$
and the subgroup of isometries of $(K, (-,-)_{\mathscr{S}})$ generated by the reflections $s_\alpha$ with respect to the roots $\alpha \in \Phi_{\mathscr{S}}$ . 
\end{lemma}

\begin{proof}
By Lemma \ref{lem_tijk}, 
it suffices to show that $\rho$ is injective. Let $g \in W_{\mathscr{S}}$ be such that $\rho(g)=\mathrm{id}$. Write \[ g=t_{j_\ell k_\ell}^{\bs_\ell} \circ \dots \circ t_{j_1 k_1}^{\bs_1}\] and denote by $\alpha_1, \dots, \alpha_\ell \in \Phi_{\mathscr{S}}$ the corresponding roots. Fix a seed $\bs =\{e_i\}_{i\in I} \in \mathscr{S}$, and 
let $\pi: (Y,D) \rightarrow (\overline{Y},\overline{D})$ be a framed toric model for $\bs$ as in \S\ref{sec_log_CY}, with exceptional curves $E_i$
indexed by $i\in I$. Consider the identification 
$\iota_\pi: K \xrightarrow{\sim} \Lambda_{(Y,D)} \subset NS(Y)$ given by \eqref{eq_iota}, and
the linear map 
\begin{align*}
    \varphi: NS(Y) &\longrightarrow N \\
    \beta &\longmapsto 
    \sum_{i\in I}(\beta \cdot E_i)e_i \,.
\end{align*}
Since $\varphi(-E_i)=e_i$ for all $i\in I$, the map $\varphi$ is surjective.
By Lemma \ref{lem_root}, the roots 
$\alpha_m$ satisfy $(\alpha_m, \alpha_m)_{\mathscr{S}}=-2$ for all 
$1 \leq m \leq \ell$, and so the classes $\iota_\pi(\alpha_m)$
in $\Lambda_{(Y,D)} \subset NS(Y)$
have self-intersection $-2$. The corresponding reflection $r_{\iota_\pi(\alpha_m)}$ acts on $NS(Y)$ and fixes the irreducible components of $D$. Moreover, we have 
$\varphi \circ \iota_\pi = \mathrm{id}$ by \eqref{eq_iota_inverse}, and so it follows 
from Lemma \ref{lem_tijk} that $\varphi \circ r_{\iota_\pi(\alpha_m)} = t_{j_m k_m}^{\bs_m} \circ \varphi$. 
Since $\rho(g)=\mathrm{id}$, the composition  $\widetilde{g}:= r_{\iota_\pi(\alpha_\ell)} \circ \dots \circ r_{\iota_\pi(\alpha_1)}$ acts trivially on $\Lambda_{(Y,D)}$. In addition, the reflections  $r_{\iota_\pi(\alpha_m)}$ act trivially on the classes of irreducible components of $D$, therefore the same holds for $\widetilde{g}$. Consequently,
$\widetilde{g}$ acts trivially on the whole of $NS(Y)$. 
Using the compatibility 
$\varphi \circ \widetilde{g} = g \circ \varphi$ together with the surjectivity of $\varphi$, we deduce that $g=\mathrm{id}$, and this completes the proof of the injectivity of $\rho$.  
\end{proof}

Finally, we describe the Weyl group $W_{\mathscr{S}}$ in terms of the geometry of log Calabi--Yau surfaces. Fix a seed $\bs \in \mathscr{S}$ and $\pi: (Y,D) \rightarrow (\overline{Y},\overline{D})$ a framed toric model for $\bs$ as in \S\ref{sec_log_CY}, with exceptional curves $E_i$
indexed by $i\in I$. 
Following \cite[Definition 1.6]{GHK}, a \emph{root} of $(Y,D)$ is an element $\alpha \in \Lambda_{(Y,D)}$ with $\alpha \cdot \alpha=-2$, which is realized by a smooth rational curve on a deformation equivalent pair $(Y',D')$. We denote by $\Phi$ the set of roots of $(Y,D)$.
The \emph{Weyl group} $W$ of $(Y,D)$ is the group of isometries of $NS(Y)$ generated by the reflections with respect to the roots $\alpha \in \Phi$. Elements of $W$ fix the classes of the irreducible components of $D$, and so are uniquely determined 
by their action on the orthogonal subgroup $\Lambda_{(Y,D)} \subset NS(Y)$. Hence, $W$ can also be viewed as a subgroup of the group of isometries of $\Lambda_{(Y,D)}$.
If $(\widetilde{Y},\widetilde{D}) \rightarrow (Y,D)$ is a corner blow-up, then there is a canonical identification $\Lambda_{(Y',D')} \simeq \Lambda_{(Y,D)}$, and so a corresponding identification of the sets of roots and of the Weyl groups of $(Y',D')$ and $(Y,D)$. Therefore, we consider in the following $(Y,D)$ up to corner blow-ups.

\begin{lemma} \label{lem_root_1}
Let $E$ and $F$ be two disjoint smooth rational curves in $Y$ such that $E^2=F^2=-1$, $E$ and $F$ are not contained in $D$ and intersect the same irreducible component of $D$. Then $E-F$ is a root of $(Y,D)$.
\end{lemma}

\begin{proof}
There exists a morphism $(Y,D) \rightarrow (Y',D')$ that contracts $E$ and $F$ to two distinct points $p$ and $q$ lying on the same irreducible component $D'_i$ of $D'$. Bringing the points $p$ and $q$ together on $D'$ and blowing them up successively defines a log Calabi--Yau surface deformation equivalent to $(Y,D)$ containing a smooth rational curve of class $E-F$, so $E-F \in \Phi$.
\end{proof}

\begin{theorem} \label{thm_weyl}
Let $\mathscr{S}$ be a mutation equivalence class of seeds in $N$, and $\pi: (Y,D) \rightarrow (\overline{Y},\overline{D})$ be a framed toric model for a seed
$\bs \in \mathscr{S}$. 
Under the identification
$\iota_\pi: K \xrightarrow{\sim} \Lambda_{(Y,D)}$, the roots and the Weyl group of $\mathscr{S}$ coincide with the roots and the Weyl group of the log Calabi--Yau surface $(Y,D)$ respectively:
    \[  \Phi_{\mathscr{S}} = \Phi  \,\,\, \text{and} \,\,\, W_{\mathscr{S}} = W \,.\]
\end{theorem}

\begin{proof}
It suffices to prove that $\Phi_{\mathscr{S}}=\Phi$. We first show the inclusion $\Phi_{\mathscr{S}} \subset \Phi$. Let $\alpha \in \Phi_{\mathscr{S}}$. By Definition \ref{def_root}, there exists a seed $\bs'=\{e_i'\}_{i\in I} \in \mathscr{S}$ and distinct elements $e_j'$ and $e_k'$ of 
$\bs'$ such that $\psi(e_j')=\psi(e_k')$ and $\alpha=e_j'-e_k'$. Since both seeds $\bs$ and $\bs'$ are in $\mathscr{S}$, there exists a sequence of seed mutations that $\bs'=\mu(\bs)$.
As reviewed in \S \ref{sec_log_CY}, it follows that, up to corner blow-ups,
there exists a framed toric model $\pi': (Y, D) \rightarrow (\overline{Y}', \overline{D}')$ for $\bs'$ related to $\bs$ by cluster mutations induced by 
$\mu$. 
By construction, there exist exceptional curves $E_j'$ and $E_k'$ for $\pi'$, intersecting the same irreducible component of $\overline{D}'$, such that $\iota_{\pi'}(\alpha)=E'_j-E'_k$, and so
$\alpha \in \Phi$ by Lemma \ref{lem_root_1}.

It remains to show the reverse inclusion $\Phi \subset \Phi_{\mathscr{S}}$. Let $\alpha \in \Phi$. If $I=\emptyset$, there is nothing to show, so assume that $I\neq \emptyset$.
Then $Y$ contains $(-1)$-curves and so $Y$ is not isomorphic to $\bP^1 \times \bP^1$ or $\mathbb{F}_2$. Therefore, by
\cite[Theorem 6.12]{friedman2015geometry},
up to replacing $(Y,D)$ by a deformation-equivalent log Calabi--Yau surface, there exist two disjoint smooth rational curves $E$ and $F$ in $Y$ such that $\alpha=E-F$, 
$E^2=F^2=-1$, $E$ and $F$ are not contained in $D$ and intersect the same irreducible component of $D$. Let $(Y',D')$ be the log Calabi--Yau surface obtained from $(Y,D)$ by contracting $E$ and $F$. Up to replacing $(Y,D)$ and $(Y',D')$ by corner blow-ups, $(Y',D')$ admits a toric model $\pi': (Y',D') \rightarrow (\overline{Y}',\overline{D}')$ by  \cite[Lemma 1.3]{GHK}. Then, by Theorem \ref{thm_factorization} 
(\cite[Proposition 3.27]{hacking_keating}, see also \cite[Theorem 1]{blanc}), there exists an identification 
$(\bC^\star)^2 \simeq \overline{Y}' \setminus \overline{D}'$ such that
the composition $(Y,D) \rightarrow (\overline{Y}',\overline{D}')$ is a
framed toric model for a seed $\bs'$ related to $\bs$ by a sequence of mutations. Since $E$ and $F$ are exceptional curves of $(Y,D) \rightarrow (\overline{Y}',\overline{D}')$ intersecting the same irreducible component of $D$, it follows that $\iota_\pi^{-1}(\alpha) \in \Phi_{\mathscr{S}}$.
\end{proof}

\subsection{Seeds of $q$-Painlev\'e type and T-polygons}
\label{sec_q_painleve}

In this section, we first review the concept of seeds of 
$q$-Painlevé type, together with the associated T-polygons. We then discuss the geometry of the associated log Calabi--Yau surfaces. The main result is Theorem \ref{thm_q_P} classifying images in $\bZ^2$ of seeds of $q$-Painlev\'e type up to mutation and $SL(2, \bZ)$.

\begin{definition} \label{def_q_painleve}
A mutation equivalence class $\mathscr{S}$ of seeds in $N$ is of \emph{$q$-Painlev\'e type} if the intersection form 
$(-,-)_{\mathscr{S}}$ on $K$ is negative semi-definite but not negative definite.
A seed in $N$ is of $q$-Painlev\'e type if its mutation equivalence class is of $q$-Painlev\'e type.
\end{definition}

By \cite[Proposition 3.22 (3)]{Mizuno}, if $\bs$ is a seed of $q$-Painlev\'e type, then there exists a unique primitive element
\begin{equation} \label{eq_delta_S}
    \delta_{\mathscr{S}} = \sum_{i\in I} c_{\bs,i} e_i \in K
\end{equation}
such that $c_{\bs,i}\in \bZ_{>0}$ and $(\delta_\mathscr{S}, \delta_\mathscr{S})_\mathscr{S}=0$. Moreover, $\delta_{\mathscr{S}}$ depends only on the mutation equivalence class $\mathscr{S}$ of $\bs$.

Let $\bs$ be a seed of $q$-Painlev\'e type. Following \cite[Eq. (3.7)]{Mizuno}, we define a convex polygon $P_{\bs}$ in $\bR^2$ as the intersection of the half-spaces $\{
v\in \bR^2\,|\, \langle v, \psi(e_i)\rangle \geq -c_{\bs, i}\}$ for all $i\in I$. By \cite[Proposition 3.30]{Mizuno}, $P_\bs$ is a Fano polygon without remainder, or equivalently a \emph{T-polygon} in the terminology of \cite{lutz}: 
\begin{itemize}
    \item[(i)] $P_{\bs}$ is a bounded convex polygon in $\bR^2$ containing the origin $0$ in its interior and each vertex of $P_\bs$ is a primitive element of $\bZ^2$,
    \item[(ii)] for every edge $e$ of $P_\bs$, the lattice length of $e$ is divisible by the lattice distance between $e$ and the origin.
\end{itemize}
Fix a cyclic ordering $\tilde{\bs}=(e_i)_{1 \leq i\leq n}$ of $\bs$. Let $v_1, \dots, v_m$ be the cyclically ordered collection of vectors in $\bZ^2$ of the form $\psi(e_i)$, without repetition. Then it follows from the proof of \cite[Proposition 3.30]{Mizuno} that $v_1, \dots, v_m$ are exactly the primitive directions of the edges of $P_\bs$. Moreover, given an edge $e$ with primitive direction $v_j$, the lattice distance between $e$ and the origin is equal to $c_{\bs,i}$ for any $1\leq i\leq n$ such that $v_j=\psi(e_i)$. Finally, the lattice length of $e$ is equal to this lattice distance times the number of $1 \leq i\leq n$ such that $\psi(e_i)=v_j$.

\begin{lemma} \label{lem_coeff}
Let $\bs=\{e_i\}_{i \in I}$ and $\bs'=\{e_i'\}_{i\in I}$
be two seeds of $q$-Painlev\'e type in $N$.
Let $\tilde{\bs}=(e_i)_{1 \leq i \leq n}$ and $\tilde{\bs}'=(e_i')_{1\leq i\leq n}$
be cyclic orderings of $\bs$ and $\bs'$ respectively. 
Assume that there exists $f \in SL(2,\bZ)$ such that $\psi(e_i)=f(\psi(e_i'))$ for all $1 \leq i\leq n$. Then, we have  $c_{\bs,i} = c_{\bs',i}$ for all $1 \leq i \leq n$.
\end{lemma}

\begin{proof}
The assumption implies that $P_\bs=f(P_{\bs'})$ with $f \in SL(2, \bZ)$. The result follows since the lattice distances between the origin and the edges of a lattice polygon are invariant under the action of $SL(2,\bZ)$.
\end{proof}

The following results collect facts about toric models of seeds of $q$-Painlev\'e type, following \cite{friedman2015geometry, lutz, mandel, Mizuno, Sak}. A log Calabi--Yau surface $(Y,D)$ is called \emph{minimal} if no component of $D$ is an exceptional curve. Every log Calabi--Yau surface admits a unique minimal model $\nu: (Y,D) \rightarrow (Y^\nu,D^\nu)$, where $\nu$ is a composition of corner blow-ups and $(Y^\nu, D^\nu)$ is a minimal log Calabi--Yau surface obtained by successively contracting exceptional curves in $D$.

\begin{lemma} \label{lem_log_cy-painleve}
Let $(Y,D)$ be a minimal log Calabi--Yau surface. Then, the following are equivalent:
\begin{itemize}
    \item[(i)] $(Y,D)$ is the minimal model of a toric model of a seed of $q$-Painlev\'e type.
    \item[(ii)] The intersection form on $\Lambda_{(Y,D)}$ is negative semi-definite but not negative definite.
    \item[(iii)] The divisor $D$ is either irreducible with $D^2=0$, or a cycle of smooth rational curves $D_i$ with $D_i^2=-2$ for all $i$.
    \item[(iv)] $(Y,D)$ is deformation equivalent to a log Calabi--Yau surface $(Y',D')$ which admits an elliptic fibration having $D'$ as a fiber.
\end{itemize}
\end{lemma}

\begin{proof}
We have (i) $\implies$ (ii) by definition of seeds of $q$-Painlev\'e type, and (ii) $\implies (i)$ since every log Calabi--Yau surface admits a toric model, and $\Lambda_{(Y,D)}$ is invariant under corner blow-ups. The equivalences
(ii) $\Leftrightarrow$ (iii) $\Leftrightarrow$ (iv) are (6) $\Leftrightarrow$ (5) $\Leftrightarrow$ (7)
in \cite[Theorem 4.2]{mandel}.
\end{proof}

\begin{prop} \label{prop_classification_log_cy}
There are exactly ten deformation equivalence classes of minimal log Calabi--Yau surfaces arising as minimal models of toric models of seeds of $q$-Painlev\'e type. Moreover, each deformation equivalence class is uniquely determined by the isometry class of the intersection form $\Lambda_{(Y,D)}$. Using the notation of \cite[Table 4]{Sak}, the ten possibilities for $\Lambda_{(Y,D)}$ are the following affine root systems:
\[ A_0^{(1)}\,\,\,A_1^{(1)}\,\,\, A_{1,|\alpha|^2=8}^{(1)}\, \,\,  (A_1 + A_{1,|\alpha|^2=14})^{(1)} \,\,\, (A_2+A_1)^{(1)}\,\,\, A_4^{(1)}\, \,\, D_5^{(1)} \,\,\, E_6^{(1)} \,\,\, E_7^{(1)} \,\,\, E_8^{(1)} \,.\]
\end{prop}

\begin{proof}
By (i) $\Leftrightarrow$ (iii) in Lemma \ref{lem_log_cy-painleve}, it suffices to classify up to deformation log Calabi--Yau surfaces $(Y,D)$ such that $D$ is either irreducible with $D^2=0$, or a cycle of smooth rational curves $D_i$ with $D_i^2=-2$ for all $i$. This is done in \cite[Proposition 9.15-9.16]{friedman2015geometry} -- see also \cite[\S 4]{Sak}.
\end{proof}

\begin{remark}
The ten deformation classes of log Calabi–Yau surfaces in 
Proposition \ref{prop_classification_log_cy} 
are precisely the ten deformation classes of log Calabi--Yau surfaces that appear in Sakai’s classification of the 
$q$-Painlevé equations \cite{BGM}, \cite[\S 3.2]{Mizuno}, \cite{Sak}. 
This connection motivates the terminology 
``$q$-Painlevé type" introduced in \cite{Mizuno}. These log Calabi–Yau surfaces are also deformation equivalent to the rational elliptic fibrations which are  mirror to the ten deformation classes of del Pezzo surfaces \cite{mirror_orbifold, AKO, corti2023cluster, GGLP, gugiatti2025mirrors, lutz}. 
In each case, the lattice $K=\Lambda_{(Y,D)}$ admits a decomposition $K= \bZ \delta_{\mathscr{S}} \oplus Q_{\mathscr{S}}$, where $Q_{\mathscr{S}}$ is a negative definite lattice. 
A direct inspection shows that the lattices $Q_\mathscr{S}$ coincide precisely with the ten lattices $K_Z^{\perp}$ given by the orthogonal complements of the canonical divisor in the N\'eron-Severi groups $NS(Z)$ of the ten deformation classes of del Pezzo surfaces $Z$ \cite[Ch 8]{dolgachev} -- see \cite{mizoguchi} for an early version of this observation. 
In Remark \ref{rem_del_pezzo_painleve}, a geometric description of the identification
$Q_{\mathscr{S}} \simeq K_Z^\perp$ is obtained using geometric helices on del Pezzo surfaces. 
\end{remark}

The following result classifies the images $\psi(\bs)$ in $\bZ^2$ of seeds of $q$-Painlev\'e type  up to mutations and $SL(2,\bZ)$. It plays a key role in the proof of the main result of this paper, Theorem \ref{thm_main}, on the classification of geometric helices on del Pezzo surfaces. 

\begin{theorem} \label{thm_q_P}
Let $\bs=\{e_i\}_{i\in I}$ and $\bs'=\{e_j'\}_{j \in J}$ be two seeds of $q$-Painlev\'e type in $N$, with corresponding mutation equivalence classes $\mathscr{S}$ and $\mathscr{S}'$ respectively. Assume that $(K,(-,-)_{\mathscr{S}})$ and 
$(K,(-,-)_{\mathscr{S}'})$ are isometric. 
Then, there exists a sequence of mutations $\mu$ and a linear automorphism 
$g \in SL(2,\bZ)$ such that 
\[ \psi(\bs') = g\circ \psi(\mu(\bs)) \,.\]
\end{theorem}

\begin{proof}
By \cite[Proposition 3.30]{Mizuno}, if $\bs$ is a seed of $q$-Painlev\'e type,
the data of the seed $\psi(\bs)$ in $\bZ^2$ is equivalent that of the associated T-polygon $P_\bs$. 
Moreover, by \cite[Propositions 3.15-3.17]{KNP}, mutations T-polygons applied to $P_\bs$, as defined in \cite[\S 3]{mutations},\cite[\S 2]{KNP}, correspond to mutations of $\psi(\bs)$ in $\bZ^2$. These mutations canonically lift to mutations of $\bs$ in $N$. 
Thus, it suffices to prove that if $\bs$ and $\bs'$ are two seeds of $q$-Painlev\'e type with isometric $(K,(-,-)_{\mathscr{S}})$ and 
$(K,(-,-)_{\mathscr{S}'})$, then the polygons $P_\bs$ and $P_{\bs'}$ are related by mutations and the action of
$SL(2, \bZ)$.
The corresponding statement with $GL(2,\bZ)$ in place of $SL(2,\bZ)$ is proved in \cite[Corollary 3.33]{Mizuno} using the classification of T-polygons up to mutations and $GL(2,\bZ)$ proved in \cite[Theorem 5.6]{KNP}. However, Theorem \ref{thm_q_P} is stated for $SL(2,\bZ)$, not for $GL(2,\bZ)$: while two T-polygons differing by $GL(2,\bZ)$ might not be related by $SL(2,\bZ)$, we claim that they are always related by $SL(2,\bZ)$ and a sequence of mutations.
To lift the result from $GL(2,\bZ)$ to $SL(2, \bZ)$, we explain how to modify the geometric proof of the classification of T-polygons given in \cite[Theorem 4.16]{lutz}. 

Let $\pi: (Y,D) \rightarrow (\overline{Y}, \overline{D})$ and $\pi': (Y',D') \rightarrow (\overline{Y}', \overline{D}')$ be framed toric models of $\bs$ and $\bs'$ respectively, as in \S \ref{sec_log_CY}. Let $\nu: (Y,D) \rightarrow (Y^\nu, D^\nu)$ and $\nu': (Y',D') \rightarrow ((Y')^\nu, (D')^\nu)$ be the corresponding minimal models.
Since $(K,(-,-)_{\mathscr{S}})$ and 
$(K,(-,-)_{\mathscr{S}'})$ are isometric, Proposition \ref{prop_classification_log_cy} implies that the log Calabi--Yau surfaces $(Y^\nu, D^\nu)$ and $((Y')^\nu, (D')^\nu)$ are deformation equivalent. 
Hence, after replacing $(Y,D)$ and $(Y',D')$ by deformation equivalent log Calabi--Yau surfaces, we may assume that there exists an isomorphism 
\[ f: (Y^\nu, D^\nu) \xrightarrow{\sim}((Y')^\nu, (D')^\nu)\,.\] Let $\varphi_f: (\bC^\star)^2 = \overline{Y} \setminus D \dashrightarrow (\bC^\star)^2 = \overline{Y}' \setminus \overline{D}'$ be the induced birational transformation. Let $\Omega = \frac{dx}{x}\wedge \frac{dy}{y}$ on $(\bC^\star)^2$. As in the proof of \cite[Theorem 4.16]{lutz}, we have $\varphi_f^\star \Omega = \pm \Omega$. If $\varphi_f^\star \Omega=\Omega$, then the end of the proof of \cite[Theorem 4.16]{lutz}, relying on the factorization result reviewed in Theorem \ref{thm_factorization}, implies that $P_\bs$ and $P_{\bs'}$ are related by mutations and the action of $SL(2, \bZ)$, since $SL(2,\bZ)$ is the group of group automorphisms of $(\bC^\star)^2$ preserving $\Omega$. Moreover, the proof of \cite[Theorem 4.16]{lutz} reduces the case $\varphi_f^\star \Omega=-\Omega$ to the case $\varphi_f^\star \Omega=\Omega$ by composing $\varphi_f$ with an element of $GL(2,\bZ)$ of determinant $-1$. To avoid using $GL(2,\bZ)$ and prove Theorem \ref{thm_q_P}, it suffices to show that, up to changing $f$, we can always assume $\varphi_f^\star \Omega=\Omega$. 

By Lemma \ref{lem_log_cy-painleve}, one can assume up to deformation that $(Y^\nu, D^\nu)$ admits an elliptic fibration $\rho: Y^\nu \rightarrow \bP^1$ having $D^\nu$ as a fiber. Moreover, $\rho$ always admits a section by \cite[Corollary 4.7]{lutz}. Since $\rho$ is relatively minimal, there exists an involution $\iota$ of $Y^\nu$ acting on smooth fibers of $\rho$ as the inverse $z \mapsto -z$ with respect to the group law defined by the section. In particular, for every fiber $F$ of $\rho$, the involution $\iota$ acts on the one-dimensional space $H^0(F,\omega_F)$ of sections of the dualizing sheaf $\omega_F$ by multiplication by $-1$. Since $Y^\nu$ is rational, we have $H^0(Y^\nu, K_{Y^\nu})=0$, and so the residue map $H^0(Y^\nu, \omega_{Y^\nu}(D^\nu)) \rightarrow H^0(D^\nu, \omega_{D^\nu})$ is an isomorphism. Therefore, we have $\iota^\star \Omega = - \Omega$. Hence, if $\varphi_f^\star \Omega=-\Omega$, we may always replace $f$ by $f'=f \circ \iota$ so that $\varphi_{f'}^\star \Omega=\Omega$, and this concludes the proof.
\end{proof}

\begin{remark}
The proof of the classification of T-polygons up to mutations and $GL(2,\bZ)$ given in \cite{KNP} is combinatorial. One might prove Theorem \ref{thm_q_P} in a similar explicit combinatorial way by finding for each of the ten mutation classes a T-polygon $P$, an element $g \in GL(2,\bZ)$ with $\det(g)=-1$, and an explicit sequence of mutations relating $P$ and $g \cdot P$ up to $SL(2,\bZ)$. The geometric proof of Theorem \ref{thm_q_P}, following \cite{lutz}, avoids such case-by-case explicit analysis. 
\end{remark}

\section{Affine Weyl groups and del Pezzo surfaces}
\label{sec_weyl_del_pezzo}

In this section, we introduce and examine the relationships between the finite Weyl group $W(Z)$, the affine Weyl group $\widehat{W}(Z)$ and the orthogonal group $O(Z)$, all acting on the Grothendieck group $K(Z)$ of the derived category of coherent sheaves on a del Pezzo surface $Z$.

\subsection{Del Pezzo surfaces, Euler form, and Weyl groups}
\label{sec_dP}

Let $Z$ be a del Pezzo surface over $\bC$, that is, a smooth projective connected surface over $\bC$ with ample anticanonical line bundle $\omega_Z^{-1}$. Then $Z$ is isomorphic either to $\bP^1 \times \bP^1$ or to 
the blow-up of $\bP^2$ at $m$ points in general position,
for some integer $0 \leq m \leq 8$ -- see for example \cite[\S 8]{dolgachev}. We write $NS(Z)$ for the N\'eron-Severi group of $Z$, and $\Pic(Z)$ for its Picard group. Since $H^1(Z,\cO_Z)=0$, the first Chern class induces an isomorphism $c_1: \Pic(Z) \xrightarrow{\sim} NS(Z)$.
We denote by $K_Z=c_1(\omega_Z)$ the canonical divisor class and by $\alpha \cdot \beta$ the intersection pairing of divisors $\alpha, \beta \in NS(Z)$. One has $K_Z^2>0$, and by the Hodge index theorem the restriction of the intersection form to the orthogonal complement
\[ K_Z^\perp:= \{ \alpha \in NS(Z) \,|\, \alpha \cdot K_Z = 0 \}\]
is negative definite.

Let $\D(Z)=\D^b\Coh(Z)$ be the bounded derived category of coherent sheaves on $Z$, and $K(Z)=K_0(\D(Z))$ its Grothendieck group. For an object $E$ in $D(Z)$, we write $[E] \in K(Z)$ for its class in $K(Z)$, $r(E)$ for its rank,
$d(E)=c_1(E)\cdot (-K_Z)$ for its anticanonical degree, and $\chi(E)$ for its Euler characteristic. The numerical invariants $r(E)$, $d(E)$, and $\chi(E)$ depend only on the class $[E] \in K(Z)$.  This gives rise to a linear map
\begin{align*}
\psi=(r,d) : K(Z) &\longrightarrow \bZ^2 \\
[E] &\longmapsto (r(E), d(E)) \,.
\end{align*}

Let \(\delta \in K(Z)\) denote the class $[\cO_x]$ of the skyscraper sheaf at any point $x \in Z$. This class is independent of the choice of $x$ and satisfies 
$(r(\delta), d(\delta),\chi(\delta))=(0,0,1)$.
The Grothendieck group $K(Z)$ is a free abelian group of rank $3\leq n\leq 11$, and admits a decomposition
\begin{align} \label{eq_K_Z}
    K(Z) &\xrightarrow{\sim} \bZ [\cO_Z] \oplus NS(Z) \oplus \bZ \delta \\
    [E]& \longmapsto (r(E), c_1(E), \chi(E)-r(E)) \,. \nonumber
\end{align}
By the Riemann--Roch theorem, one could replace $\chi(E)-r(E)$ in \eqref{eq_K_Z} by $\mathrm{ch}_2(E)+\frac{d(E)}{2}$.
With respect to this decomposition, the  kernel of the degree map is
\[ \mathrm{Ker}(d)= \bZ[\cO_Z] \oplus K_Z^\perp \oplus \bZ \delta\,,\]
the kernel of the rank map is 
\[ \mathrm{Ker}(r)=   NS(Z) \oplus \bZ \delta\,,\]
and the kernel of $\psi=(r,d)$ is
\begin{equation} \label{eq_ker_psi}
\mathrm{Ker}(\psi) = K_Z^\perp \oplus \bZ \delta \,.\end{equation}
We denote by $\chi(-,-)$ the Euler form on $K(Z)$, defined by 
\[ \chi(E,F)=\sum_{i\in \bZ}(-1)^i \dim \Hom^i(E,F) \,.\]
By the Riemann-Roch theorem, the Euler form is given  explicitly by 
\begin{equation} \label{eq_RR}
\chi(E,F)= r(E)\chi(F)+r(F)\chi(E)-r(E)r(F) -c_1(E) \cdot c_1(F) -r(F) d(E) \,.\end{equation}
In particular, we have $\chi(\delta, \delta)=0$, and $\chi(\delta, E)=\chi(E,\delta)=r(E)$ for any object $E$ of $D(Z)$.

\begin{lemma} \label{lem_sym}
Let $Z$ be a del Pezzo surface.
\begin{itemize}
\item[(i)] The skew-symmetrization of the Euler form $\chi(-,-)$ is given by:
\begin{equation}
\chi(E,F) - \chi(F,E) =r(E)d(F)-r(F)d(E) \,.
\end{equation}
\item[(ii)] The restriction of the Euler form $\chi(-,-)$ to
$\mathrm{Ker}(d)$ is symmetric: for every $[E], [F] \in \mathrm{Ker}(d)$, one has
\begin{equation} \label{eq_chi_A} \chi (E,F)=r(E)\chi(F)+r(F)\chi(E)-r(E)r(F) -c_1(E) \cdot c_1(F) \,.\end{equation}
Equivalently, for every $(r,C,m) \in  \mathrm{Ker}(d) = \bZ [\cO_Z] \oplus K_Z^{\perp} \oplus \bZ \delta$, one has
\begin{equation} \label{eq_chi_new}
\chi((r,C,m), (r',C',m')) = rm'+r'm +rr' -C \cdot C' \,.\end{equation}
\item[(iii)] The restriction of the Euler form $\chi(-,-)$ to $\mathrm{Ker}(r) = NS(Z) \oplus \bZ \delta$ is given by the negative of the intersection form on $NS(Z)$:
for every $C, C' \in NS(Z)$ and $m,m' \in \bZ$, one has 
\[ \chi((0,C,m), (0,C',m')) = - C \cdot C'\,.\]
\end{itemize}
\end{lemma}

\begin{proof} The result follows immediately from \eqref{eq_RR}.
\end{proof}

Finally, we define affine and finite roots and Weyl groups of a del Pezzo surface.

\begin{definition}  \label{def_affine}
Let $Z$ be a del Pezzo surface.
An \emph{affine root} of $Z$ is an element $\alpha \in \mathrm{Ker}(\psi) =  K_Z^\perp \oplus \bZ \delta$ such that $\chi(\alpha,\alpha)=2$.
We denote by $\widehat{\Phi}_Z$ the set of affine roots of $Z$.
For every $\alpha \in \widehat{\Phi}_Z$, the reflection of $K(Z)$ with respect to $\alpha$ is the linear automorphism
\begin{align} \label{eq_reflections}
    r_\alpha : K(Z) &\longrightarrow K(Z) \\
    \beta &\longmapsto \beta - \chi(\beta, \alpha) \alpha \,. \nonumber
\end{align}
The \emph{affine Weyl group} $\widehat{W}(Z)$ is the group of linear automorphisms of $K(Z)$ generated by the reflections with respect to the affine roots $\widehat{\Phi}_Z$.
\end{definition}

\begin{remark}
Every affine root $\alpha \in \widehat{\Phi}_Z$ belongs to $\mathrm{Ker}(\psi)$ by definition, and so $\chi(\alpha,\beta)=\chi(\beta, \alpha)$ for every $\beta \in K(Z)$ by Lemma \ref{lem_sym}(i).
\end{remark}

\begin{definition} \label{def_weyl_finite}
    Let $Z$ be a del Pezzo surface. A \emph{finite root} of $Z$ is an affine root contained in $K_Z^\perp$. We denote by $\Phi_Z$ the set of finite roots of $Z$. 
    The \emph{finite Weyl group} $W(Z)$ is the group of linear automorphisms of $K(Z)$ generated by reflections with respect to the finite roots $\alpha \in \Phi_Z$.
\end{definition}

The group $W(Z)$ fixes the classes $K_Z, [\cO_Z], \delta \in K(Z)$ and preserves the subgroup $K_Z^\perp$. Thus, $W(Z)$ is determined by its action by isometries on $K_Z^\perp$.
By the Hodge index theorem, the intersection form is negative definite on $K_Z^\perp$, 
and so $W(Z)$ is a finite group.

\begin{prop} \label{prop_finite_weyl}
Let $Z$ be a del Pezzo surface. Then, the finite Weyl group $W(Z)$ is isomorphic to the group of isometries of $NS(Z)$ fixing $K_Z$.
\end{prop}

\begin{proof}
When $Z$ is isomorphic to the blow-up of $\bP^2$ in $3 \leq m \leq 8$ points, this is \cite[Corollary 8.2.15]{dolgachev}. The remaining cases are verified directly as follows. We use the fact that if the reflection with respect to some primitive $\alpha \in NS(Z)$ preserves $NS(Z)$, then $|\alpha^2| \leq 2$:
indeed, this reflection is given by 
\[ \beta \longmapsto \beta -2 \frac{\beta \cdot \alpha}{\alpha^2}  \alpha\]
and as $NS(Z)$ is unimodular and $\alpha$ is primitive, there exists $\beta$ such that $\beta \cdot \alpha=1$, and so $2/\alpha^2\in \bZ$. 
If $Z=\bP^2$, then $K_Z^\perp=0$, and so the result is trivial. If $Z=\bP^1 \times\bP^1$, then $K_Z^\perp=\bZ \beta$ with $\beta^2=-2$, and so the only non-trivial isometry of $K_Z^{\perp}$ is the reflection with respect to $\beta$. If $Z$ is the blow-up of $\bP^2$ in one point, then, as reviewed in Lemma \ref{lem_del_pezzo_lie}(iii) below, $K_Z^\perp=\bZ \alpha$ with $\alpha^2=-8$. The only non-trivial isometry of $K_Z^\perp$ is the reflection with respect to $\alpha$, which does not extend to an isometry of $NS(Z)$ fixing $K_Z$ because $|\alpha^2|=8>2$. Similarly, if $Z$ is the blow-up of $\bP^2$ in two points, then, as reviewed in Lemma \ref{lem_del_pezzo_lie}(ii) below, there exists a unique up to sign class $\beta \in K_Z^\perp$ with $\beta^2=-2$, and the orthogonal of $\bZ \beta$ in $K_Z^\perp$ is $\bZ \alpha$ with $\alpha^2=-14$. 
An isometry $f$ of $K_Z^{\perp}$ must send $\beta$ to $\pm \beta$. Upon composing with the reflection with respect to $\beta$, we may assume that $f(\beta)=\beta$. If $f$ is non-trivial, then $f$ is the reflection with respect to $\alpha$, which again does not extend to an isometry of $NS(Z)$ fixing $K_Z$ because $|\alpha^2|=14>2$.
\end{proof}

\subsection{Orthogonal groups and finite Weyl groups}

In this section, we introduce the orthogonal group of a del Pezzo surface $Z$ as a group of isometries of $K(Z)$, and we describe it explicitly in terms of the finite Weyl group in Proposition \ref{prop_weyl_affine_gen}.

\begin{definition} \label{def_orthogonal}
    Let $Z$ be a del Pezzo surface. The \emph{orthogonal group} of $Z$, 
    denoted by $O(Z)$, is the group of isometries of $(K(Z), \chi(-,-))$ preserving the rank and the degree.
\end{definition}

The finite and affine Weyl groups $W(Z)$ and $\widehat{W}(Z)$ are naturally subgroups of $O(Z)$.
Define the subgroup $\Pic^0(Z)$ of $\Pic(Z)$ by 
\[ \Pic^0(Z):=\{ L \in \Pic(Z)\,|\, c_1(L) \cdot K_Z= 0\} \,.\]

\begin{lemma} \label{lem_tensor_0}
The action of $\Pic^0(Z)$ on $K(Z)$ by tensor product realizes $\Pic^0(Z)$ as a subgroup of the orthogonal group $O(Z)$.
\end{lemma}

\begin{proof}
For every $[E], [F] \in K(Z)$, and $L=\cO(D) \in \Pic^0(Z)$ with $D\in K_Z^\perp$,  we have $r(E \otimes L)= r(E)$, $c_1(E \otimes L)=c_1(E)+ r(E) D$, and $\chi(E \otimes L, F\otimes L)=\chi(E,F)$.
\end{proof}

\begin{lemma} \label{lem_tensor}
For every $L=\cO_Z(D) \in \Pic^0(Z)$ with $D \in K_Z^\perp$, the action of $L$ on $K(Z)$ by tensor product fixes the class $K_Z \in NS(Z) \subset K(Z)$, preserves the subgroup $\bZ [\cO_Z] \oplus K_Z^\perp \oplus \bZ \delta$, and acts on 
the latter as follows:
\begin{align} \label{eq_t_L}
    t_L: \bZ [\cO_Z] \oplus K_Z^\perp \oplus \bZ \delta &\longrightarrow \bZ [\cO_Z] \oplus K_Z^\perp \oplus \bZ \delta\\
     (r, C, m) &\longmapsto (r, C+r D, m+C \cdot D + \frac{r}{2} D^2) \,. \nonumber
\end{align}
\end{lemma}

\begin{proof}
For every $[E] \in K(Z)$, we have $r(E \otimes L)= r(E)$, $c_1(E \otimes L)=c_1(E)+ r(E) D$, and, 
by the Riemann--Roch theorem:
\[ \chi(E \otimes L)  = \chi(E)+c_1(E) \cdot D + \frac{r(E)}{2}(D^2-D\cdot K_Z)=\chi(E)+c_1(E) \cdot D + \frac{r(E)}{2} D^2\,,\]
where we used that $D \cdot K_Z=0$. 
\end{proof}

\begin{lemma} \label{lem_O}
For every $f \in O(Z)$, there exists $L \in Pic^0(Z)$ such that $f([\cO_Z])=[L]$.
\end{lemma}

\begin{proof}
Since $f$ preserves both the rank and the degree, the class $f([\cO_Z])$
is of the form $f([\cO_Z])=(1,D,m) \in \bZ [\cO_Z] \oplus K_Z^{\perp} \oplus \bZ \delta$. 
Because $f$ is also an isometry for the Euler form, we obtain from \eqref{eq_chi_new}:
\[1 = \chi(\cO_Z, \cO_Z) = \chi(f([\cO_Z]), f([\cO_Z])) 
= 2m+1-D^2 \,.\]
It follows that $m=\frac{D^2}{2}$. On the other hand, since $D\in K_Z^\perp$, the Riemann-Roch theorem yields $\chi(\cO_Z(D))=1+\frac{D^2}{2}$. Therefore, $\chi(\cO_Z(D))-r(\cO_Z(D))=\frac{D^2}{2}=m$, and so the line bundle $L:=\cO_Z(D) \in \Pic^0(Z)$ satisfies $[L]=f([\cO_Z])$.
\end{proof}

\begin{prop} \label{prop_weyl_affine_gen}
The orthogonal group $O(Z)$ is generated by the finite Weyl group $W(Z)$ and the subgroup $\Pic^0(Z)$. Moreover, these subgroups combine as a semi-direct product, so that
\[ O(Z)=W(Z)\ltimes \Pic^0(Z) \,. \]
\end{prop}

\begin{proof}
By Lemma \ref{lem_O}, for any $f \in O(Z)$, there exists $g \in \Pic^0(Z)$ such that $g \circ f([\cO_Z])=[\cO_Z]$. Therefore, to prove the generation of $O(Z)$ by $W(Z)$ and $\Pic^0(Z)$, it suffices to prove that every $f \in O(Z)$ such that $f([\cO_Z])=[\cO_Z]$ lies in $W(Z)$. 

Let $f \in O(Z)$ fixing $[\cO_Z]$. 
As $f$ preserves the rank, it preserves the decomposition of $K(Z)$ into the direct sum of $\bZ[\cO_Z]$ and $NS(Z) \oplus \bZ \delta$.
We first analyze the image of $\delta$. 
Since $f$ preserves rank and degree, we must have $f(\delta)\in K_Z^{\perp}\oplus \bZ \delta$. 
Now $\chi(\delta,\delta)=0$, and the Euler form is positive definite on $K_Z^\perp$ by the Hodge index theorem, and zero on $\bZ \delta$, so $f(\delta)=\pm \delta$. To determine the sign, observe that, since $f$ fixes $[\cO_Z]$, we have $\chi([\cO_{Z}],f(\delta))=\chi(f([\cO_Z]),f(\delta))=
\chi([\cO_Z], \delta)=1$. This implies that $f(\delta)=\delta$.

Next consider the action of $f$ on $NS(Z)$. For any class $C \in NS(Z)$, the image $f(C)$ is of the form $(0,C',m) \in \bZ[\cO_Z] \oplus NS(Z) \oplus \bZ \delta$. 
Since $f$ preserves the Euler form and fixes $[\cO_Z]$, we have 
\[ 0=\chi((0,C,0))=\chi([\cO_Z],(0,C,0))=\chi([\cO_Z], (0,C',m))=\chi((0,C',m))=m\,,\]
thus $m=0$. Therefore, $f$ preserves the decomposition $NS(Z) \oplus \bZ \delta$ and acts trivially on $\bZ \delta$. Finally, as $f$ preserves the degree and the intersection form is non-degenerate on $NS(Z)$, we necessarily have $f(K_Z)=K_Z$. The restriction of $f$ to $NS(Z)$ is an isometry fixing $K_Z$, and hence $f \in W(Z)$ by Proposition 
\ref{prop_finite_weyl}. This concludes the proof that the orthogonal group $O(Z)$ is generated by $W(Z)$ and $\Pic^0(Z)$.

To prove that $O(Z)=W(Z)\ltimes \Pic^0(Z)$, first observe that the intersection of $W(Z)$ and $\Pic^0(Z)$ is reduced to the identity since $W(Z)$ is finite and $\Pic^0(Z)$ is torsion free. Therefore, it remains to show that $\Pic^0(Z)$ is a normal subgroup of $O(Z)$. 
Any element of $O(Z)$ preserves $\mathrm{Ker}(r,d)=K_Z^\perp \oplus \bZ \delta$, and preserves $\bZ \delta$ since 
$\bZ \delta$ is the kernel of the Euler form. Therefore, there exists an induced action on the quotient $(K_Z^\perp \oplus \bZ \delta)/\bZ \delta \simeq K_Z^\perp$, and so a group homomorphism $O(Z) \rightarrow O(K_Z^\perp)$, where $O(K_Z^\perp)$ is the group of isometries of $K_Z^\perp$. 
Then $\Pic^0(Z)$ is exactly the kernel of
$O(Z) \rightarrow O(K_Z^\perp)$ and so is a normal subgroup of $O(Z)$.
This completes the proof.
\end{proof}

\subsection{Orthogonal groups and affine Weyl groups}
In this section we prove Proposition~\ref{prop_orthogonal_affine_weyl}, which describes the relation between the orthogonal group of a del Pezzo surface $Z$ and its affine Weyl group. We also explain how this affine Weyl group naturally identifies with those arising in the theory of affine Lie algebras.
These results are not used in the proof of the main theorem in \S\ref{sec_final} and are included for their intrinsic interest.

\begin{definition}
 The \emph{finite root lattice} $R(Z)$
    of a del Pezzo surface $Z$ is the subgroup of $K_Z^\perp$ generated by the finite roots $\alpha \in \Phi_Z$.
\end{definition}

\begin{lemma} \label{lem_del_pezzo_lie}
    Let $Z$ be a del Pezzo surface. Then, the following holds:
    \begin{enumerate}
        \item[(i)] If $Z$ is not isomorphic to the blow-up of $\bP^2$ in one or two points, then
        \[ K_Z^\perp = R(Z)\,.\]
        Moreover, $R(Z)$ is the root lattice of a semisimple Lie algebra $\mathfrak{g}_Z$ in the list 
        \[ A_0\,, \,\,\,A_1\,,\,\,\,  A_1 + A_2\,, \,\,\,  A_4\,,\, \,\, D_5\,, \,\,\, E_6\,,\,\,\, E_7\,, \,\,\, E_8 \,.\]
        \item[(ii)] If $Z$ is isomorphic to the blow-up of $\bP^2$ in two points, then $R(Z)$ is the root lattice of the simple Lie algebra $\mathfrak{g}_Z$ of type $A_1$, that is, $R(Z)=\bZ \beta $ with $\beta \cdot \beta=-2$, and the orthogonal of $R(Z)$ in $K_Z^\perp$ is $R(Z)^\perp= \bZ \alpha$
        with $\alpha \cdot \alpha = -14$. 
        \item[(iii)] If $Z$ is isomorphic to the blow-up of $\bP^2$ in one point, then $R(Z)= \emptyset$ and
        \[ K_Z^\perp = \bZ \alpha\]
        with $\alpha \cdot \alpha = -8$.
    \end{enumerate} 
\end{lemma}

\begin{proof}
This is a classical result in the theory of del Pezzo surfaces, following from the description of $Z$ as a blow-up of $\bP^2$, see for instance \cite[Chapter IV]{manin} or
\cite[Chapter 8]{dolgachev}. 
In (ii), one can take $\beta=E_1-E_2$ and $\alpha=-2H+3E_1+3E_2$,   
where $H$ is the pullback of the class of a line in $\bP^2$, and $E_1$, $E_2$ are the exceptional curves of the blow-up $Z \rightarrow \bP^2$. In (iii), one can take $\alpha=H-3E$, where $H$ is the pullback of the class of a line in $\bP^2$, and $E$ is the exceptional curve of the blow-up $Z \rightarrow \bP^2$.
\end{proof}

By Lemma~\ref{lem_del_pezzo_lie}, the real vector space 
$R(Z) \otimes \bR$ admits a canonical identification with the dual of the Cartan subalgebra of a finite-dimensional semisimple Lie algebra $\mathfrak{g}_Z$.
We explain below how enlarging $R(Z)$ to
$\bZ [\cO_Z] \oplus R(Z) \oplus \bZ \delta$
naturally gives rise to the dual of the Cartan subalgebra of an affine Lie algebra. 
For background on affine Lie algebras, we refer to \cite{kac}.

\begin{prop} \label{prop_affine_lie}
Let $Z$ be a del Pezzo surface. Then, $(\bZ [\cO_Z] \oplus R(Z) \oplus \bZ \delta) \otimes \bR$ 
endowed with the negative of the restriction of the Euler form is isometric to the dual of the Cartan subalgebra of an affine Lie algebra $\widehat{\mathfrak{g}}_Z$. Moreover, the affine roots of $Z$ are the real roots of the corresponding affine root system.
\end{prop}

\begin{proof}
Using the notation of \cite[Chapter 6]{kac}, the dual of the Cartan subalgebra of an affine Lie algebra is of the form 
\[ \mathfrak{h}^\star = \bR \Lambda_0 \oplus \mathring{\mathfrak{h}}^\star \oplus \bR \delta \,,\]
where $\mathring{\mathfrak{h}}^\star$  is the dual of the Cartan subalgebra of a finite-dimensional semisimple Lie algebra.
Comparing \cite[Eq. (6.2.7)]{kac} with \eqref{eq_K_Z}, and \cite[Eq. (6.2.1)-(6.2.4)]{kac} with \eqref{eq_chi_new}, we obtain, that 
\[ \Lambda_0 \mapsto -[\cO_Z]+\frac{\delta}{2}\,,\,\,\,\, \mathring{\mathfrak{h}}^\star \simeq R(Z) \otimes \bR\,,\,\,\,\,\,\delta \mapsto \delta\,,\] defines an isometry between $(\bZ [\cO_Z] \oplus R(Z) \oplus \bZ \delta) \otimes \bR$  and the dual of the Cartan subalgebra of an affine Lie algebra $\widehat{\mathfrak{g}}_Z$. 
\end{proof}

\begin{remark} The affine Lie algebras $\widehat{\mathfrak{g}}_Z$ appearing by Proposition \ref{prop_affine_lie} form the list
\[ A_0^{(1)}\,,\,\,\,A_1^{(1)}\,,\,\,\,   (A_2+A_1)^{(1)}\,,\,\,\, A_4^{(1)}\,,\, \,\, D_5^{(1)}\,, \,\,\, E_6^{(1)}\,, \,\,\, E_7^{(1)}\,, \,\,\, E_8^{(1)} \,.\]
\end{remark}

\begin{prop} \label{prop_orthogonal_affine_weyl}
    Let $Z$ be a del Pezzo surface. Then, the following holds:
    \begin{enumerate}
        \item[(i)] If $Z$ is not isomorphic to the blow-up of $\bP^2$ in one or two points, then
        \[ O(Z) = \widehat{W}(Z)\,.\]
        \item[(ii)] If $Z$ is isomorphic to the blow-up of $\bP^2$ in two points, then 
        \[ O(Z) = \widehat{W}(Z) \ltimes  \bZ\,.\]
        \item[(iii)] If $Z$ is isomorphic to the blow-up of $\bP^2$ in one point, then $\widehat{W}(Z)$ is trivial and
        \[ O(Z)=\bZ\,.\]
       
    \end{enumerate} 
\end{prop}

\begin{proof}
By Lemma \ref{lem_tensor} and Proposition \ref{prop_weyl_affine_gen}, both  $O(Z)=W(Z)\ltimes \Pic^0(Z)$ and $\widehat{W}(Z)$ fix the class 
$K_Z \in NS(Z) \subset K(Z)$
and preserve $\mathrm{Ker}(d)$. 
Therefore, it suffices to prove the result in restriction to $\mathrm{Ker}(d)$.
For (i), assume that $Z$ is not isomorphic to the blow-up of $\bP^2$ in one or two points. By Lemma \ref{lem_del_pezzo_lie}(i), we have $K_Z^\perp=R(Z)$, and so by Proposition \ref{prop_affine_lie}, $\mathrm{Ker}(d) \otimes \bR$ is isometric to the dual of the Cartan subalgebra 
\[\mathfrak{h}^\star = \bR \Lambda_0 \oplus \mathring{\mathfrak{h}}^\star \oplus \bR \delta\]
of an affine Lie algebra $\widehat{\mathfrak{g}}_Z$. 
Therefore, \cite[Proposition 2.5]{kac} identifies the affine Weyl group $\widehat{W}(Z)$ with the semi-direct product 
\[ W(Z)\ltimes \mathring{\mathfrak{h}}^\star\]
where $\mathring{\mathfrak{h}}^\star = K_Z^\perp$ acts by ``translations" given explicitly in \cite[Eq. 6.5.2]{kac}. Using the identification between $\mathrm{Ker}(d) \otimes \bR$ and $\mathfrak{h}^\star$ given in the proof of Proposition \ref{prop_affine_lie}, we see that the action of $K_Z^\perp= \mathring{\mathfrak{h}}^\star$ on $\mathfrak{h}^\star$ given by \cite[Eq. 6.5.2]{kac} agrees precisely with the action of $\Pic^0(Z) \simeq K_Z^\perp$
given by Lemma \ref{lem_tensor}. For (ii)-(iii), the result follows similarly using Lemma \ref{lem_del_pezzo_lie}(ii)-(iii).
\end{proof}

\section{Geometric helices on del Pezzo surfaces and tilting operations} 
\label{sec_very_strong}

In this section, we first recall general facts about exceptional collections on del Pezzo surfaces and then review following \cite{BS} the notions of geometric helices and tilting operations.

\subsection{Exceptional collections on del Pezzo surfaces}
\label{sec_ex_collections}
In this section, we recall several foundational results on exceptional collections on del Pezzo surfaces, mainly due to Gorodentsev \cite{Gor} and Kuleshov–Orlov \cite{KO}. We adopt the notation introduced in \S\ref{sec_dP} for the derived categories of del Pezzo surfaces.
We define the slope of a vector bundle $E$ on a del Pezzo surface $Z$ by  $\mu(E)=d(E)/r(E)\in \bQ$. We then define slope stability in the usual way. If $E$ is a torsion sheaf, we set $\mu(E)=\infty$.

An object $E$ in the derived category $D(Z)$ is \emph{exceptional} if $\Hom(E,E)=\bC$ and $\Hom^k(E,E)=0$ for all $k \neq 0$. By \cite[Proposition 2.10]{KO}, every exceptional object in $D(Z)$ is isomorphic to a shift of a coherent sheaf. 
Moreover, if $E$ is an exceptional sheaf, then, by \cite[Proposition 2.9]{KO} and \cite[2.4.\! Theorem]{Gor}, either
\begin{itemize}
\item[(i)] $E$ is a slope stable vector bundle,  or 
\item[(ii)] $E=\cO_D(k)$ is the push-forward of a line bundle $\cO_{\bP^1}(k)$, $k \in \bZ$, from an  exceptional curve $D \simeq \bP^1 \subset Z$.
\end{itemize}
By \cite[Corollary 2.5]{Gor}, an exceptional sheaf $E$ is uniquely determined by its class $[E] \in K(Z)$ in the Grothendieck group.

\begin{lemma} \label{lem_primitive}
Let $E$ be an exceptional object in $D(Z)$. Then:
\begin{itemize}
\item[(i)] $(r(E), d(E)) \neq (0, 0)$. 
\item[(ii)] $r(E)$ and $d(E)$ are coprime.
\end{itemize}
\end{lemma}

\begin{proof}
By \cite[Proposition 2.10]{KO}, every exceptional object in $D(Z)$ is a shift of a sheaf on $Z$, so it suffices to prove the result when $E$ is a sheaf.  When $E=\cO_{D}(k)$ for an exceptional curve $D$, we have $(r(E), d(E))=(0,1)$ and so the result holds.
When $E$ is an exceptional bundle on $Z$, the restriction 
$E|_C$ to a smooth anticanonical curve $C \subset Z$ is a simple vector bundle by \cite[Lemma 3.6]{KO}. 
Viewed as a vector bundle on the elliptic curve 
$C$, it has rank $r(E)$ and degree
$d(E)$.
By Atiyah’s classification of vector bundles on elliptic curves \cite{atiyah}, the rank and degree of a simple vector bundle on 
$C$ are coprime. Hence 
$r(E)$ and $d(E)$ are coprime, which proves the claim.
\end{proof}

An \emph{exceptional collection} $\bE=(E_1,\dots, E_m)$ is a sequence of exceptional objects $E_i$ in $D(Z)$ such that $\Hom^k(E_j,E_i)=0$ for all $k \in \bZ$ and $1 \leq i <j \leq m$. An exceptional pair is an exceptional collection consisting of two objects. An exceptional collection $\bE$ is \emph{full} if the smallest full triangulated subcategory of $D(Z)$ containing all objects of $\bE$ is $D(Z)$. If $\bE=(E_1,\dots,E_n)$ is a full exceptional collection, then $([E_i])_{1\leq i\leq n}$ is a basis of $K(Z)$, and so the length $n$ of $\bE$ has to be equal to the rank of $K(Z)$. Conversely, since $Z$ is a del Pezzo surface, any exceptional collection of length $n=\mathrm{rk}\, K(Z)$ is full by \cite[Theorem 6.11]{KO}.

\begin{lemma} \label{lem_ex_0}
\label{any}
Suppose $(E_1,E_2)$ is an exceptional pair in $D(Z)$. Then 
\begin{equation}
\label{eq_chi_ex}
\chi(E_1,E_2)=r(E_1)d(E_2)-r(E_2)d(E_1).\end{equation}
Moreover, if $(E_1,E_2)$ is an exceptional pair of sheaves on $Z$, then $\Ext^2(E_1,E_2)=0$  and at most one of the spaces $\Ext^k(E_1,E_2)$ with $k=0,1$ is non-zero.
\end{lemma}

\begin{proof}
Since $(E_1,E_2)$ is an exceptional pair, we have $\chi(E_2, E_1)=0$, and so \eqref{eq_chi_ex} follows from Lemma \ref{lem_sym}(i).
The second part of Lemma \ref{lem_ex_0} is precisely \cite[Corollary 2.11]{KO}
\end{proof}

\begin{lemma} \label{lem_ex}
Let $(E_1,E_2)$ be an exceptional pair of sheaves on $Z$, then the following holds:
\begin{itemize}
\item[(i)] $\Hom(E_1,E_2)\neq  0\iff \mu(E_1)<\mu(E_2)$,
\item[(ii)] $\Ext^1(E_1,E_2)\neq  0\iff\mu(E_1)>\mu(E_2)$,
\item[(iii)] $\Ext^k(E_1,E_2)=0\,\,\, \text{for all}\,\,\, k\in \bZ\iff \mu(E_1)=\mu(E_2)$.
\end{itemize}
\end{lemma}

\begin{proof}
The result follows immediately from Lemma \ref{lem_ex_0}.
\end{proof}

Two objects $E_1$ and $E_2$ in $D(Z)$ are said to be \emph{orthogonal} if $\Hom^k(E_1, E_2)=\Hom^k(E_2,E_1)=0$ for all $k \in \bZ$. Since every exceptional object in $D(Z)$ is the shift of a coherent sheaf, it follows from Lemma \ref{lem_ex}(iii) that two exceptional objects $E_1$ and $E_2$ in $D(Z)$ are orthogonal if and only if 
$\mu(E_1)=\mu(E_2)$. 
%Given an exceptional collection $\bE=(E_1, \dots, E_m)$, a \emph{block} in $\bE$ is a sequence $(E_{i+1}, \dots, E_{i+k})$ of consecutive objects in $\bE$ which are mutually orthogonal.

\subsection{Geometric helices on del Pezzo surfaces}
Let $Z$ be a del Pezzo surface, and set $n:= \mathrm{rk}\,K(Z)$.
A \emph{helix} on $Z$ is a sequence of objects $\bH= (E_i)_{i \in \bZ}$ in $\D(Z)$ such that
\begin{itemize}
\item[(i)] for each $i\in \bZ$ the corresponding thread $(E_{i+1}, \ldots, E_{i+n})$ is a full exceptional collection in $\D(Z)$,
\item[(ii)] $E_{i+n}=E_i\tensor \omega_Z^{-1}$ for all $i\in \bZ$.
\end{itemize}
By \cite[Remark 3.2]{BS}(c), it is in fact sufficient to require that a single thread of 
$\bH$ is a full exceptional collection.
The following notion appears under various names in
\cite{bergh2002non, bridgeland_tstructures, BS,  Her, Her_Karp, segal}. We adopt the terminology introduced in \cite{BS}.

\begin{definition}
A helix $\bH = (E_i)_{i\in \bZ}$ on $Z$ is \emph{geometric} if for all $i<j$
 \[\Hom^k(E_i,E_j)=0\text{ unless }k=0.\]
\end{definition} 

The following exceptional collections are precisely the threads of geometric helices. We follow the terminology introduced in \cite{nord}.

\begin{definition} 
A \emph{very strong exceptional collection} on $Z$ is a full exceptional collection $\bE=(E_1,\dots, E_n)$  such that
\[\Hom^k(E_i,E_j\tensor \omega_Z^{-p})=0,\]
for all $k\neq 0$ when $1\leq i,j\leq n$ and $p\geq 0$.
\end{definition}

Recall that an exceptional collection $\bE=(E_1, \cdots, E_n)$ is \emph{strong} if $\Hom^k(E_i, E_j)=0$ for all $k \neq 0$ and $1 \leq i, j \leq n$. A \emph{cyclic strong} exceptional collection is a full exceptional collection $\bE$ such that, denoting by $(E_i)_{i\in \bZ}$ the corresponding helix, the exceptional collection $(E_{i+1}, \dots, E_{i+n})$ is strong for all $i\in \bZ$. A very strong exceptional collection is clearly cyclic strong. The following results show that the converse also holds on del Pezzo surfaces.

\begin{lemma} \label{lem_bundles}
Let $Z$ be a del Pezzo surface. 
Suppose $\bE'=(E'_1,\ldots, E'_n)$ is a cyclic strong exceptional collection on $Z$. Then there is a unique $p\in \bZ$ such that all objects $E_i=E'_i[p]$ lie in $\Coh(Z)$. Moreover, each $E_i$ is a vector bundle.
\end{lemma}

\begin{proof}
Let $\bE'=(E'_1,\ldots, E'_n)$ be a cyclic strong exceptional collection on $Z$. Suppose by contradiction that $\bE'$ is not the global shift of a very strong exceptional collection of sheaves. Then, since $D(Z)$ is indecomposable by \cite[Proposition 3.10]{Huy}, there exist $i<j$ such that $E_i'$, $E_j'$ are shifts of sheaves by different integers and $\Hom(E_i', E_j') \neq 0$. Up to a applying a global shift to the collection, we may assume that $E_i'$ is a sheaf, and so $E_j'=E_j[k]$ where $E_j$ is a sheaf and $k \in \bZ \setminus \{0\}$. Then $\mathrm{Ext}^k(E_i', E_j)= \Hom(E_i', E_j[k])=\Hom(E_i', E_j') \neq 0$. Since $(E_i', E_j)$ is an exceptional pair of sheaves, Lemma \ref{lem_ex_0} implies $k=1$ and Lemma \ref{lem_ex}(ii) implies $\mu(E_i')>\mu(E_j)$. 
Consider now the helix $(E_\ell')_{\ell \in \bZ}$ generated by $\bE'$. By construction, we have $\mu(E_{i+n}')=\mu(E_i' \otimes \omega_Z^{-1})=\mu(E_i')+K_Z^2 > \mu(E_j)$.
On the other hand, since $j <i+n <j+n$ and  $(E_j',\dots, E_{j+n-1}')$ is an exceptional collection, $(E_j, E_{i+n}')$ is an exceptional pair of sheaves.
Applying Lemma \ref{lem_ex}(i), we deduce that $\Hom(E_j, E_{i+n}') \neq 0$ and so 
$\Hom^{1}(E_j', E_{i+n}')=\Hom^{1}(E_j[1], E_{i+n}')=\Hom(E_j, E_{i+n}')\neq 0$, in contradiction with the fact that
$(E_j',\dots, E_{j+n-1}')$ is strong.

It remains to prove that if $\bE=(E_1, \dots, E_n)$ be a cyclic strong exceptional collection of sheaves, then every $E_i$ is a vector bundle. 
Assume by contradiction that $\bE$ contains torsion sheaves. Since $\bE$
 is full, it cannot consist entirely of torsion sheaves. Moreover, $D(Z)$ is indecomposable, and so there exists $1 \leq i,j \leq n$ such that $E_i$ is torsion free, $E_j$ is torsion and either $\Hom(E_i, E_j)\neq 0$ or $\Hom(E_j,E_i) \neq 0$. Since there are no non-zero maps from a torsion sheaf to a torsion free sheaf, we have $\Hom(E_j,E_i)=0$, so $\Hom(E_i,E_j) \neq 0$ and necessarily $i<j$. 
 As above, consider the helix $(E_\ell')_{\ell \in \bZ}$ generated by $\bE'$. Then $(E_j, E_{i+n})$ is an exceptional pair of sheaves in the exceptional collection $(E_j, \dots, E_{j+n-1})$, with $E_{i+n}=E_i \otimes \omega_Z^{-1}$ torsion free, and so with $\mu(E_j)=\infty >\mu(E_{i+n})$.
 Then Lemma \ref{lem_ex}(ii) implies that
 $\Ext^1(E_j,E_{i+n}) \neq 0$, contradicting the assumption that $(E_j, \dots, E_{j+n-1})$ is strong. This concludes the proof of (i).
\end{proof}

\begin{prop} \label{prop_bundles}
Let $\bE=(E_1,\ldots, E_n)$ be a full exceptional collection on a del Pezzo surface $Z$. Then, the following are equivalent:
\begin{itemize}
\item[(i)] $\bE$ is very strong.
\item[(ii)] $\bE$ is cyclic strong.
\item[(iii)] Up to a shift, $\bE$ is a full strong exceptional collection of vector bundles such that   
\begin{equation}\label{slope}\mu(E_1)\leq \mu(E_2)\leq \cdots \leq \mu(E_n)\leq \mu(E_1)+K_Z^2.\end{equation}
\end{itemize}
\end{prop}

\begin{proof}
The implication (i) $\implies$(ii) from very strong to cyclic strong is immediate. 
To prove (ii) $\implies$ (iii), note that $\bE$ cyclic strong implies that, up to a shift, $\bE$ is a collection of vector bundles by Lemma \ref{lem_bundles}. The inequalities \eqref{slope} then follow from the cylic strong condition together with Lemma \ref{lem_ex}, using that $\mu(E_i \otimes \omega_Z^{-1})=\mu(E_i)+K_Z^2$, and $K_Z^2>0$ since $Z$ is a del Pezzo surface.
Finally, it remains to prove (iii) $\implies$ (i). Let $\bE=(E_1,\dots, E_n)$ be a full strong exceptional collection of vector bundles satisfying
$\mu(E_1) \leq \dots \leq \mu(E_n)\leq \mu(E_1)+K_Z^2$. 
To show that $\bE$ is very strong, let $1 \leq i,j \leq n$ and $p \geq 0$. If $p=0$, then $\Hom^k(E_i,E_j)=0$ for $k \neq 0$ since $\bE$ is strong. If $p>0$, then $\mu(E_i) \leq \mu(E_1)+K_Z^2 \leq \mu(E_j)+K_Z^2 \leq \mu(E_j)+pK_Z^2=\mu(E_j \otimes \omega_Z^{-p})$. Hence, by Lemma \ref{lem_ex_0} and Lemma \ref{lem_ex}(i), we obtain that $\Hom^k(E_i, E_j \otimes \omega_Z^{-p})=0$ for all $k \neq 0$. 
Therefore, $\bE$ is very strong, completing the proof.
\end{proof}

The following result shows that very strong exceptional collections and geometric helices exist on all del Pezzo surfaces.

\begin{prop} \label{prop_existence}
Let $Z$ be a del Pezzo surface. Then, 
\begin{itemize}
    \item[(i)] there exists a geometric helix, and so very strong exceptional collections, on $Z$,
    \item[(ii)] there exists a very strong exceptional collection of line bundles on $Z$ if and only if $\mathrm{rk}\, NS(Z)\leq 7$, that is, if and only if $Z$ is not isomorphic to a blow-up of $\bP^2$ in seven or eight points.
\end{itemize}
\end{prop}

\begin{proof}
 Statement (i) follows from the proof of \cite[Proposition 7.3]{bergh2002non}, which in turn relies on \cite[Claim 6.5]{KO}.
Alternatively, it can be deduced from \cite[Example 8.6]{BS}, where explicit examples of very strong exceptional collections are constructed on all del Pezzo surfaces, using the existence of three-block exceptional collections on del Pezzo surfaces which are blow-ups of $\bP^2$ at $m \geq 3$ points \cite[Proposition 4.2]{KN}.  

Statement (ii) holds for cyclic strong exceptional collections by \cite[Theorems 5.13-5.14]{HP}, and so for very strong exceptional collections by Proposition \ref{prop_bundles}.
\end{proof}

\subsection{Geometric helices and tilting operations}
\label{sec_geometric_tilting}

Let $\bH=(E_i)_{i \in \bZ}$ be a geometric helix on a del Pezzo surface $Z$. Then, one can produce new geometric helices $\bH'=(E_i')_{i\in \bZ}$ using the following operations:
\begin{itemize}
    \item[(i)] \emph{Rotation} by $k \in \bZ$: $E_i'= E_{i+k}$ for all $i \in \bZ$.
    \item[(ii)] \emph{Shift} in the derived category by 
    $k \in \bZ$: $E_i'=E_i[k]$ for all $i \in \bZ$.
    \item[(iii)] \emph{Orthogonal reordering}: there exist $j, k \in \bZ$, $j<k$, such that $E_j, E_{j+1}, \dots, E_k$ are mutually orthogonal, as reviewed at the end of \S \ref{sec_ex_collections}: then, $E_i'=E_i$ for $i \notin \{j,k\}$, $E_j'=E_{k}$, $E_{k}'=E_j$.
    \item[(iv)] \emph{Tensor product by a line bundle} $L \in \Pic(Z)$: $E_i'=E_i \otimes L$ for all $i \in \bZ$.
\end{itemize}

In this section, we review the notion of tilting for geometric helices, which is a more interesting operation on the set of geometric helices introduced in \cite{BS}; see also \cite[\S 5.1]{Her}. We begin by recalling some properties of the dual exceptional collection.
By \cite[Lemma 2.5]{BS}, for every full exceptional collection $\bE=(E_1, \dots,E_n)$ in $D(Z)$, there exists a unique full exceptional collection 
$\bF=(F_n, \dots, F_1)$ in $D(Z)$, called the \emph{dual exceptional collection} to $\bE$, such that $\Hom^k(E_i, F_j)=\bC$ if $i=j$ and $k=0$, and 
$\Hom^k(E_i, F_j)=0$ else.
Note that the dual exceptional collection
$\bF=(F_n,\dots,F_1)$ to a very strong exceptional collection $\bE=(E_1,\dots,E_n)$
is never very strong. Indeed, by \cite[Corollary 2.10]{BS} and \cite[Lemma 2.5]{BS}, we have $F_1=E_1$ and $F_n=E_n\otimes \omega_Z[2]$. In particular, $\bF$ does not satisfy Lemma \ref{lem_bundles}, and so is not very strong by Proposition \ref{prop_bundles}.

\begin{prop} \label{prop_cyclic}
Let $Z$ be a del Pezzo surface and let  $\bE=(E_1,\dots, E_n)$ be  a very strong exceptional collection on $Z$, with dual exceptional collection  $\bF=(F_n, \dots, F_1)$. Let $p \in \bZ$ be as in Lemma \ref{lem_bundles} the unique integer such that all objects $E_i[p]$ lie in $\Coh(Z)$. Then, up to orthogonal reordering, 
\begin{itemize}
\item[(i)] there exist sheaves $F_n', \dots, F_1'$ and $n \geq b \geq c >1$ such that 
\[ (F_n[p], \dots, F_1[p])= (F_n'[2], \dots, F_b'[2], F_{b-1}'[1], \dots, F_{c}'[1], F_{c-1}', \dots, F_1')\,.\]
Moreover, $\mu(F_i') \geq \mu(F_j')$ for every $1 \leq i,j \leq n$ such that either 
$n \geq i>j\geq b$, or $b-1 \geq i>j \geq c$, or $c-1\geq i >j \geq 1$.
\item[(ii)] the vectors $(r(F_i), d(F_i)) \in \bZ^2 $ are cyclically ordered in $\bR^2$, that is,  there exist determinations $\theta_i$ of the arguments of the complex numbers $r(F_i)+\sqrt{-1} d(F_i)$ such that 
\[ \theta_1 \leq \theta_2 \leq \dots \leq \theta_{n-1} \leq \theta_n \leq \theta_1 + 2 \pi\,.\]
\end{itemize}
\end{prop}

\begin{proof}
Statement (i) follows from \cite[Lemma 8.3]{BS}, together with  $F_1=E_1$ and $F_n=E_n\otimes \omega_Z[2]$ to determine the range of shifts. For (ii), first note that the vectors $r(F_i)+\sqrt{-1}d(F_i)$ lie in the left half-plane (resp.\! the right half-plane) for $1 \leq i \leq c-1$ and $b \leq i \leq n$ (resp.\! $c \leq i \leq b-1$), and so it suffices to check the inequalities on arguments in each of these intervals.
Then, 
$ \theta_1 \leq \theta_2 \leq \dots \leq \theta_{n-1} \leq \theta_n$ follows from (i). Finally, since $F_1=E_1$,  $F_n=E_n\otimes \omega_Z[2]$, and $\mu(E_n)\leq \mu(E_1)+K_Z^2$ by Proposition \ref{prop_bundles}, we have 
\[ \mu(F_n')=\mu(E_n \otimes \omega_Z)=\mu(E_n) -K_Z^2 \leq \mu(E_1) =\mu(F_1') \,,\]
and so $\theta_n \leq \theta_1  + 2 \pi$.
\end{proof}

\begin{example} \label{example_cyclic}
By \cite[Example 2.8]{BS}, $\bE=(\cO, \cO(1,0), \cO(0,1), \cO(1,1))$ is a full strong exceptional collection on $Z=\bP^1 \times \bP^1$,
with dual collection $\bF=(\cO(-1,-1)[2], \cO(0,-1)[1], \cO(-1,0)[1], \cO)$.
We have 
\[ 0=\mu(\cO) \leq \mu(\cO(1,0))=2 \leq \mu(\cO(0,1))=2 \leq \mu(\cO(1,1))=4 \leq \mu(\cO)+K_Z^2=8 \,,\]
and so $\bE$ is very strong by Proposition \ref{prop_bundles}. The vectors $(r(F_i),d(F_i))$ are $(1, -4)$, $(-1,2)$, $(-1,2)$, $(1,0)$, which are indeed cyclically ordered in $\bR^2$ -- see Figure \ref{fig1}.
\end{example}

\begin{figure}
\begin{center}
\begin{tikzpicture}[scale=0.7]

    % Axes
    \draw[->] (-5,0) -- (5,0) node[right] {$r$};
    \draw[->] (0,-4) -- (0,3) node[above] {$d$};

    % Vectors
    \draw[->, thick] (0,0) -- (1,0) 
        node[below right] {$(r(F_1), d(F_1))$};

    \draw[->, thick] (0,0) -- (-1,2) 
        node[above left] {$(r(F_2), d(F_2))=(r(F_3), d(F_3))$};

    \draw[->, thick] (0,0) -- (1,-4) 
        node[below right] {$(r(F_4), d(F_4))$};
\end{tikzpicture}
\caption{The cyclically oriented vectors $(r(F_i), d(F_i))$ for the very strong exceptional collection $\bE$ on $Z=\bP^1 \times \bP^1$ considered in Example \ref{example_cyclic}.}
\label{fig1}
\end{center}
\end{figure}
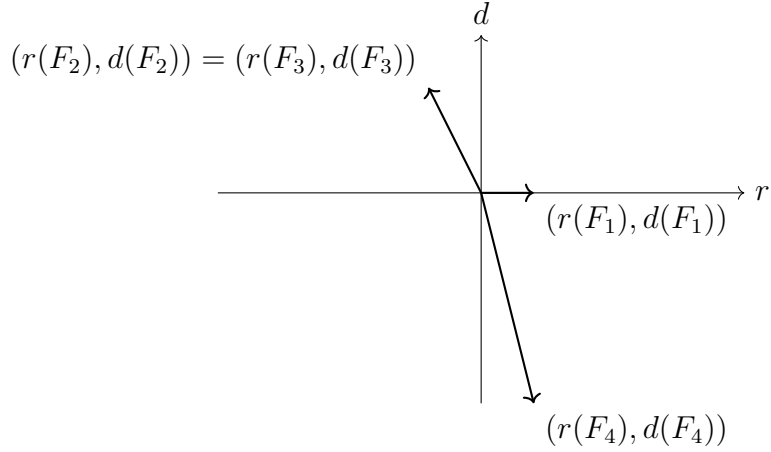

We first define tilting operations on good very strong exceptional collections.

\begin{definition} \label{def_good}
A very strong exceptional collection $\bE=(E_1,\dots,E_j, \dots, E_n)$ on a del Pezzo surface $Z$ is called \emph{good}
for the object $E_j$
if the dual exceptional collection $\bF = (F_n, \ldots,F_j,\ldots F_1)$ has the following property: 
\[k>j \implies \Hom^p(F_k,F_j)=0\text{ unless }p=1,\]
\[k<j  \implies \Hom^p(F_j,F_k)=0\text{ unless }p=1\,.\]
\end{definition}

We denote by $\langle -,-\rangle$ the pullback by 
$\psi=(r,d): K(Z) \rightarrow \bZ^2$ of 
$\langle-,-\rangle=\det(-,-)$ on $\bZ^2$.

\begin{lemma} \label{lem_good}
A very strong exceptional collection $\bE=(E_1,\dots,E_j, \dots, E_n)$ on a del Pezzo surface $Z$ is good
for the object $E_j$
if and only if the dual exceptional collection $\bF = (F_n, \ldots,F_j,\ldots F_1)$ has the following property: 
\[k>j \implies \langle [F_k], [F_j]\rangle \leq 0 \,,\]
\[k<j  \implies \langle [F_k], [F_j] \rangle \geq 0,.\]
\end{lemma}

\begin{proof}
For $k>j$, the pair $(F_k, F_j)$ is exceptional, and so $\chi(F_j,F_k)=0$, and $\chi(F_k,F_j)=\langle [F_k], [F_j]\rangle$ by Lemma \ref{lem_ex_0}. Moreover, by Lemma \ref{lem_ex_0}, at most one of the spaces $\Hom^p(F_k,F_j)$ is non-zero, and so the result follows. For $k <j$, the result follows in a similar way.
\end{proof} 

Recall that for any objects $E$, $F$ in $D(Z)$, the left and right mutations $L_F E$ and $R_F E$ are the unique objects in $D(Z)$ up to isomorphism for which there exist exact triangles
\begin{equation}\label{eq_left_mutation}
\Hom^\bullet(F,E) \otimes F \xrightarrow{ev} E \rightarrow L_F E\,, \end{equation}
and
\[ R_F E \rightarrow E \xrightarrow{coev} \Hom^\bullet(E,F)^\star \otimes F \,. \]

\begin{definition} \label{def_tilt_ex}
Let $Z$ be a del Pezzo surface, and $\bE=(E_1,\dots,E_j, \dots, E_n)$ be a very strong exceptional collection on $Z$ which is good for $E_j$. Then, the left 
\emph{tilt} of $\bE$ from $j$ to $1$ is
\begin{equation} \label{eq_left_tilt}
\mu_{j,1}^+(\bE)
=(L_{E_1} \dots L_{E_{j-1}} E_j[-1], E_1, \dots, E_{j-1}, E_{j+1}, \dots, E_n)\end{equation}
and the right \emph{tilt} of $\bE$ from $1$ to $j$ is
\begin{equation} \label{eq_right_tilt}
\mu_{1,j}^-(\bE)
= (E_2, \dots, E_{j}, R_{E_j} \dots R_{E_2} E_1[1], E_{j+1}, \dots, E_n)
\end{equation}
respectively -- see \cite[Remark 7.7]{BS}.
By \cite[Theorem 7.4]{BS}, both $\mu_{j,1}^+(\bE)$
and $\mu_{j,1}^-(\bE)$ are very strong exceptional collections.
\end{definition}

By construction, $\mu_{j,1}^+(\bE)=(E_1', \dots, E_n')$ is good for $E_j'=E_{j-1}$, and 
\[ \mu_{1,j}^- \circ \mu_{j,1}^+(\bE) = \bE \,.\]
Similarly, $\mu_{1,j}^-(\bE)=(E_1'', \dots, E_n'')$ is good for $E_j''= R_{E_j} \dots R_{E_2} E_1[1]$, and 
\[ \mu_{j,1}^+ \circ \mu_{1,j}^-(\bE) = \bE \,.\]

The following result determines the 
corresponding dual exceptional collections.

\begin{lemma} \label{lem_dual}
Let $\bE=(E_1,\dots,E_j, \dots, E_n)$ be a very strong exceptional collection on a del Pezzo surface $Z$ that is good for $E_j$.
Let $\bF=(F_n, \dots, F_j, \dots, F_1)$ be the dual exceptional collection of $\bE$.
Then, the dual exceptional collection of $\mu_{j,1}^+(\bE)$ is 
\[ (F_n, \dots, F_{j+1}, L_{F_j[-1]}F_{j-1}, \dots, L_{F_j[-1]} F_1, F_j[-1])\,,\]
and the dual exceptional collection of 
$\mu_{1,j}^-(\bE)$
is 
\[ (F_n, \dots, F_{j+1}, F_1[1], R_{F_1[1]} F_j, \dots, R_{F_1[1]} F_2) \,.\]
\end{lemma}

\begin{proof}
This follows by direct calculation using \cite[Lemma 2.5]{BS} which states that the dual exceptional collection $(F_n, \dots, F_1)$ of an exceptional collection $(E_1, \dots, E_n)$ is given by $F_j=L_{E_1} \dots L_{E_{j-1}}E_j$ for all $1 \leq j \leq n$.
\end{proof}

To define tilting operations on geometric helices, we will need to find good threads in helices. To do this, we will use the following result. 

\begin{lemma} \label{lem_rotation}
Let $\bE=(E_1, \dots, E_n)$ be a thread with dual collection $\bF=(F_n, \dots, F_1)$ in a helix $\bH=(E_i)_{i\in \bZ}$ on a del Pezzo surface $Z$. 
Consider the neighboring thread $(E_0, \dots, E_{n-1})$ and let $\bF'=(F_{n-1}', \dots, F_0')$
be its dual collection.
Then the classes in $K(Z)$ of the objects of  $\bF'$ 
are given by
$[F_i']= T_{[F_n]}([F_i])$ for all  $1 \leq i \leq n-1$   and  $[F_0']=T_{[F_n]}([F_n])=[F_n]$,
where, for any $\alpha \in K(Z)$,  
the linear map $T_{\alpha}: K(Z) \rightarrow K(Z)$ is defined by
\begin{equation} \label{eq_T_alpha}
    T_\alpha(\beta)=\beta + \langle \beta, \alpha\rangle \alpha \,.
\end{equation}
\end{lemma}

\begin{proof}
By \cite[Lemma 3.4]{BS}, we have \begin{equation*}
   \bF' = (L_{F_n} F_{n-1}, \dots, L_{F_n} F_1, F_n[-2]) \,,
\end{equation*}
and so $[F_0']=[F_n]$, and for all $1 \leq i \leq n-1$, 
\[[F_i']=[F_i] - \chi(F_n, F_i) [F_n] \,.\]
For $n>i$, the pair $(F_n, F_i)$ is exceptional, so $\chi(F_i, F_n)=0$, and so 
\[ \chi(F_n, F_i)= \langle [F_n], [F_i]\rangle=- \langle [F_i], [F_n] \rangle\]
by Lemma \ref{lem_ex_0}. Hence, we obtain 
$[F_i']=[F_i]+\langle [F_i], [F_n] \rangle [F_n]$, and this 
completes the proof.
\end{proof}

The proof of the following proposition is a reformulation of the argument given in \cite[\S 8]{BS} -- see also \cite[\S 5.1]{Her}.

\begin{prop} \label{prop_good}
Let $\bH= (E_i)_{i\in \bZ}$ 
be a geometric helix on a del Pezzo surface $Z$ and choose an element $E_j$ of $\bH$. 
Then there is a thread $\bE= (E_{i+1},\ldots,E_j,\ldots, E_{i+n})$ of $\bH$ which is good for $E_j$.
\end{prop}

\begin{proof}
Start with any thread $\bE=(E_{i+1}, \dots, E_{i+n})$ of $\bH$ containing $E_j$. Let $\bF=(F_{i+n}, \dots, F_{i+1})$ be the dual exceptional collection. By Proposition \ref{prop_cyclic}(ii), the vectors $(r(F_{i+k}), d(F_{i+k}))_{0 \leq k \leq n}$ are cyclically ordered in $\bR^2$.
By Lemma \ref{lem_rotation}, rotating $\bE$ acts on the set of vectors $(r(F_{i+k}), d(F_{i+k}))_{1 \leq k \leq n}$ by an element of $SL(2,\bZ)$ and by cyclic relabeling. 
Therefore, up to replacing $\bE$ by a different thread in $\bH$, one can assume that $k >j$ implies 
$\langle F_k, F_j\rangle \leq 0$ and that $k <j$ implies 
$\langle F_k, F_j\rangle \geq 0$. Then $\bE$ is good for $E_j$ by Lemma \ref{lem_ex}.
\end{proof}

We can finally define tilting operations on geometric helices.

\begin{definition} \label{def_tilitng_helices}
Let $\bH=(E_i)_{i \in \bZ}$ be a geometric helix on a del Pezzo surface $Z$. For every $j\in \bZ$, there exists a thread $\bE$ of $\bH$ that is good for $E_j$ by Proposition 
\ref{prop_good}. Then, we define the left and right \emph{tilts} $\mu_{j,1}^+(\bH)$
and $\mu_{1,j}^-(\bH)$
of $\bH$ at $j$  as the helices generated by the very strong exceptional collection $\mu_{j,1}^+(\bE)$ and $\mu_{1,j}^-(\bE)$ respectively  defined as in 
\eqref{eq_left_tilt}-\eqref{eq_right_tilt}.
By \cite[Theorem 7.4]{BS}, the helices $\mu_{j,1}^+(\bH)$ and $\mu_{1,j}^-(\bH)$ are geometric. 
Moreover, for every $j \in \bZ$, the tilted helices $\mu_{j,1}^+(\bH)$
and $\mu_{1,j}^-(\bH)$
are uniquely determined up to orthogonal reordering of objects of $\bH$, and only depend on $j$ modulo $n$.
\end{definition}

\begin{lemma} \label{lem_shift}
    Let $\bH=(E_i)_{i \in \bZ}$ be a geometric helix on a del Pezzo surface $Z$, 
and let $p \in \bZ$ be as in Lemma
\ref{lem_bundles} the unique integer such that all objects $E_i[p]$ lie in $\Coh(Z)$. Then, for every $j \in \bZ$, and for every object $E_i'$ of the tilted helices  $\mu_{j,1}^+(\bH)$
and $\mu_{1,j}^-(\bH)$, the object $E_i'[p]$ lies in $\Coh(Z)$.
\end{lemma}

\begin{proof}
Since $\mu_{j,1}^+(\bH)$ is a geometric helix, there exists a unique $p'\in \bZ$ such that all the objects $E_i'$ of $\mu_{j,1}^+(\bH)$ satisfy $E_i'[p']\in \Coh(Z)$ by Lemma \ref{lem_bundles}
and Proposition \ref{prop_bundles}. 
By Definition \ref{def_tilt_ex}, there exist objects of $\mu_{j,1}^+(\bH)$ contained in $\bH$, and so $p'=p$. The result follows similarly for $\mu_{1,j}^-(\bH)$.  
\end{proof}

\section{Classification of geometric helices on del Pezzo surfaces}
\label{sec_final}

In this section, we first explain how to associate a seed of $q$-Painlevé type to every very strong exceptional collection on a del Pezzo surface. We then prove that tilting of geometric helices generated by very strong exceptional collections is compatible with seed mutations. Next, we show that the affine Weyl group acts on the set of very strong exceptional collections via sequences of tilting operations and orthogonal reorderings. Finally, we prove the main result of this paper, Theorem~\ref{thm_main}, referred to as Theorem \ref{thm_main_intro} in the Introduction, which shows that all geometric helices on a del Pezzo surface are related to each other by a sequence of elementary operations.

\subsection{Very strong exceptional collections and seeds of $q$-Painlev\'e type}

In this section, we first explain how to construct seeds from very strong exceptional collections on a del Pezzo surface $Z$. We then show in Theorem \ref{thm_q_painleve} that these seeds are of $q$-Painlev\'e type in the sense reviewed in \S \ref{sec_q_painleve}. 
We will apply the notation and results of $\S \ref{sec_seeds_mutations} $ with $N=K(Z)$ and 
\[ \psi=(r,d): N= K(Z) \rightarrow \bZ^2\,.\]
Note that $\psi$ has indeed a finite cokernel by the following elementary result.

\begin{lemma}
If the del Pezzo surface $Z$ is not isomorphic to $\bP^2$ or $\bP^1 \times \bP^1$, then the map $\psi=(r,d):K(Z) \rightarrow \bZ^2$ is surjective.
 If $Z= \bP^2$ (resp.\! $Z=\bP^1 \times \bP^1$), then the cokernel of $\psi$ is a finite group of order two (resp.\! three).
\end{lemma}

\begin{proof}
    Since $\psi([\cO_Z])=(1,0)$, the vector $(1,0)$ always lies in the image of $K(Z)$. If $Z$ is not isomorphic to $\bP^2$ or $\bP^1 \times \bP^1$, then $Z$ contains an exceptional curve $E$. In this case, $\psi([\cO_E])=(0,1)$, so $(0,1)$ is also in the image of $K(Z)$, and $\psi$ is surjective.
   If $Z=\bP^2$ (resp. $\bP^1 \times \bP^1$), we have $K_Z=\cO_{\bP^2}(-3)$ (resp.\! $\cO_{\bP^1 \times \bP^1}(-2,-2)$) and so the result follows.
\end{proof}

\begin{definition} \label{def_seeds_from_ex}
Let $\bE=(E_1, \dots, E_n)$ be a very strong exceptional collection on a del Pezzo surface $Z$. Let $\bF=(F_n, \dots, F_1)$ be the dual exceptional collection. 
The \emph{seed} associated to $\bE$ is the multiset $\{[F_i]\}_{1 \leq i \leq n}$ of classes in $N=K(Z)$:
\[ \bs(\bE):= \{[F_i]\}_{1 \leq i\leq n}\,.\]
The \emph{cyclically ordered} seed $\tilde{\bs}(\bE)$ associated to $\bE$ is the seed $\bs(\bE)$ together with the cyclic ordering induced by the ordering of $\bE$. 
\end{definition}

\begin{remark}
    By Lemma \ref{lem_primitive}, the vectors $\psi([F_i])$ are non-zero and primitive in $\bZ^2$. Moreover, since $\bE$ and $\bF$ are full exceptional collections, the classes $([F_i])_{1 \leq i \leq n}$ form a basis of $K(Z)$. Hence, $\bs(\bE)$ is indeed a seed in the sense of \S\ref{sec_seeds}. By Proposition \ref{prop_cyclic}, the vectors $\psi([F_i])$ are cyclically ordered in $\bR^2$. Therefore,  $\tilde{\bs}(\bE)$ is indeed a cyclically ordered seed in the sense of \eqref{eq_cyclic} in \S\ref{sec_seeds}.
\end{remark}

We defined in \S\ref{sec_intersection} a bilinear form $\chi_{\tilde{\bs}}(-,-)$ on $N$ for every cyclically ordered seed -- see \eqref{eq_bilinear}. On the other hand, as reviewed in \S\ref{sec_dP}, $K(Z)$ carries a natural Euler form $\chi(-,-)$. 
The result below shows that these two bilinear form coincide when $\tilde{\bs}(\bE)$ is the cyclically ordered seed associated with a very strong exceptional collection.

\begin{lemma} \label{lem_isometry1}
Let $\bE=(E_1, \dots, E_n)$ be a very strong exceptional collection on a del Pezzo surfaces $Z$. Then, the bilinear forms $\chi_{\tilde{\bs}(\bE)}(-,-)$ and $\chi(-,-)$ on $N=K(Z)$ coincide:
\[ \chi_{\tilde{\bs}(\bE)}(-,-) = \chi(-,-) \,.\]
\end{lemma}

\begin{proof}
Let $\bF=(F_n, \dots, F_1)$ be the dual exceptional collection.
Since $([F_i])_{1 \leq i \leq n}$  forms a basis of $K(Z)$, it suffices to prove that $\chi_{\bs(\bE)}([F_i], [F_j]) = \chi(F_i, F_j)$ for all $1 \leq i, j \leq n$. If $i=j$, then $\chi(F_i,F_i)=1$ because $F_i$ is an exceptional object. 
If $i>j$, the pair $(F_i, F_j)$ is exceptional and so 
$\chi(F_j, F_i)=0$, and $\chi(F_i, F_j)=\langle [F_i], [F_j]\rangle $ by Lemma \ref{lem_ex_0}. The result then follows by comparison with the definition of 
$\chi_{\tilde{\bs}(\bE)}(-,-)$ given in \eqref{eq_bilinear}.
\end{proof}

As a result, we obtain the following description of the intersection form $(-,-)_{\mathscr{S}(\bE)}$ associated by \S\ref{sec_intersection} to the mutation equivalence class $\mathscr{S}(\bE)$ of the seed $\bs(\bE)$.

\begin{lemma} \label{lem_isometry2}
Let $\bE=(E_1, \dots, E_n)$ be a very strong exceptional collection on a del Pezzo surface $Z$. Then, the intersection form $(-,-)_{\mathscr{S}(\bE)}$ on $K=\mathrm{Ker}(\psi)$ coincides with minus the restriction of the Euler form: for every $\alpha, \beta  \in K=\mathrm{Ker}(\psi)$, we have 
\[ (\alpha, \beta)_{\mathscr{S}(\bE)} = - \chi(\alpha, \beta) \,.\]
\end{lemma}

\begin{proof}
By Lemma \ref{lem_brackets}, we have $(\alpha, \beta)_{\mathscr{S}(\bE)}=-\chi_{\tilde{\bs}(\bE)}(\alpha, \beta)$ for all $\alpha, \beta \in K$. By Lemma \ref{lem_isometry1}, we also have $\chi_{\tilde{\bs}(\bE)}(\alpha, \beta) = \chi(\alpha, \beta)$ for all $\alpha$, $\beta \in K(Z)$, and so the result follows.
\end{proof}

\begin{theorem} \label{thm_q_painleve}
   Let $\bE=(E_1, \dots, E_n)$ be a very strong exceptional collection on a del Pezzo surface $Z$. Then, the mutation equivalence class of seeds $\mathscr{S}(\bE)$ is of $q$-Painlev\'e type.
\end{theorem}

\begin{proof}
By Lemma \ref{lem_isometry2}, we have $(K, (-,-)_{\mathscr{S}(\bE)}) \simeq (\mathrm{Ker}(\psi), -\chi(-,-)|_{\mathrm{Ker}(\psi)})$.
On the other hand, we have $\mathrm{Ker}(\psi) =K_Z^{\perp} \oplus \bZ \delta$
by \eqref{eq_ker_psi}. 
By Lemma \ref{lem_sym}(iii), $\bZ \delta$ is the kernel of $-\chi(-,-)|_{\mathrm{Ker}(\psi)}$ and $-\chi(-,-)|_{K_Z^\perp}$ is the intersection form on $K_Z^\perp$, which is negative definite by the Hodge index theorem. Thus, the intersection form $(-,-)_{\mathscr{S}(\bE)}$ is negative semi-definite but not negative definite. 
Hence, the mutation equivalence class of seeds $\mathscr{S}(\bE)$ is of 
$q$-Painlev\'e type by Definition \ref{def_q_painleve}.
\end{proof}

By \eqref{eq_delta_S} in \S \ref{sec_q_painleve}, for every seed $\bs=\{e_i\}_{1 \leq i \leq n}$ whose mutation equivalence class $\mathscr{S}$ is of $q$-Painlev\'e type, there exists a unique primitive element $\delta_{\mathscr{S}} = \sum_{i=1}^n c_{\bs, i} e_i$ with $c_{\bs,i}\in \bZ_{>0}$ for all $1 \leq i \leq n$ and $(\delta_{\mathscr{S}}, \delta_{\mathscr{S}})_{\mathscr{S}}=0$.

\begin{lemma} \label{lem_rank}
  Let $\bE=(E_1, \dots, E_n)$ be a very strong exceptional collection of bundles on a del Pezzo surface $Z$. Then, we have $\delta = \delta_{\mathscr{S}(\bE)}$
  and $r(E_i) = c_{\bs(\bE),i}$ for all $1 \leq i \leq n$.
\end{lemma}

\begin{proof}
Let $\bF=(F_n, \dots, F_1)$ be the dual exceptional collection. 
Write $\delta =[\cO_x]= \sum_{i=1}^n a_i [F_i]$ with $a_i \in \bZ$. Then, applying $\chi(E_j, -)$ and using the duality between $\bE$ and $\bF$, we obtain
\[ r(E_j) = \chi(E_j, \cO_x)=\chi(E_j, \delta)=\sum_{i=1}^n a_i \chi(E_j, F_i)
=\sum_{i=1}^n a_i \delta_{ji}=a_j \,.\]
Hence $\delta=\sum_{j=1}^n r(E_j)[F_j]$.
Moreover, we have $r(E_j)\in \bZ_{>0}$ for all $1 \leq k \leq n$ since the $E_j$'s are bundles by assumption, 
and $(\delta, \delta)_{\mathscr{S}(\bE)}=-\chi(\delta, \delta)=0$ by Lemma \ref{lem_isometry2}. Since in addition $\delta$ is a primitive element of $K(Z)$, it follows that $\delta=\delta_{\mathscr{S}(\bE)}$ and $r(E_j) = c_{\bs(\bE),j}$ for all $1 \leq j \leq n$ by uniqueness of $\delta_{\mathscr{S}(\bE)}$ and $(c_{\bs(\bE),j})_{1 \leq j \leq n}$.
\end{proof}

\begin{remark} \label{rem_del_pezzo_painleve}
    As reviewed in \S \ref{sec_q_painleve}, when $\mathscr{S}$ is of $q$-Painlev\'e type, the lattice $(K,(-,-)_{\mathscr{S}})$ is isometric to a lattice of the form $\bZ \delta_{\mathscr{S}} \oplus Q_{\mathscr{S}}$, where $Q_{\mathscr{S}}$ is negative definite. By the proof of Theorem \ref{thm_q_painleve}, we obtain an isomorphism $Q_{\mathscr{S}(\bE)} \simeq K_Z^\perp$.
\end{remark}

\subsection{Tilting operations and seed mutations}
 In this section, we describe the compatibility between the seed mutations introduced in 
 \S \ref{sec_seeds} and the tilting operations of geometric helices reviewed in \S\ref{sec_geometric_tilting}. We begin by explaining  the compatibility with tilting operations of good very strong exceptional collections in the sense of Definition \ref{def_good}.

\begin{lemma} \label{lem_good_tilt}
Let $\bE=(E_1, \dots, E_n)$ be a very strong exceptional collection on a del Pezzo surface $Z$. Assume that $\bE$ is good for the object $E_j$ for some $1 \leq j \leq n$. Then, tilts of $\bE$ commute with seed mutations of $\bs(\bE)$, that is, 
\[ \bs(\mu_{j,1}^+(\bE)) = \mu_j^+(\bs(\bE))\text{      and     } \bs(\mu_{1,j}^-(\bE)) = \mu_1^-(\bs(\bE))\,.\]
\end{lemma}

\begin{proof}
By Lemma \ref{lem_dual}, the classes
$[F_i']$, $1 \leq i \leq n$, 
of the objects of the 
exceptional collection dual to $\mu_{j,1}^+(\bE)$ are given by $[F_j']=-[F_j]$, and for every $i \neq j$, $[F_i']=[F_i] - \chi(F_j, F_i)[F_j]$. For $j \geq i$, the pair $(F_j, F_i)$ is exceptional, so $\chi(F_i, F_j)=0$ and $\chi(F_j,F_i)=\langle [F_j], [F_i]\rangle=-
\langle [F_i], [F_j] \rangle \leq 0$ by Lemmas \ref{lem_ex_0}-\ref{lem_ex}. Hence, for all $i \neq j$, we obtain 
\[ [F_i']=[F_i]+[\langle [F_i], [F_j]\rangle]_+ [F_j]\,,\]
and so the result follows by comparison with the formula \eqref{eq_seed_mutation} describing seed mutations. The result for $\mu_{1,j}^-(\bE)$ is obtained analogously.
\end{proof}

In the following result, we denote by $\mathrm{Aut}_K(K(Z),\langle-,-\rangle)$ the group of automorphisms of $K(Z)$ preserving the skew-symmetric bilinear form $\langle -,-\rangle$ and acting trivially on $K=\mathrm{Ker}(\langle -,- \rangle)$.

\begin{prop} \label{prop_tilting}
    Let $\bH=(E_i)_{i\in \bZ}$ be a geometric helix on a del Pezzo surface $Z$. Let $\bE$ be a thread of $\bH$ containing the object $E_j$ for some $j \in \bZ$. 
    Then, there exist threads $\bE'_+$ and $\bE'_-$ in the tilted helices $\mu_{j,1}^+(\bH)$ and $\mu_{1,j}^-(\bE)$ respectively, and $f \in \mathrm{Aut}_K(K(Z),\langle-,-\rangle)$ such that
    \[ \bs(\bE'_+)=f(\mu_j^+(\bs(\bE)))\text{          and       }  \bs(\bE'_-)= f (\mu_1^-(\bs(\bE)))\,.\]
\end{prop}

\begin{proof}
By Proposition \ref{prop_good}, there exists a thread $\widetilde{\bE}$ of $\bH$ which is good with respect to $E_j$. Since $\widetilde{\bE}$ is obtained from $\bE$ by rotation in $\bH$, there exists $f \in \mathrm{Aut}(K(Z))$ such that $\bs(\widetilde{\bE})=f(\bs(\bE))$ by Lemma \ref{lem_rotation}. More precisely, $f$ can be chosen as a composition of automorphisms of the form $T_\alpha$ as in \eqref{eq_T_alpha}, each of which preserves $\langle -,- \rangle$ and acts trivially on $K=\mathrm{Ker}(\langle -,- \rangle)$. Consequently, $f$ is contained in the group  $\mathrm{Aut}_K(K(Z),\langle-,-\rangle)$.
By Lemma \ref{lem_good_tilt}, setting $\bE'_+ := \mu_{j,1}^+(\widetilde{\bE})$
and $\bE'_- := \mu_{1,j}^-(\widetilde{\bE})$ , we obtain $\bs(\bE'_+)= \mu_{j}^+(\bs(\widetilde{\bE}))$ and 
$\bs(\bE'_-)= \mu_{1}^-(\bs(\widetilde{\bE}))$. 
Moreover, since $f$ preserves $\langle-,-\rangle$, its action on seeds commutes with the seed mutations given by 
\eqref{eq_seed_mutation}.
Therefore, we conclude that
\[ \bs(\bE'_+)=\mu_{j}^+(\bs(\widetilde{\bE})) = \mu_j^+(f(\bs(\bE)))=f(\mu_j^+(\bs(\bE)))\,,\]
and similarly $\bs(\bE'_-)= f (\mu_1^-(\bs(\bE)))$.
\end{proof}

\subsection{Affine Weyl groups and tilting operations}
\label{sec_weyl_affine}
We introduced in Definitions 
\ref{def_affine}-\ref{def_orthogonal} 
the affine Weyl group $\widehat{W}(Z)$ and the orthogonal $O(Z)$ of a del Pezzo surface $Z$, which both act by isometries on the Grothendieck group $K(Z)$. In this section, we first show that $O(Z)$ admits a natural action on the set of geometric helices. We then prove that, for any very strong exceptional collection $\bE$ on $Z$, the group $\widehat{W}(Z)$ coincides with the Weyl group $W_{\mathscr{S}(\bE)}$ associated in \S\ref{sec_weyl} to the mutation equivalence class $\mathscr{S}(\bE)$ of the seed $\bs(\bE)$. As an application, Theorem \ref{thm_weyl_action} shows that the action of $\widehat{W}(Z) \subset O(Z)$ on geometric helices can be described in terms of tilting operations and orthogonal reorderings.

\begin{lemma} \label{lem_weyl_action}
Let $\bH=(E_i)_{i \in \bZ}$ be a geometric helix on a del Pezzo surface $Z$, 
and let $p \in \bZ$ be as in Lemma \ref{lem_bundles} the unique integer such that all objects $E_i[p]$ lie in $\Coh(Z)$. For every $w \in O(Z)$, there exists a unique geometric helix $w \cdot \bH :=(E_i')_{i \in \bZ}$ such that $[E_i']=w([E_i])$ for all $i\in \bZ$, and all objects $E_i'[p]$ lie in $\Coh(Z)$.

\end{lemma}

\begin{proof}
Up to a global shift, we may assume that $p=0$, so that $\bH$ is a geometric helix of sheaves. By \cite[Corollary 2.5]{Gor}, exceptional sheaves are uniquely determined by their classes in $K(Z)$; this immediately yields the uniqueness of $w \cdot \bH$. It therefore remains to establish existence.
By Proposition \ref{prop_weyl_affine_gen}, we have a decomposition $O(Z)=W(Z)\ltimes \Pic^0(Z)$. Since $\Pic^0(Z)$ acts on geometric helices by tensor product, it suffices to construct the action for elements $w \in W(Z)$.
Let $\bE=(E_1,\dots, E_n)$ be a very strong exceptional collection arising as a thread of $\bH$. 
By \cite[Proposition 5.3]{KN}, there exists a full exceptional collection $\bE'=(E_1', \dots, E_n')$ such that $[E_i']=w([E_i])$ for all $1 \leq i \leq n$. In particular, we have $r(E_i')=r(E_i)$ and $d(E_i')=d(E_i)$, hence $\mu(E_i')=\mu(E_i)$ for all $1 \leq i\leq n$.
It then follows from Proposition \ref{prop_bundles} that $\bE'$ is again a very strong exceptional collection. We therefore define $w \cdot \bE$ to be the geometric helix generated by $\bE'$.
\end{proof}

\begin{remark} \label{rem_KN}
The proof of \cite[Proposition 5.3]{KN} given by Karpov–Nogin, which establishes the existence of the action of $W(Z)$ on the set of full exceptional collections of sheaves on a del Pezzo surface $Z$, relies on the fact that all 
full exceptional collections are related by mutations \cite[Theorem 7.7]{KO}.
An alternative, more elementary and geometric argument can be obtained by interpreting $W(Z)$ as a monodromy group of the moduli space of del Pezzo surfaces, together with the observation that exceptional collections on such surfaces deform in smooth families. Indeed, the reflection with respect to a $(-2)$-curve class $\alpha \in K_Z^\perp$ is exactly the monodromy around a small loop in the moduli space of del Pezzo surfaces encircling the locus where $\alpha$ becomes effective.
\end{remark}

\begin{theorem} \label{thm_weyl_affine}
    Let $\bE$ be a very strong exceptional collection on a del Pezzo surface $Z$. Then, the Weyl group $W_{\mathscr{S}(\bE)}$ coincides with the affine Weyl group $\widehat{W}(Z)$, and the set of roots $\Phi_{\mathscr{S}(\bE)}$ coincides with the set of affine roots $\widehat{\Phi}_Z$:
    \[ W_{\mathscr{S}(\bE)} = \widehat{W}(Z) \,\,\,\,\text{and}\,\,\,\, \Phi_{\mathscr{S}(\bE)}=\widehat{\Phi}_Z \,.\]
\end{theorem}

\begin{proof}
Let $(Y,D) \rightarrow (\overline{Y}, \overline{D})$ be a framed toric model for $\bs(\bE)$ as in \S \ref{sec_log_CY}. By Theorem \ref{thm_q_painleve}, the seed $\bs(\bE)$ is of $q$-Painlev\'e type. In particular, 
the intersection form $(-,-)_{\mathscr{S}(\bE)}$ is not negative definite. Hence, the hypotheses of  \cite[Theorem 9.14]{friedman2015geometry} 
are satisfied, and so the set of roots $\Phi$ of $(Y,D)$ is the set of classes $\alpha \in \Lambda_{(Y,D)}$ such that $\alpha \cdot \alpha=-2$.
Since $\Lambda_{(Y,D)} \simeq K_Z^\perp$, it follows that $\Phi$ coincides with the set $\widehat{\Phi}_Z$ of affine roots of $Z$. On the other hand, by Theorem \ref{thm_weyl}, $\Phi$ is also precisely the set of roots 
$\Phi_{\mathscr{S}(\bE)}$ of $\mathscr{S}(\bE)$, and so 
\begin{equation} \label{eq_Phi_Phi}
\widehat{\Phi}_Z = \Phi = \Phi_{\mathscr{S}(\bE)} \,.\end{equation}
By Lemma \ref{lem_tijk} together with Lemma \ref{lem_isometry1}, the Weyl group $W_{\mathscr{S}(\bE)}$ is generated by the reflections $\beta \mapsto \beta - \chi(\beta,\alpha)\alpha$ with $\alpha \in \Phi_{\mathscr{S}(\bE)}$. On the other hand, 
by Definition \ref{def_affine}, the affine Weyl group $\widehat{W}(Z)$ is generated by the reflections $\beta \mapsto \beta - \chi(\beta,\alpha)\alpha$ with $\alpha \in \widehat{\Phi}_Z$. Therefore, the equality $\Phi_{\mathscr{S}(\bE)}= \widehat{\Phi}_Z$ implies that $W_{\mathscr{S}(\bE)}= \widehat{W}(Z)$. 
\end{proof}

\begin{remark}
In \cite[Appendix A]{Mizuno}, building on the previous work \cite{BGM}, Mizuno constructs by means  of an explicit case by case analysis an embedding of $\widehat{W}(Z)$, regarded as a group of isometries of $K$,
into the group $W_{\mathscr{S}(\bE)}$ introduced in \S\ref{sec_weyl}.
Consequently, the identification $\widehat{W}(Z)=W_{\mathscr{S}(\bE)}$ given by Theorem \ref{thm_weyl_affine} 
admits an explicit description
using the formulas given in \cite[Appendix A]{Mizuno}.
Conversely, the results of the present paper furnish a geometric and conceptual explanation for the existence of the embedding 
constructed in \cite[Appendix A]{Mizuno}.
\end{remark}

\begin{lemma} \label{lem_orthogonal}
Let $\bE=(E_1, \dots, E_n)$ be a full exceptional collection on a del Pezzo surface $Z$, and $\bF=(F_n, \dots, F_1)$ be the dual exceptional collection. Then, consecutive objects $E_{j+1}, \dots, E_{j+k}$ are mutually orthogonal if and only if $F_{j+k}, \dots, F_{j+1}$ are mutually orthogonal.
\end{lemma}

\begin{proof} We first prove the ``only if" direction. 
By \cite[Lemma 2.5]{BS}, for every $1 \leq \ell \leq k$, we have $F_{j+\ell}=L_{E_1}\dots L_{E_{j+\ell-1}}(E_{j+\ell})$. As $E_{j+1}, \dots, E_{j+k}$ are mutually orthogonal, we have $L_{E_{j+m}}(E_{j+\ell})=E_{j+\ell}$ for all $1 \leq m \leq k$, and so
\begin{equation} \label{eq_FE}
F_{j+\ell}=L_{E_1}\dots L_{E_j} (E_{j+\ell}) \,.\end{equation}
For every objects $E$ and $A$ such that $\Hom^i(A,E)=0$ for all $i\in \bZ$, we have $\chi(E,A)=\langle \psi([E]), \psi([A])\rangle$ by Lemma \ref{lem_sym}(i), and so, using the definition of $L_E(A)$ in \eqref{eq_left_mutation}, 
\begin{equation} \label{eq_psiL}
\psi([L_E A])=\psi([A])-\langle \psi([E]), \psi([A])\rangle \psi([E])\,.\end{equation}
As $E_{j+1}, \dots, E_{j+k}$ are mutually orthogonal, the vectors $\psi([E_{j+1}]), \dots, \psi([E_{j+k}])$
are collinear in $\bZ^2$ by Lemma \ref{lem_ex}(iii). Since $\bE$ is an exceptional collection, one can repeatedly apply \eqref{eq_psiL} to deduce from \eqref{eq_FE} that the vectors $\psi([F_{j+1}]), \dots, \psi([F_{j+k}])$ are also collinear in $\bZ^2$, and so $F_{j+k}, \dots, F_{j+1}$ are mutually orthogonal by Lemma \ref{lem_ex}(iii). The ``if" direction is proved similarly.
\end{proof}

\begin{theorem} \label{thm_weyl_action}
Let $\bH=(E_i)_{i\in \bZ}$ be a geometric helix on a del Pezzo surface $Z$. For any element $w \in \widehat{W}(Z)$,
the geometric helix $w \cdot \bH$ is related to $\bH$ by a sequence of tilting operations and orthogonal reorderings. 
\end{theorem}

\begin{proof}
Let $\bE$ be any thread of $\bH$.
By Theorem \ref{thm_weyl_affine}, we have $\widehat{W}(Z)=W_{\mathscr{S}(\bE)}$.
By Definition \ref{def_weyl_S}, the group $W_{\mathscr{S}(\bE)}$
is generated by the elements $t_{jk}^{\bs}$ as in \eqref{eq_t_jk}
indexed by seeds 
$\bs =\{e_i\}_{1 \leq i\leq n} \in \mathscr{S}(\bE)$ and distinct elements $e_j$, $e_k$, $j<k$, 
of $\bs$
such that $\psi(e_j)=\psi(e_k)$. 
Hence, it suffices to prove
Theorem \ref{thm_weyl_action} for $w$ equal to one of these generators $t_{jk}^{\bs}$. Since both seeds $\mathbf{s}(\bE)$ and $\bs$ are in $\mathscr{S}(\bE)$, there exists a sequence $\mu$ of seed mutations such that $\bs= \mu(\bs(\bE))$.

By the correspondence between seed mutations and tilting operations of geometric helices
given in Proposition \ref{prop_tilting}, there exists a very strong exceptional collection  $\widetilde{\mu}(\bE)$, obtained from $\bE$ by a sequence $\widetilde{\mu}$ of helix rotations and tilting operations, such that $\bs(\widetilde{\mu}(\bE))=f(\mu(\bs(\bE)))= f(\bs)$ for some 
$f \in \mathrm{Aut}_K(K(Z), \langle-,-\rangle)$. 
Write $\widetilde{\mu}(\bE)=(\widetilde{E}_1, \dots, \widetilde{E}_n)$, and consider the dual exceptional collection 
$(\widetilde{F}_n, \dots, \widetilde{F}_1)$.
Since $\bs(\widetilde{\mu}(\bE))=f(\bs)$, we have $[\widetilde{F}_j]=f(e_j)$ and $[\widetilde{F}_k]=f(e_k)$. As $\psi(e_j)=\psi(e_k)$, we have $e_j-e_k \in K$, so $f(e_j)-f(e_k)=f(e_j-e_k)=e_j-e_k \in K$, and so $\psi(f(e_j))=\psi(f(e_k))$.
Moreover, the vectors $\psi([\widetilde{F}_i])$ for $1 \leq i \leq n$ are cyclically ordered in $\bR^2$ by Proposition \ref{prop_cyclic}(ii), and so
the objects
$\widetilde{F}_k,\widetilde{F}_{k-1}, \dots, \widetilde{F}_j$ are mutually orthogonal by Lemma \ref{lem_ex}(iii).

Hence, by Lemma \ref{lem_orthogonal}, the objects $\widetilde{E}_j, \widetilde{E}_{j+1}, \dots, \widetilde{E}_k$ are also mutually orthogonal,  and we denote by $(j \,k) \circ \widetilde{\mu}(\bE)$ the very strong exceptional collection obtained from  $\widetilde{\mu}(\bE)$ by the orthogonal reordering defined by the permutation exchanging $\widetilde{E}_j$ and $\widetilde{E}_k$. 
Since the images by $\psi=(r,d): K(Z) \rightarrow \bZ^2$ of 
$\bs((j \,k) \circ \widetilde{\mu}(\bE))$  and $\bs(\widetilde{\mu}(\bE))$ coincide up to relabeling of parallel vectors, there exists a sequence $\widetilde{\mu}^{-1}$ of inverses of the helix rotations and tilting operations in $\widetilde{\mu}$. Set $\bE' := \widetilde{\mu}^{-1} \circ (j \,k) \circ \widetilde{\mu}(\bE)$. Then, we have 
\[ \bs(\bE') = \mu^{-1} \circ f^{-1} \circ t_{jk}^\bs  \circ f (\mu(\bs(\bE))) \,.\]
Since $t_{jk}^\bs$ is given by Lemma \ref{lem_tijk} and $f$ acts trivially on $K$, it follows that $t_{jk}^\bs$ and $f$ commute. Therefore, we obtain 
\[ \bs(\bE') = \mu^{-1} \circ t_{jk}^\bs \circ \mu(\bs(\bE)) \,. \]
Similarly, since $t_{jk}^\bs$ is given by Lemma \ref{lem_tijk} and the seeds mutations are given by \eqref{eq_seed_mutation}, it follows that $t_{jk}^\bs$ and $\mu$ commute, and so
$\bs(\bE') = t_{jk}^\bs(\bs(\bE))$. By \cite[Corollary 2.5]{Gor}, exceptional sheaves are uniquely determined by their classes in $K(Z)$, and so we necessarily have $\bH'= t_{jk}^{\bs} \cdot \bH$, where $\bH'$ is the geometric helix generated by $\bE'$.
\end{proof}

\subsection{Classification of geometric helices}
\label{sec_classification}

In this section, we prove the main result of this paper, Theorem \ref{thm_main}, referred to as Theorem \ref{thm_main_intro} in the Introduction, providing a classification of geometric helices on del Pezzo surfaces. We first prove some preliminary results, which will be combined in the proof of Theorem \ref{thm_main}.

\begin{lemma} \label{lem_1}
Let $\bE$ and $\bE'$ be two very strong exceptional collections on a del Pezzo surface $Z$, generating geometric helices $\bH$ and $\bH'$.
Assume that there exists an element $g \in SL(2,\bZ)$ such that 
\[ \psi(\bs(\bE')) = g \circ \psi(\bs(\bE))\,.\]
Then there exist a helix $\bH''$ obtained from $\bH'$ by rotations and orthogonal reorderings, a thread $\bE''$ in $\bH''$,  
and an element $f \in SL(2,\bZ)$ such that, writing $\bE''=(E_1'', \dots, E_n'')$ and $\bE=(E_1, \dots, E_n)$, we have:
\[ \psi([F_i'']) = f( \psi([F_i])) \text{  for all  } 1 \leq i\leq n\,,\]
where $\bF''=(F_n'', \dots, F_1'')$ and $\bF=(F_n, \dots, F_1)$ are the exceptional collections dual to $\bE''$ and $\bE$ respectively.
\end{lemma}

\begin{proof}
An element  $g \in SL(2,\bZ)$ preserves the orientation of $\bR^2$. On the other hand, the collections of vectors $\psi(\bs(\bE'))$ and $\psi(\bs(\bE))$ are 
each cyclically ordered in $\bR^2$ by Proposition \ref{prop_cyclic}(ii). 
Therefore, up to orthogonal reorderings, there exists a cyclic permutation $\sigma$ of $\bZ/n\bZ$ such that, $\psi([F'_{\sigma(i)}])  =g(\psi([F_i]))$ for all $i$, where $\bF'=(F_n', \dots, F_1')$ and $\bF=(F_n, \dots, F_1)$ are the exceptional collections dual to $\bE'$ and $\bE$ respectively. By Lemma \ref{lem_rotation}, there exist a helix $\bH''$ obtained from $\bH'$ by rotation and orthogonal reordering, a thread $\bE''$ in $\bH''$,  
and an element $h \in SL(2,\bZ)$ such that 
$\psi([F_i''])=h(\psi([F'_{\sigma(i)}]))$ for all $1 \leq i \leq n$, where $\bF''=(F_n'', \dots, F_1'')$ is the exceptional collection dual to $\bE''$. Therefore, we obtain 
$ \psi([F_i'']) = f( \psi([F_i]))$ for all $1 \leq i \leq n$, where $f=h \circ g \in SL(2,\bZ)$.
\end{proof} 

\begin{lemma} \label{lem_2}
Let $\bE=(E_1, \dots, E_n)$ and $\bE'
=(E_1', \dots, E_n')$ be two very strong exceptional collections of bundles on a del Pezzo surface $Z$.
Suppose that there exists an element $f \in SL(2,\bZ)$ such that 
\[ \psi([F_i])=f(\psi([F_i']))\text{    for all   }1 \leq i \leq n\,,\]
where $\bF'=(F_n', \dots, F_1')$ and $\bF=(F_n, \dots, F_1)$ are the exceptional collections dual to $\bE'$ and $\bE$ respectively.
Then, there exists a line bundle $L \in \Pic(Z)$ such that, writing $\bE'' = L \otimes \bE$, we have
\[ \psi([F_i'']) = \psi([F_i'])\text{    for all    } 1\leq i \leq n\,,\]
where $\bF''=(F_n'', \dots, F_1'')$ is the exceptional collections dual to $\bE'$
\end{lemma}

\begin{proof}
Since $\psi([F_i])=f(\psi([F_i']))$ for all $1 \leq i\leq n$ and $f\in SL(2,\bZ)$, it follows from Lemma \ref{lem_coeff}
that $c_{\bs(\bE),i}=c_{\bs(\bE)',i}$ for all $1 \leq i\leq n$.
As $\bE$ and $\bE'$ are very strong exceptional collections of bundles, we have $r(E_i)=c_{\bs(\bE),i}$ and 
$r(E_i')=c_{\bs(\bE)',i}$ for all $1\leq i\leq n$ by
Lemma \ref{lem_rank}. Thus, 
we obtain that $r(E_i)=r(E_i')$ for all $1\leq i \leq n$.

We now prove that $r(F_i)=r(F_i')$ for all $1 \leq i\leq n$.
Let $\bF=(F_n, \dots, F_1)$ and $\bF'=(F_n',\dots, F_1')$ be the dual exceptional collections to $\bE$ and $\bE'$ respectively. By duality between $\bE$ and $\bF$, and similarly for $\bE'$ and $\bF'$, we have 
\begin{equation} 
\label{eq_duality_E_F}
[F_i]= \sum_{j=1}^n \chi(F_i, F_j) [E_j] \,\,\,\text{and}\,\,\,[F_i']= \sum_{j=1}^n \chi(F_i', F_j') [E_j']  \,\,\,\, \text{for all}\,\,\,\, 1 \leq i \leq n \,.\end{equation}
Since the pairs $(F_i, F_j)$ with $i\geq j$ are exceptional,  the Euler form is described as follows by Lemma \ref{lem_ex_0}:
\begin{equation} \label{eq_euler_form}
\chi(F_i, F_j):=
\begin{cases}
1 \,\,\,\text{if} \,\,\, i=j \\
\langle \psi([F_i]), \psi([F_j']) \rangle\,\,\text{if}\,\,\, i>j\\
0 \,\,\,\, \text{if}\,\,\, i <j \,.
\end{cases}
\end{equation}
and similarly for $\bF'$. Since $\psi([F_i])=f(\psi([F_i']))$ for all $1 \leq i \leq n$ and
$f\in SL(2,\bZ)$ preserves $\langle-,-\rangle$, we obtain
\[ \langle \psi([F_i]), \psi([F_j])\rangle=\langle \psi([F_i']), \psi([F_j'])\rangle\,\,\,\text{for all}\,\,\, i, j \,,\]
and so 
$\chi(F_i,F_j)=\chi(F_i', F_j')$ for all $i$, $j$. Therefore, comparing the expressions in \eqref{eq_duality_E_F},  and using that $r(E_i)=r(E_i')$ for all $i$, we deduce that $r(F_i)=r(F_i')$ for all $i$. 

Since the classes $(\psi([F_i]))_{1\leq i \leq n}$ span $\bZ^2 \otimes \bQ$, it follows that  $r(f(\alpha))=r(\alpha)$ for all $\alpha \in 
\bZ^2$. 
Because $f\in SL(2,\bZ)$, this implies that $f$ has the form \[ f((r,d))=(r,d+kr)\] for some $k \in \bZ$.
Let $\ell \in \bZ_{>0}$ be the divisibility of $K_Z$ in $NS(Z)$. Since $([F_i])_{1\leq i \leq n}$ forms a basis of $K(Z)$, there exists $j$ with $r(F_j)\neq 0$. By \cite[Proposition 2.10]{KO}, $F_j$ is a shift of a vector bundle, so $r(F_j)$ and $d(F_j)$ are 
non-zero and coprime by Lemma \ref{lem_primitive}. On the other hand, the relation $\psi([F_j'])= f(\psi([F_j]))$ implies \[ d(F_j')=d(F_j)+k \,r(F_j)\,.\]
Since $d(F_j')=c_1(F_j') \cdot (-K_Z)$ and $d(F_j)=c_1(F_j) \cdot (-K_Z)$, both $d(F_j)$ and $d(F_j')$ are divisible by $\ell$. Therefore, $\ell$ divides $k \, r(F_j)$. Since $r(F_j)$ is coprime to $d(F_j)$ and $d(F_j)$ is divisible by $\ell$, it follows that $\ell$ does not divide $r(F_j)$ and $\ell$ divides $k$.

There exists a line bundle $L \in \Pic(Z)$ with $d(L)=\ell$. Indeed, if $Z$ is not $\bP^2$ or $\bP^1 \times \bP^1$, then $Z$ contains an exceptional curve $D$, so with $D^2=-1$ and $D\cdot (-K_Z)=1$, and one may take $L=\cO(D)$. If $Z=\bP^2$, then $K_Z=\cO_{\bP^2}(-3)$, $\ell=3$, and one may take $L=\cO_{\bP^2}(1)$. If $Z=\bP^1 \times \bP^1$, then $K_Z=\cO_{\bP^1 \times \bP^1}(-2,-2)$, $\ell=2$, and one may take $L=\cO_{\bP^1 \times \bP^1}(1,0)$.
Finally, set $\bE'' = L^{\frac{k}{\ell}} \otimes \bE$. Then, $\psi([F_i'']) = \psi([F_i'])$ for all $1 \leq i\leq n$, and this completes the proof.
\end{proof}

\begin{lemma} \label{lem_3}
Let $\bE$ and $\bE'$ be two very strong exceptional collections on a del Pezzo surface $Z$ such that 
\[ \psi([F_i'])=\psi([F_i]) \text{    for all  }
1 \leq i\leq n\,,\]
where $\bF'=(F_n', \dots, F_1')$ and $\bF=(F_n, \dots, F_1)$ are the exceptional collections dual to $\bE'$ and $\bE$ respectively.
Then, there exists an element $g \in O(Z)$ of the orthogonal group of $Z$ such that,  
\[ [E_i']=g([E_i])\text{     for all     }1\leq i\leq n\,.\]
\end{lemma}

\begin{proof}
By assumption,  
\begin{equation} \label{eq_psi_F}
\psi([F_i])=\psi([F_i']) \,\,\,\, \text{for all} \,\,\,\,1 \leq i \leq n\,.
\end{equation}
Since the Euler form is described as in
\eqref{eq_euler_form} for both $\bF$ and $\bF'$, we deduce that
\begin{equation} \label{eq_chi_F}
    \chi(F_i,F_j) = \chi(F_i', F_j')\,\,\,\, \text{for all}\,\,\,\, 1 \leq i,j \leq n \,.
\end{equation}
Since both $([F_i])_{1 \leq i \leq n}$ and $([(F_i)'])_{1 \leq i \leq n}$ are bases of $K(Z)$, there exists a unique automorphism $g: K(Z) \rightarrow K(Z)$ such that $g([F_i])=[(F_i)']$ for all $1 \leq i \leq n$.
By \eqref{eq_psi_F} and \eqref{eq_chi_F}, $g$ is an isometry of $(K(Z), \chi(-,-))$ commuting with $\psi=(r,d)$, and so is an element of the orthogonal group $O(Z)$ by Definition \ref{def_orthogonal}.
\end{proof}

As reviewed in \S\ref{sec_geometric_tilting}, given a geometric helix 
 on a del Pezzo surface, one can produce new geometric helices 
 by applying the following operations: rotation, shifting, orthogonal reordering, tensoring by a line bundle, and tilting. The following result, which is the main theorem of this paper, referred to as Theorem \ref{thm_main_intro} in the Introduction, shows conversely that any two geometric helices are related by a sequence of these operations.

\begin{theorem} \label{thm_main}
    Let $Z$ be a del Pezzo surface. Any two geometric helices on $Z$ are related by a sequence of the following operations:
     rotation, shifting, orthogonal reordering, tensoring by a line bundle, and tilting.
\end{theorem}

\begin{proof}
Let $\bH$ and $\bH'$ be two geometric helices in $D(Z)$. We prove that they are related by the operations described in Theorem \ref{thm_main}. 
By Lemma \ref{lem_bundles} and 
Proposition \ref{prop_bundles}, each of $\bH$ and $\bH'$ differs by a global shift from a geometric helix of bundles. Hence, after applying suitable global shifts, we may assume that both $\bH$ and $\bH'$ are geometric helices of bundles.
We prove that $\bH$ and $\bH'$ 
are related by rotation, orthogonal reordering, tensoring by a line bundle, and tilting.  These operations preserve geometric helices of bundles: this is immediate for   rotation, orthogonal reordering, tensoring by a line bundle, and follows from Lemma \ref{lem_shift} in the case of tilting.

Let $\bE$ and $\bE'$ be very strong exceptional collections obtained as threads of $\bH$ and $\bH'$ respectively. 
Consider the associated seeds $\bs(\bE)$ and $\bs(\bE')$ as in Definition \ref{def_seeds_from_ex}. 
By Theorem \ref{thm_q_painleve}, their mutation
equivalence classes $\mathscr{S}(\bE)$ and $\mathscr{S}(\bE')$ are both of $q$-Painlev\'e type. 
Moreover, Lemma \ref{lem_isometry2} shows that $(K,(-,-)_{\mathscr{S}(\bE)})$ and 
$(K,(-,-)_{\mathscr{S}(\bE')})$ are both isometric to 
\[ 
(K_Z^\perp \oplus \bZ \delta, -\chi(-,-))\,.\]
By Theorem \ref{thm_q_P} together with the correspondence between seed mutations and tilting operations of geometric helices given in Proposition \ref{prop_tilting}, 
there exists a thread $\mu(\bE)$ in a geometric helix obtained from $\bH$ by a sequence of tilting operations and orthogonal reorderings, and an automorphism 
$g \in  SL(2,\bZ)$ such that 
\[ \psi(\bs(\bE')) = g \circ \psi(\mu(\bs(\bE))) \,.\]
Replacing $\bE$ by $\mu(\bE)$, we may assume that 
\[ \psi(\bs(\bE')) = g \circ \psi(\bs(\bE)) \,.\]

We then apply Lemma \ref{lem_1} and Lemma \ref{lem_2}: after replacing $\bE'$ by a thread of a helix obtained from $\bH'$ by a sequence of 
rotations, orthogonal reorderings, and tensorings by a line bundle, we may arrange that
\[ \psi([F_i'])=\psi([F_i]) \text{    for all  }
1 \leq i\leq n\,,\]
where $\bF'=(F_n', \dots, F_1')$ and $\bF=(F_n, \dots, F_1)$ are the exceptional collections dual to $\bE'$ and $\bE$ respectively.

Lemma \ref{lem_3} then shows that there exists an element $g \in O(Z)$ of the 
orthogonal group of $Z$ such that $[E_i'] = g([E_i])$ for all $1 \leq i\leq n$. 
By Proposition \ref{prop_weyl_affine_gen}, we have  $O(Z)=W(Z)\ltimes \Pic^0(Z)$, where $\Pic^0(Z)$ acts on $K(Z)$ by tensor product with line bundles. 
Therefore, up to tensoring $\bE'$ by a line bundle, we may assume that $g$ is an element  of the finite Weyl group $W(Z)$, and so in particular of the affine Weyl group $\widehat{W}(Z)$.
By  Theorem \ref{thm_weyl_action},
there exists a very strong exceptional collection $\bE''=(E_1'',\dots, E_n'')$ such that \[ [E_i'']=g([E_i])\,\,\, \text{for all}\,\,\, 1 \leq i \leq n\,.\]
Moreover, the geometric helix generated by $\bE''$ is obtained from the geometric 
helix generated by $\bE$ by a sequence of tilting operations and orthogonal reorderings. 
Therefore, up to replacing $\bE$ by $\bE''$, we may assume that $[E_i']=[E_i]$ for all $1 \leq i \leq n$.

Finally, by \cite[Corollary 2.5]{Gor}, an exceptional sheaf is uniquely determined by its class in $K(Z)$. Therefore, since $\bE$ and $\bE'$ are very strong exceptional collections of bundles, $[E_i']=[E_i]$ implies $E_i'=E_i$, so $\bE'=\bE$ and this completes the proof.
\end{proof}

\begin{remark}
The group of auto-equivalences $\Aut(D(Z))$ of $D(Z)$ naturally acts on the set of geometric helices.
Since $\omega_Z^{-1}$ is ample, it follows from \cite[Theorem 3.1]{bondal_orlov} that
$\Aut(D(Z))=\Aut(Z) \rtimes (\Pic(Z) \times \bZ[1])$, where $\Aut(Z)$ is the group of automorphisms of $Z$.
We now explain why $\Aut(Z)$ does not appear in the statement of Theorem \ref{thm_main}. 
First, the exceptional objects are rigid, and so the  connected component $\Aut^0(Z)$ of the identity acts trivially on exceptional collections. Hence, the action of $\Aut(Z)$ factors through an action of the discrete quotient $\Aut(Z)/\Aut^0(Z)$. 
By \cite[Proposition 8.2.39]{dolgachev}, this quotient injects via its action on $K_Z^\perp$ in the finite Weyl group $W(Z)$.
By Theorem \ref{thm_weyl_action}, $W(Z)$ acts on geometric helices via tilting operations and orthogonal reorderings.
Thus, the action of $\Aut(Z)$ is already accounted for by these operations.
\end{remark}

\subsection{Orthogonal group and admissible isometries}

Let $Z$ be a del Pezzo surface and $\bE$ a very strong exceptional collection on $Z$. In this final section, we give an interpretation of the orthogonal group $O(Z)$ of $Z$ introduced in Definition \ref{def_orthogonal} in terms of the geometry of a framed toric model $(Y,D)$ for $\bs(\bE)$ as in \S\ref{sec_log_CY}. This result is not used in the proof of the main result in \S \ref{sec_classification} above but is of independent interest. 

Given a log Calabi--Yau surface $(Y,D)$, consisting of a smooth projective surface $Y$ and a singular reduced anticanonical divisor $D$, the group of \emph{admissible isometries} $\Gamma(Y,D)$ of $(Y,D)$ is the group of isometries of $NS(Y)$ fixing the classes of the irreducible components of $D$ and preserving the nef cone of a generic deformation of $D$ -- see \cite[Definition 9.1]{friedman2015geometry} and \cite[Lemma 4.5]{GHK}. By \cite[Theorem 5.15]{GHK}, $\Gamma(Y,D)$ is also the global monodromy group of deformations of $(Y,D)$. 

For the following statement, observe that the restriction to $\Lambda_{(Y,D)}$ as in \eqref{eq_lambda} induces an injection of $\Gamma(Y,D)$ in the group of isometries of $\Lambda_{(Y,D)}$ since $\Gamma(Y,D)$ fixes the classes of the irreducible components of $D$. Moreover, recall from \eqref{eq_isom} that there exists a natural identification $ K \simeq \Lambda_{(Y,D)}$.

\begin{prop} \label{prop_admissible}
Let $Z$ be a del Pezzo surface, $\bE$ a very strong exceptional collection on $Z$ and $(Y,D)$ a framed toric model of the seed $\bs(\bE)$.  Then, the actions of the orthogonal group $O(Z)$ and of the group of admissible isometries $\Gamma(Y,D)$ on $ K \simeq \Lambda_{(Y,D)}$  induce an isomorphism 
\[ O(Z) \simeq \Gamma(Y,D) \,.\]
\end{prop}

\begin{proof}
By Theorem \ref{thm_q_painleve}, the seed $\bs(\bE)$ is of $q$-Painlev\'e type. In particular, 
the intersection form $(-,-)_{\mathscr{S}(\bE)}$ is negative semi-definite but not negative definite. Hence, \cite[Lemma 9.18]{friedman2015geometry} applies: there exists an orthogonal decomposition $\Lambda_{(Y,D)} = \bZ[D] \oplus \overline{\Lambda}$, together with an injection $\overline{\Lambda} \rightarrow \Gamma(Y,D)$, inducing a semi-direct product decomposition 
\[ \Gamma(Y,D)= G \ltimes \overline{\Lambda} \,,\]
where $G$ is the group of isometries of $\overline{\Lambda}$ which extend to isometries of a unimodular lattice $\Lambda_0$ of type $E_8$ containing $\overline{\Lambda}$ as a saturated sublattice. 
On the other hand, by Proposition \ref{prop_weyl_affine_gen}, we have
\[ O(Z)=W(Z)\ltimes \Pic^0(Z) \,. \]
Thus, to prove Proposition \ref{prop_admissible}, it suffices to identify these semi-direct product descriptions of $\Gamma(Y,D)$ and $O(Z)$.
Using the identifications $\Lambda_{(Y,D)} = K= \mathrm{Ker}(\psi)= \bZ \delta \oplus K_Z^\perp$, one checks that $\bZ [D]=\bZ \delta$ and $\overline{\Lambda}=K_Z^\perp$.
Moreover, comparing \eqref{eq_t_L} with the explicit description of the injection $\overline{\Lambda} \rightarrow \Gamma(Y,D)$ given in the proof of \cite[Lemma 9.18]{friedman2015geometry}, one sees that the action by tensor product of $\Pic^0(Z)$ on $K_Z^\perp$ coincides with the action of $\overline{\Lambda}$ on $\Lambda_{(Y,D)}$. 
It remains to identify $G$ and $W(Z)$ as groups acting by isometries on $\overline{\Lambda}=K_Z^\perp$. To do this, note that if $L$ is a lattice, realized as a saturated sublattice of an unimodular lattice $L_0$, then the group of isometries of $L$ that extends to  isometries of $L_0$ acting trivially on $L^\perp$ is independent of $L_0$: indeed, it is
precisely the group of isometries of $L$ acting trivially on the discriminant group $L^\star/L$. By definition, $G$ is the group of isometries of $\overline{\Lambda}$ that extend to isometries of the unimodular lattice $\Lambda_0$ acting trivially on $\overline{\Lambda}^\perp$. On the other hand, by Proposition \ref{prop_finite_weyl}, $W(Z)$ is the precisely the group of isometries of $K_Z^\perp = \overline{\Lambda}$ that extend to isometries of the unimodular lattice $NS(Z)$ acting trivially on $\bQ K_Z \cap NS(Z)=(K_Z^\perp)^\perp$. Hence, we conclude that $G=W(Z)$ and this completes the proof.
\end{proof}

\begin{remark}
    Let $\bs$ be a seed of $q$-Painlev\'e type and let $(Y,D)$ be a framed toric model for $\bs$. After contracting any exceptional curve contained in $D$, we may assume that no component of $D$ is exceptional.
    Then, it follows from \cite[Proposition 3.22]{Mizuno} and \cite[Proposition 2]{Sak} that $Y$ is isomorphic to a blow-up of $\bP^2$ in nine points, and that
    $D$ is an anticanonical cycle of $(-2)$-curves $D_i$. In \cite{Mizuno, Sak}, the \emph{Cremona group} $\mathrm{Cr}(Y)$ is defined to be the group of isometries of $NS(Y)$ which fix the class of $D$ and preserve the nef cone of a generic deformation. In contrast to admissible isometries, elements of $\mathrm{Cr}(Y)$ need not preserve the individual classes of the divisors $D_i$.
    The action of $\mathrm{Cr}(Y)$ by permutations on the components $D_i$ yields a semi-direct product decomposition
   \[ \mathrm{Cr}(Y) = H \ltimes \Gamma(Y,D)\,,\]
   where $H$ is a finite group described explicitly in \cite[Theorem 26]{Sak} and \cite[Theorem 3.34]{Mizuno}. 
Although Proposition \ref{prop_admissible} identifies $\Gamma(Y,D)$ with the orthogonal group $O(Z)$ of the corresponding del Pezzo surface $Z$, it is not clear to us whether the entire Cremona group $\mathrm{Cr}(Y)$ admits a similarly natural interpretation in terms of $Z$.
\end{remark}

%%%%%%%%%%%%%%%%%%%%%%%%%%%%%%%%%%%%%%%%%%%%%%%
%%%%%%%%%%%%%%%%%%%%%%%%%%%%%%%%%%%%%%%%%%%%%%%
%%%%%%%%%%%%%%%%%%%%%%%%%%%%%%%%%%%%%%%%%%%%%%%
%%%%%%%%%%%%%%%%%%%%%%%%%%%%%%%%%%%%%%%%%%%%%%%

\bibliographystyle{plain}
\bibliography{bibliography}

@article {mirror_orbifold,
    AUTHOR = {Akhtar, Mohammad and Coates, Tom and Corti, Alessio and
              Heuberger, Liana and Kasprzyk, Alexander and Oneto, Alessandro
              and Petracci, Andrea and Prince, Thomas and Tveiten, Ketil},
     TITLE = {Mirror symmetry and the classification of orbifold del {P}ezzo
              surfaces},
   JOURNAL = {Proc. Amer. Math. Soc.},
  FJOURNAL = {Proceedings of the American Mathematical Society},
    VOLUME = {144},
      YEAR = {2016},
    NUMBER = {2},
     PAGES = {513--527},
      ISSN = {0002-9939,1088-6826},
   MRCLASS = {14J26 (52B20)},
  MRNUMBER = {3430830},
MRREVIEWER = {Makiko\ Mase},
       DOI = {10.1090/proc/12876},
       URL = {https://doi.org/10.1090/proc/12876},
}

@article {mutations,
    AUTHOR = {Akhtar, Mohammad and Coates, Tom and Galkin, Sergey and
              Kasprzyk, Alexander M.},
     TITLE = {Minkowski polynomials and mutations},
   JOURNAL = {SIGMA Symmetry Integrability Geom. Methods Appl.},
  FJOURNAL = {SIGMA. Symmetry, Integrability and Geometry. Methods and
              Applications},
    VOLUME = {8},
      YEAR = {2012},
     PAGES = {Paper 094, 17},
      ISSN = {1815-0659},
   MRCLASS = {52B20 (14J33 14M25 16S34)},
  MRNUMBER = {3007265},
MRREVIEWER = {Nathan\ Owen\ Ilten},
       DOI = {10.3842/SIGMA.2012.094},
       URL = {https://doi.org/10.3842/SIGMA.2012.094},
}

@article{arguz2026mock,
  title={Mock modularity of log {G}romov--{W}itten invariants: the mirror to $\mathbb{P}^2$},
  author={Arg{\"u}z, H{\"u}lya},
  journal={arXiv preprint arXiv:2602.08153},
  year={2026}
}

@article {ABquivers,
    AUTHOR = {Arg\"uz, H\"ulya and Bousseau, Pierrick},
     TITLE = {Quivers and curves in higher dimension},
   JOURNAL = {Trans. Amer. Math. Soc.},
  FJOURNAL = {Transactions of the American Mathematical Society},
    VOLUME = {378},
      YEAR = {2025},
    NUMBER = {1},
     PAGES = {389--420},
      ISSN = {0002-9947,1088-6850},
   MRCLASS = {14N35},
  MRNUMBER = {4840309},
       DOI = {10.1090/tran/9232},
       URL = {https://doi.org/10.1090/tran/9232},
}

@article {aspinwall,
    AUTHOR = {Aspinwall, Paul S. and Melnikov, Ilarion V.},
     TITLE = {D-branes on vanishing del {P}ezzo surfaces},
   JOURNAL = {J. High Energy Phys.},
  FJOURNAL = {Journal of High Energy Physics. A SISSA Journal},
      YEAR = {2004},
    NUMBER = {12},
     PAGES = {042, 30},
      ISSN = {1126-6708,1029-8479},
   MRCLASS = {81T30 (14J81 16G20 18E30)},
  MRNUMBER = {2128461},
MRREVIEWER = {Bal\'azs\ Szendr\H oi},
       DOI = {10.1088/1126-6708/2004/12/042},
       URL = {https://doi.org/10.1088/1126-6708/2004/12/042},
}

@article {aspinwall2,
    AUTHOR = {Aspinwall, Paul S. and Fidkowski, Lukasz M.},
     TITLE = {Superpotentials for quiver gauge theories},
   JOURNAL = {J. High Energy Phys.},
  FJOURNAL = {Journal of High Energy Physics. A SISSA Journal},
      YEAR = {2006},
    NUMBER = {10},
     PAGES = {047, 25},
      ISSN = {1126-6708,1029-8479},
   MRCLASS = {81T30 (14J81 18E30 81T13)},
  MRNUMBER = {2266656},
MRREVIEWER = {Yang-Hui\ He},
       DOI = {10.1088/1126-6708/2006/10/047},
       URL = {https://doi.org/10.1088/1126-6708/2006/10/047},
}

@article {atiyah,
    AUTHOR = {Atiyah, Michael F.},
     TITLE = {Vector bundles over an elliptic curve},
   JOURNAL = {Proc. London Math. Soc. (3)},
  FJOURNAL = {Proceedings of the London Mathematical Society. Third Series},
    VOLUME = {7},
      YEAR = {1957},
     PAGES = {414--452},
      ISSN = {0024-6115,1460-244X},
   MRCLASS = {14.55 (14.20)},
  MRNUMBER = {131423},
MRREVIEWER = {F.\ Hirzebruch},
       DOI = {10.1112/plms/s3-7.1.414},
       URL = {https://doi.org/10.1112/plms/s3-7.1.414},
}

@article {AKO,
    AUTHOR = {Auroux, Denis and Katzarkov, Ludmil and Orlov, Dmitri},
     TITLE = {Mirror symmetry for del {P}ezzo surfaces: vanishing cycles and
              coherent sheaves},
   JOURNAL = {Invent. Math.},
  FJOURNAL = {Inventiones Mathematicae},
    VOLUME = {166},
      YEAR = {2006},
    NUMBER = {3},
     PAGES = {537--582},
      ISSN = {0020-9910,1432-1297},
   MRCLASS = {14J32 (14A22 14J26 18E30 53D12 53D40)},
  MRNUMBER = {2257391},
MRREVIEWER = {Richard\ P.\ Thomas},
       DOI = {10.1007/s00222-006-0003-4},
       URL = {https://doi.org/10.1007/s00222-006-0003-4},
}

@article {BMP,
    AUTHOR = {Beaujard, Guillaume and Manschot, Jan and Pioline, Boris},
     TITLE = {Vafa-{W}itten invariants from exceptional collections},
   JOURNAL = {Comm. Math. Phys.},
  FJOURNAL = {Communications in Mathematical Physics},
    VOLUME = {385},
      YEAR = {2021},
    NUMBER = {1},
     PAGES = {101--226},
      ISSN = {0010-3616,1432-0916},
   MRCLASS = {14N35 (81T30)},
  MRNUMBER = {4275783},
       DOI = {10.1007/s00220-021-04074-2},
       URL = {https://doi.org/10.1007/s00220-021-04074-2},
}

@article {BGM,
    AUTHOR = {Bershtein,  Mikhail  and Gavrylenko, Pavlo and Marshakov, Andrei},
     TITLE = {Cluster integrable systems, {$q$}-{P}ainlev\'e{} equations and
              their quantization},
   JOURNAL = {J. High Energy Phys.},
  FJOURNAL = {Journal of High Energy Physics},
      YEAR = {2018},
    NUMBER = {2},
     PAGES = {077, front matter+33},
      ISSN = {1126-6708,1029-8479},
   MRCLASS = {81R12},
  MRNUMBER = {3789593},
       DOI = {10.1007/jhep02(2018)077},
       URL = {https://doi.org/10.1007/jhep02(2018)077},
}

@article {blanc,
    AUTHOR = {Blanc, J\'er\'emy},
     TITLE = {Symplectic birational transformations of the plane},
   JOURNAL = {Osaka J. Math.},
  FJOURNAL = {Osaka Journal of Mathematics},
    VOLUME = {50},
      YEAR = {2013},
    NUMBER = {2},
     PAGES = {573--590},
      ISSN = {0030-6126},
   MRCLASS = {14E07 (53D05)},
  MRNUMBER = {3080816},
MRREVIEWER = {Paul\ Michael\ Reschke},
       URL = {http://projecteuclid.org/euclid.ojm/1371833501},
}

@article {bondal_orlov,
    AUTHOR = {Bondal, Alexei and Orlov, Dmitri},
     TITLE = {Reconstruction of a variety from the derived category and
              groups of autoequivalences},
   JOURNAL = {Compositio Math.},
  FJOURNAL = {Compositio Mathematica},
    VOLUME = {125},
      YEAR = {2001},
    NUMBER = {3},
     PAGES = {327--344},
      ISSN = {0010-437X,1570-5846},
   MRCLASS = {18E30 (14F05)},
  MRNUMBER = {1818984},
MRREVIEWER = {Richard\ P.\ Thomas},
       DOI = {10.1023/A:1002470302976},
       URL = {https://doi.org/10.1023/A:1002470302976},
}

@article {bridgeland_tstructures,
    AUTHOR = {Bridgeland, Tom},
     TITLE = {t-structures on some local {C}alabi-{Y}au varieties},
   JOURNAL = {J. Algebra},
  FJOURNAL = {Journal of Algebra},
    VOLUME = {289},
      YEAR = {2005},
    NUMBER = {2},
     PAGES = {453--483},
      ISSN = {0021-8693,1090-266X},
   MRCLASS = {14J32 (14J81 18E30)},
  MRNUMBER = {2142382},
MRREVIEWER = {Andrei\ D.\ Halanay},
       DOI = {10.1016/j.jalgebra.2005.03.016},
       URL = {https://doi.org/10.1016/j.jalgebra.2005.03.016},
}

@article{bridgeland2024invariant,
  title={Invariant stability conditions on certain {C}alabi-{Y}au threefolds},
  author={Bridgeland, Tom and Del Monte, Fabrizio and Giovenzana, Luca},
  journal={arXiv preprint arXiv:2412.08531},
  year={2024}
}

@article {BS,
    AUTHOR = {Bridgeland, Tom and Stern, David},
     TITLE = {Helices on del {P}ezzo surfaces and tilting {C}alabi-{Y}au
              algebras},
   JOURNAL = {Adv. Math.},
  FJOURNAL = {Advances in Mathematics},
    VOLUME = {224},
      YEAR = {2010},
    NUMBER = {4},
     PAGES = {1672--1716},
      ISSN = {0001-8708,1090-2082},
   MRCLASS = {14F05 (14J26 16E35 16G20)},
  MRNUMBER = {2646308},
MRREVIEWER = {Johannes\ Walcher},
       DOI = {10.1016/j.aim.2010.01.018},
       URL = {https://doi.org/10.1016/j.aim.2010.01.018},
}

@article {closset,
    AUTHOR = {Closset, Cyril and Del Zotto, Michele},
     TITLE = {On 5d {SCFT}s and their {BPS} quivers part {I}: {B}-branes and
              brane tilings},
   JOURNAL = {Adv. Theor. Math. Phys.},
  FJOURNAL = {Advances in Theoretical and Mathematical Physics},
    VOLUME = {26},
      YEAR = {2022},
    NUMBER = {1},
     PAGES = {37--142},
      ISSN = {1095-0761,1095-0753},
   MRCLASS = {81T40 (14M25 16G20 81T30 81T33 81T60)},
  MRNUMBER = {4504847},
MRREVIEWER = {Chien-Hao\ Liu},
       DOI = {10.4310/atmp.2022.v26.n1.a2},
       URL = {https://doi.org/10.4310/atmp.2022.v26.n1.a2},
}

@article {clossetu,
    AUTHOR = {Closset, Cyril and Magureanu, Horia},
     TITLE = {The {$U$}-plane of rank-one 4d {$N = 2$} {KK} theories},
   JOURNAL = {SciPost Phys.},
  FJOURNAL = {SciPost Physics},
    VOLUME = {12},
      YEAR = {2022},
    NUMBER = {2},
     PAGES = {Paper No. 065, 139},
      ISSN = {2542-4653},
   MRCLASS = {81T40 (81T60)},
  MRNUMBER = {4387554},
       DOI = {10.21468/scipostphys.12.2.065},
       URL = {https://doi.org/10.21468/scipostphys.12.2.065},
}

@article{corti2023cluster,
  title={Cluster varieties and toric specializations of {F}ano varieties},
  author={Corti, Alessio},
  journal={arXiv preprint arXiv:2304.04141},
  year={2023}
}

@article {DWZ,
    AUTHOR = {Derksen, Harm and Weyman, Jerzy and Zelevinsky, Andrei},
     TITLE = {Quivers with potentials and their representations. {I}.
              {M}utations},
   JOURNAL = {Selecta Math. (N.S.)},
  FJOURNAL = {Selecta Mathematica. New Series},
    VOLUME = {14},
      YEAR = {2008},
    NUMBER = {1},
     PAGES = {59--119},
      ISSN = {1022-1824,1420-9020},
   MRCLASS = {16G10 (13F60 16G20 16S38)},
  MRNUMBER = {2480710},
MRREVIEWER = {M\'aty\'as\ Domokos},
       DOI = {10.1007/s00029-008-0057-9},
       URL = {https://doi.org/10.1007/s00029-008-0057-9},
}

@book {dolgachev,
    AUTHOR = {Dolgachev, Igor V.},
     TITLE = {Classical algebraic geometry},
      NOTE = {A modern view},
 PUBLISHER = {Cambridge University Press, Cambridge},
      YEAR = {2012},
     PAGES = {xii+639},
      ISBN = {978-1-107-01765-8},
   MRCLASS = {14-02 (14-01)},
  MRNUMBER = {2964027},
MRREVIEWER = {Arnaud\ Beauville},
       DOI = {10.1017/CBO9781139084437},
       URL = {https://doi.org/10.1017/CBO9781139084437},
}

@article {evans_smith,
    AUTHOR = {Evans, Jonathan David and Smith, Ivan},
     TITLE = {Markov numbers and {L}agrangian cell complexes in the complex
              projective plane},
   JOURNAL = {Geom. Topol.},
  FJOURNAL = {Geometry \& Topology},
    VOLUME = {22},
      YEAR = {2018},
    NUMBER = {2},
     PAGES = {1143--1180},
      ISSN = {1465-3060,1364-0380},
   MRCLASS = {53D35 (14J17 53D42)},
  MRNUMBER = {3748686},
MRREVIEWER = {Alexander\ Fel\cprime shtyn},
       DOI = {10.2140/gt.2018.22.1143},
       URL = {https://doi.org/10.2140/gt.2018.22.1143},
}

@article{fomin2024cyclically,
  title={Cyclically ordered quivers},
  author={{F}omin, {S}ergey and {N}eville, {S}cott},
  journal={arXiv preprint arXiv:2406.03604},
  year={2024}
}

@article{friedman2015geometry,
  title={On the geometry of anticanonical pairs},
  author={{F}riedman, {R}obert},
  journal={arXiv preprint arXiv:1502.02560},
  year={2015}
}

@book {Fultontoric,
    AUTHOR = {Fulton, William},
     TITLE = {Introduction to toric varieties},
    SERIES = {Annals of Mathematics Studies},
    VOLUME = {131},
      NOTE = {The William H. Roever Lectures in Geometry},
 PUBLISHER = {Princeton University Press, Princeton, NJ},
      YEAR = {1993},
     PAGES = {xii+157},
      ISBN = {0-691-00049-2},
   MRCLASS = {14M25 (14-02 14J30)},
  MRNUMBER = {1234037},
MRREVIEWER = {T.\ Oda},
       DOI = {10.1515/9781400882526},
       URL = {https://doi.org/10.1515/9781400882526},
}

@article {GGLP,
    AUTHOR = {Grassi, Antonella and Gugiatti, Giulia and Lutz, Wendelin and
              Petracci, Andrea},
     TITLE = {Reflexive polygons and rational elliptic surfaces},
   JOURNAL = {Rend. Circ. Mat. Palermo (2)},
  FJOURNAL = {Rendiconti del Circolo Matematico di Palermo. Second Series},
    VOLUME = {72},
      YEAR = {2023},
    NUMBER = {6},
     PAGES = {3185--3221},
      ISSN = {0009-725X,1973-4409},
   MRCLASS = {14J33 (14J26 14J27)},
  MRNUMBER = {4622146},
MRREVIEWER = {Dino\ Festi},
       DOI = {10.1007/s12215-023-00922-3},
       URL = {https://doi.org/10.1007/s12215-023-00922-3},
}

@article {Gor,
    AUTHOR = {Gorodentsev, Alexey L.},
     TITLE = {Exceptional bundles on surfaces with a moving anticanonical
              class},
   JOURNAL = {Izv. Akad. Nauk SSSR Ser. Mat.},
  FJOURNAL = {Izvestiya Akademii Nauk SSSR. Seriya Matematicheskaya},
    VOLUME = {52},
      YEAR = {1988},
    NUMBER = {4},
     PAGES = {740--757, 895},
      ISSN = {0373-2436},
   MRCLASS = {14F05},
  MRNUMBER = {966982},
MRREVIEWER = {I.\ Dolgachev},
       DOI = {10.1070/IM1989v033n01ABEH000813},
       URL = {https://doi.org/10.1070/IM1989v033n01ABEH000813},
}

@article {GorKul,
    AUTHOR = {Gorodentsev, Alexey L. and Kuleshov,  Sergey A.},
     TITLE = {Helix theory},
   JOURNAL = {Mosc. Math. J.},
  FJOURNAL = {Moscow Mathematical Journal},
    VOLUME = {4},
      YEAR = {2004},
    NUMBER = {2},
     PAGES = {377--440, 535},
      ISSN = {1609-3321,1609-4514},
   MRCLASS = {14F05 (14J32 14J45 18F20 18F30)},
  MRNUMBER = {2108443},
MRREVIEWER = {Andrei\ D.\ Halanay},
       DOI = {10.17323/1609-4514-2004-4-2-377-440},
       URL = {https://doi.org/10.17323/1609-4514-2004-4-2-377-440},
}

@article {GHKbir,
    AUTHOR = {Gross, Mark and Hacking, Paul and Keel, Sean},
     TITLE = {Birational geometry of cluster algebras},
   JOURNAL = {Algebr. Geom.},
  FJOURNAL = {Algebraic Geometry},
    VOLUME = {2},
      YEAR = {2015},
    NUMBER = {2},
     PAGES = {137--175},
      ISSN = {2313-1691,2214-2584},
   MRCLASS = {14E05 (13F60 14M25)},
  MRNUMBER = {3350154},
MRREVIEWER = {Sergio\ Mathew\ Da Silva},
       DOI = {10.14231/AG-2015-007},
       URL = {https://doi.org/10.14231/AG-2015-007},
}

@article{gugiatti2025mirrors,
  title={On the mirrors of low-degree del {P}ezzo surfaces},
  author={Gugiatti, Giulia and Rota, Franco},
  journal={arXiv preprint arXiv:2506.21758},
  year={2025}
}

@article {GHK,
    AUTHOR = {Gross, Mark and Hacking, Paul and Keel, Sean},
     TITLE = {Moduli of surfaces with an anti-canonical cycle},
   JOURNAL = {Compos. Math.},
  FJOURNAL = {Compositio Mathematica},
    VOLUME = {151},
      YEAR = {2015},
    NUMBER = {2},
     PAGES = {265--291},
      ISSN = {0010-437X,1570-5846},
   MRCLASS = {14J10 (14J26)},
  MRNUMBER = {3314827},
MRREVIEWER = {Makiko\ Mase},
       DOI = {10.1112/S0010437X14007611},
       URL = {https://doi.org/10.1112/S0010437X14007611},
}

@article {hacking,
    AUTHOR = {Hacking, Paul},
     TITLE = {Exceptional bundles associated to degenerations of surfaces},
   JOURNAL = {Duke Math. J.},
  FJOURNAL = {Duke Mathematical Journal},
    VOLUME = {162},
      YEAR = {2013},
    NUMBER = {6},
     PAGES = {1171--1202},
      ISSN = {0012-7094,1547-7398},
   MRCLASS = {14J60 (14D06 14J10)},
  MRNUMBER = {3053568},
MRREVIEWER = {Arvid\ Perego},
       DOI = {10.1215/00127094-2147532},
       URL = {https://doi.org/10.1215/00127094-2147532},
}

@incollection {hacking2,
    AUTHOR = {Hacking, Paul},
     TITLE = {Compact moduli spaces of surfaces and exceptional vector
              bundles},
 BOOKTITLE = {Compactifying moduli spaces},
    SERIES = {Adv. Courses Math. CRM Barcelona},
     PAGES = {41--67},
 PUBLISHER = {Birkh\"auser/Springer, Basel},
      YEAR = {2016},
      ISBN = {978-3-0348-0920-7; 978-3-0348-0921-4},
   MRCLASS = {14J10 (14J29 14J60 58D27)},
  MRNUMBER = {3495111},
MRREVIEWER = {Alan\ Matthew\ Thompson},
}

@article {hacking_keating,
    AUTHOR = {Hacking, Paul and Keating, Ailsa},
     TITLE = {Homological mirror symmetry for log {C}alabi-{Y}au surfaces},
      NOTE = {With an appendix by Wendelin Lutz},
   JOURNAL = {Geom. Topol.},
  FJOURNAL = {Geometry \& Topology},
    VOLUME = {26},
      YEAR = {2022},
    NUMBER = {8},
     PAGES = {3747--3833},
      ISSN = {1465-3060,1364-0380},
   MRCLASS = {53D37 (14B05 14J33 18G70)},
  MRNUMBER = {4562569},
MRREVIEWER = {Donatella\ Iacono},
       DOI = {10.2140/gt.2022.26.3747},
       URL = {https://doi.org/10.2140/gt.2022.26.3747},
}

@article {hacking_keating_2,
    AUTHOR = {Hacking, Paul and Keating, Ailsa},
     TITLE = {Symplectomorphisms of some {W}einstein 4-manifolds},
   JOURNAL = {Geom. Topol.},
  FJOURNAL = {Geometry \& Topology},
    VOLUME = {30},
      YEAR = {2026},
    NUMBER = {2},
     PAGES = {645--699},
      ISSN = {1465-3060,1364-0380},
   MRCLASS = {14J17 (53D05 53D37)},
  MRNUMBER = {5047784},
       DOI = {10.2140/gt.2026.30.645},
       URL = {https://doi.org/10.2140/gt.2026.30.645},
}

@article{hara2026stability,
  title={Stability conditions on noncommutative crepant resolutions of 3-dimensional isolated singularities},
  author={Hara, Wahei and Hirano, Yuki},
  journal={arXiv preprint arXiv:2603.04858},
  year={2026}
}

@article {Her,
    AUTHOR = {Herzog, Christopher P.},
     TITLE = {Seiberg duality is an exceptional mutation},
   JOURNAL = {J. High Energy Phys.},
  FJOURNAL = {Journal of High Energy Physics. A SISSA Journal},
      YEAR = {2004},
    NUMBER = {8},
     PAGES = {064, 31},
      ISSN = {1126-6708,1029-8479},
   MRCLASS = {81T30 (14J81 81T45)},
  MRNUMBER = {2109842},
MRREVIEWER = {Yang-Hui\ He},
       DOI = {10.1088/1126-6708/2004/08/064},
       URL = {https://doi.org/10.1088/1126-6708/2004/08/064},
}

@article {Her_Karp,
    AUTHOR = {Herzog, Christopher P. and Karp, Robert L.},
     TITLE = {On the geometry of quiver gauge theories (stacking exceptional
              collections)},
   JOURNAL = {Adv. Theor. Math. Phys.},
  FJOURNAL = {Advances in Theoretical and Mathematical Physics},
    VOLUME = {13},
      YEAR = {2009},
    NUMBER = {3},
     PAGES = {599--636},
      ISSN = {1095-0761,1095-0753},
   MRCLASS = {14F05 (14J81 81T30)},
  MRNUMBER = {2610572},
MRREVIEWER = {Yang-Hui\ He},
       DOI = {10.4310/atmp.2009.v13.n3.a1},
       URL = {https://doi.org/10.4310/atmp.2009.v13.n3.a1},
}

@article {HP,
    AUTHOR = {Hille, Lutz and Perling, Markus},
     TITLE = {Exceptional sequences of invertible sheaves on rational
              surfaces},
   JOURNAL = {Compos. Math.},
  FJOURNAL = {Compositio Mathematica},
    VOLUME = {147},
      YEAR = {2011},
    NUMBER = {4},
     PAGES = {1230--1280},
      ISSN = {0010-437X,1570-5846},
   MRCLASS = {14F05 (14J26 14M25)},
  MRNUMBER = {2822868},
MRREVIEWER = {Pawel\ Sosna},
       DOI = {10.1112/S0010437X10005208},
       URL = {https://doi.org/10.1112/S0010437X10005208},
}

@book {Huy,
    AUTHOR = {Huybrechts, Daniel},
     TITLE = {Fourier-{M}ukai transforms in algebraic geometry},
    SERIES = {Oxford Mathematical Monographs},
 PUBLISHER = {The Clarendon Press, Oxford University Press, Oxford},
      YEAR = {2006},
     PAGES = {viii+307},
      ISBN = {978-0-19-929686-6; 0-19-929686-3},
   MRCLASS = {14F05 (14-02 18E30)},
  MRNUMBER = {2244106},
MRREVIEWER = {Bal\'azs\ Szendr\H oi},
       DOI = {10.1093/acprof:oso/9780199296866.001.0001},
       URL = {https://doi.org/10.1093/acprof:oso/9780199296866.001.0001},
}

@article {IyamaReiten,
    AUTHOR = {Iyama, Osamu and Reiten, Idun},
     TITLE = {Fomin-{Z}elevinsky mutation and tilting modules over
              {C}alabi-{Y}au algebras},
   JOURNAL = {Amer. J. Math.},
  FJOURNAL = {American Journal of Mathematics},
    VOLUME = {130},
      YEAR = {2008},
    NUMBER = {4},
     PAGES = {1087--1149},
      ISSN = {0002-9327,1080-6377},
   MRCLASS = {16G30 (16G20 16S38)},
  MRNUMBER = {2427009},
MRREVIEWER = {Hugh\ Ross\ Thomas},
       DOI = {10.1353/ajm.0.0011},
       URL = {https://doi.org/10.1353/ajm.0.0011},
}

@article {IyamaWemyss,
    AUTHOR = {Iyama, Osamu and Wemyss, Michael},
     TITLE = {Maximal modifications and {A}uslander-{R}eiten duality for
              non-isolated singularities},
   JOURNAL = {Invent. Math.},
  FJOURNAL = {Inventiones Mathematicae},
    VOLUME = {197},
      YEAR = {2014},
    NUMBER = {3},
     PAGES = {521--586},
      ISSN = {0020-9910,1432-1297},
   MRCLASS = {13C14 (13D09 16G70)},
  MRNUMBER = {3251829},
MRREVIEWER = {Siamak\ Yassemi},
       DOI = {10.1007/s00222-013-0491-y},
       URL = {https://doi.org/10.1007/s00222-013-0491-y},
}

@article{IWtits,
  title={Tits cone intersections and applications},
  author={Iyama, Osamu and Wemyss, Michael},
  journal={ preprint, available at \url{https://www.maths.gla.ac.uk/~mwemyss/MainFile_for_web.pdf}},
  year={2023}
}

@book {kac,
    AUTHOR = {Kac, Victor G.},
     TITLE = {Infinite-dimensional {L}ie algebras},
   EDITION = {Third},
 PUBLISHER = {Cambridge University Press, Cambridge},
      YEAR = {1990},
     PAGES = {xxii+400},
      ISBN = {0-521-37215-1; 0-521-46693-8},
   MRCLASS = {17B65 (17B67 17B68 58F07)},
  MRNUMBER = {1104219},
       DOI = {10.1017/CBO9780511626234},
       URL = {https://doi.org/10.1017/CBO9780511626234},
}

@article {KN,
    AUTHOR = {Karpov, Boris V. and Nogin, Dmitri Yu.},
     TITLE = {Three-block exceptional sets on del {P}ezzo surfaces},
   JOURNAL = {Izv. Ross. Akad. Nauk Ser. Mat.},
  FJOURNAL = {Izvestiya Rossiiskoi Akademii Nauk. Seriya Matematicheskaya},
    VOLUME = {62},
      YEAR = {1998},
    NUMBER = {3},
     PAGES = {3--38},
      ISSN = {1607-0046,2587-5906},
   MRCLASS = {14J60 (14F05)},
  MRNUMBER = {1642152},
MRREVIEWER = {Jaros\l aw\ A.\ Wi\'sniewski},
       DOI = {10.1070/im1998v062n03ABEH000205},
       URL = {https://doi.org/10.1070/im1998v062n03ABEH000205},
}

@article {KNP,
    AUTHOR = {Kasprzyk, Alexander and Nill, Benjamin and Prince, Thomas},
     TITLE = {Minimality and mutation-equivalence of polygons},
   JOURNAL = {Forum Math. Sigma},
  FJOURNAL = {Forum of Mathematics. Sigma},
    VOLUME = {5},
      YEAR = {2017},
     PAGES = {Paper No. e18, 48},
      ISSN = {2050-5094},
   MRCLASS = {14J26 (14J33 14M25 52B20)},
  MRNUMBER = {3686766},
MRREVIEWER = {Dimitrios\ I.\ Dais},
       DOI = {10.1017/fms.2017.10},
       URL = {https://doi.org/10.1017/fms.2017.10},
}

@article {KO,
    AUTHOR = {Kuleshov, Sergey A. and Orlov, Dmitri O.},
     TITLE = {Exceptional sheaves on {D}el {P}ezzo surfaces},
   JOURNAL = {Izv. Ross. Akad. Nauk Ser. Mat.},
  FJOURNAL = {Izvestiya Rossiiskoi Akademii Nauk. Seriya Matematicheskaya},
    VOLUME = {58},
      YEAR = {1994},
    NUMBER = {3},
     PAGES = {53--87},
      ISSN = {1607-0046,2587-5906},
   MRCLASS = {14J60},
  MRNUMBER = {1286839},
MRREVIEWER = {Yuri\ G.\ Prokhorov},
       DOI = {10.1070/IM1995v044n03ABEH001609},
       URL = {https://doi.org/10.1070/IM1995v044n03ABEH001609},
}

@article {lutz,
    AUTHOR = {Lutz, Wendelin},
     TITLE = {Mirrors to del {P}ezzo surfaces and the classification of
              {$T$}-polygons},
   JOURNAL = {SIGMA Symmetry Integrability Geom. Methods Appl.},
  FJOURNAL = {SIGMA. Symmetry, Integrability and Geometry. Methods and
              Applications},
    VOLUME = {20},
      YEAR = {2024},
     PAGES = {Paper No. 095, 20},
      ISSN = {1815-0659},
   MRCLASS = {14J33 (14E07)},
  MRNUMBER = {4843384},
       DOI = {10.3842/SIGMA.2024.095},
       URL = {https://doi.org/10.3842/SIGMA.2024.095},
}

@article {mandel,
    AUTHOR = {Mandel, Travis},
     TITLE = {Classification of rank 2 cluster varieties},
   JOURNAL = {SIGMA Symmetry Integrability Geom. Methods Appl.},
  FJOURNAL = {SIGMA. Symmetry, Integrability and Geometry. Methods and
              Applications},
    VOLUME = {15},
      YEAR = {2019},
     PAGES = {Paper 042, 32},
      ISSN = {1815-0659},
   MRCLASS = {13F60 (14J32)},
  MRNUMBER = {3954363},
MRREVIEWER = {Lara\ Bossinger},
       DOI = {10.3842/SIGMA.2019.042},
       URL = {https://doi.org/10.3842/SIGMA.2019.042},
}

@book {manin,
    AUTHOR = {Manin, Yuri I.},
     TITLE = {Cubic forms},
    SERIES = {North-Holland Mathematical Library},
    VOLUME = {4},
   EDITION = {Second},
      NOTE = {Algebra, geometry, arithmetic,
              Translated from the Russian by M. Hazewinkel},
 PUBLISHER = {North-Holland Publishing Co., Amsterdam},
      YEAR = {1986},
     PAGES = {x+326},
      ISBN = {0-444-87823-8},
   MRCLASS = {11Gxx (14Gxx 14J20)},
  MRNUMBER = {833513},
}

@article {mizoguchi,
    AUTHOR = {Mizoguchi, Shun'ya and Yamada, Yasuhiko},
     TITLE = {{$W(E_{10})$} symmetry, {M}-theory and {P}ainlev\'e{}
              equations},
   JOURNAL = {Phys. Lett. B},
  FJOURNAL = {Physics Letters. B. Particle Physics, Nuclear Physics and
              Cosmology},
    VOLUME = {537},
      YEAR = {2002},
    NUMBER = {1-2},
     PAGES = {130--140},
      ISSN = {0370-2693,1873-2445},
   MRCLASS = {81T30 (14J99 34M55)},
  MRNUMBER = {1910991},
MRREVIEWER = {Johannes\ Walcher},
       DOI = {10.1016/S0370-2693(02)01870-1},
       URL = {https://doi.org/10.1016/S0370-2693(02)01870-1},
}

@article {Mizuno,
    AUTHOR = {Mizuno, Yuma},
     TITLE = {{$q$}-{P}ainlev\'e{} equations on cluster {P}oisson varieties
              via toric geometry},
   JOURNAL = {Selecta Math. (N.S.)},
  FJOURNAL = {Selecta Mathematica. New Series},
    VOLUME = {30},
      YEAR = {2024},
    NUMBER = {2},
     PAGES = {Paper No. 19, 37},
      ISSN = {1022-1824,1420-9020},
   MRCLASS = {13F60 (14M25 34M55 39A13)},
  MRNUMBER = {4695875},
MRREVIEWER = {Xueqing\ Chen},
       DOI = {10.1007/s00029-023-00906-2},
       URL = {https://doi.org/10.1007/s00029-023-00906-2},
}

@article{nord,
  title={{F}ull exceptional collections on
{F}ano varieties and mutations},
  author={{N}ordskova, {A}nya},
  journal={Doctoral dissertation, {H}asselt {U}niversity},
  year={2025}
}

@article{NvdB,
  title={{NCCR}s of cones over del {P}ezzo surfaces},
  author={{N}ordskova, {A}nya and {V}an den Bergh, {M}ichel},
  journal={arXiv preprint arXiv:2604.11319},
  year={2026}
}

@article {perling,
    AUTHOR = {Perling, Markus},
     TITLE = {Combinatorial aspects of exceptional sequences on (rational)
              surfaces},
   JOURNAL = {Math. Z.},
  FJOURNAL = {Mathematische Zeitschrift},
    VOLUME = {288},
      YEAR = {2018},
    NUMBER = {1-2},
     PAGES = {243--286},
      ISSN = {0025-5874,1432-1823},
   MRCLASS = {14F05 (14J26 14J29 14M25 32S25)},
  MRNUMBER = {3774412},
MRREVIEWER = {Benjamin\ Schmidt},
       DOI = {10.1007/s00209-017-1887-y},
       URL = {https://doi.org/10.1007/s00209-017-1887-y},
}

@article {Sak,
    AUTHOR = {Sakai, Hidetaka},
     TITLE = {Rational surfaces associated with affine root systems and
              geometry of the {P}ainlev\'e{} equations},
   JOURNAL = {Comm. Math. Phys.},
  FJOURNAL = {Communications in Mathematical Physics},
    VOLUME = {220},
      YEAR = {2001},
    NUMBER = {1},
     PAGES = {165--229},
      ISSN = {0010-3616,1432-0916},
   MRCLASS = {14H10 (14D06 14J26 32G20 34M55 37K20 37K35)},
  MRNUMBER = {1882403},
MRREVIEWER = {I.\ Dolgachev},
       DOI = {10.1007/s002200100446},
       URL = {https://doi.org/10.1007/s002200100446},
}

@article {segal,
    AUTHOR = {Segal, Ed},
     TITLE = {The {$A_\infty$} deformation theory of a point and the derived
              categories of local {C}alabi-{Y}aus},
   JOURNAL = {J. Algebra},
  FJOURNAL = {Journal of Algebra},
    VOLUME = {320},
      YEAR = {2008},
    NUMBER = {8},
     PAGES = {3232--3268},
      ISSN = {0021-8693,1090-266X},
   MRCLASS = {16E30 (16G20)},
  MRNUMBER = {2450725},
MRREVIEWER = {Mikael\ Vejdemo Johansson},
       DOI = {10.1016/j.jalgebra.2008.06.019},
       URL = {https://doi.org/10.1016/j.jalgebra.2008.06.019},
}

@article{tevelev2022categorical,
  title={Categorical aspects of the {K}oll\'ar--{S}hepherd-{B}arron correspondence},
  author={{T}evelev, {J}enia and {U}rz\'ua, {G}iancarlo},
  journal={arXiv preprint arXiv:2204.13225},
  year={2022}
}

@article{bergh2002non,
  title={Non-commutative crepant resolutions},
  author={{V}an den Bergh, {M}ichel},
  journal={arXiv preprint math/0211064},
  year={2002}
}

@article{urzua2025wahl,
  title={{W}ahl singularities in degenerations of del {P}ezzo surfaces},
  author={{U}rz\'ua, {G}iancarlo and {Z}\'uniga, {J}uan {P}ablo},
  journal={arXiv preprint arXiv:2504.19929},
  year={2025}
}

@article {vianna,
    AUTHOR = {Vianna, Renato},
     TITLE = {Infinitely many monotone {L}agrangian tori in del {P}ezzo
              surfaces},
   JOURNAL = {Selecta Math. (N.S.)},
  FJOURNAL = {Selecta Mathematica. New Series},
    VOLUME = {23},
      YEAR = {2017},
    NUMBER = {3},
     PAGES = {1955--1996},
      ISSN = {1022-1824,1420-9020},
   MRCLASS = {53D12 (53D05)},
  MRNUMBER = {3663599},
MRREVIEWER = {Yingbo\ Han},
       DOI = {10.1007/s00029-017-0312-z},
       URL = {https://doi.org/10.1007/s00029-017-0312-z},
}

@article {vianna0,
    AUTHOR = {Vianna, Renato},
     TITLE = {Infinitely many exotic monotone {L}agrangian tori in
              $\mathbb{P}^2$},
   JOURNAL = {J. Topol.},
  FJOURNAL = {Journal of Topology},
    VOLUME = {9},
      YEAR = {2016},
    NUMBER = {2},
     PAGES = {535--551},
      ISSN = {1753-8416,1753-8424},
   MRCLASS = {53D12 (53D37 53D42)},
  MRNUMBER = {3509972},
MRREVIEWER = {Stefan\ Nemirovski},
       DOI = {10.1112/jtopol/jtw002},
       URL = {https://doi.org/10.1112/jtopol/jtw002},
}

@article {deTh_vdB,
    AUTHOR = {de Thanhoffer de V\"olcsey, Louis and Van den Bergh, Michel},
     TITLE = {Some new examples of nondegenerate quiver potentials},
   JOURNAL = {Int. Math. Res. Not. IMRN},
  FJOURNAL = {International Mathematics Research Notices. IMRN},
      YEAR = {2013},
    NUMBER = {20},
     PAGES = {4672--4686},
      ISSN = {1073-7928,1687-0247},
   MRCLASS = {16G20},
  MRNUMBER = {3118872},
MRREVIEWER = {Michael\ Wemyss},
       DOI = {10.1093/imrn/rns182},
       URL = {https://doi.org/10.1093/imrn/rns182},
}

@incollection {vdB_ICM,
    AUTHOR = {Van den Bergh, Michel},
     TITLE = {Noncommutative crepant resolutions, an overview},
 BOOKTITLE = {I{CM}---{I}nternational {C}ongress of {M}athematicians. {V}ol.
              2. {P}lenary lectures},
     PAGES = {1354--1391},
 PUBLISHER = {EMS Press, Berlin},
      YEAR = {2023},
      ISBN = {978-3-98547-060-0; 978-3-98547-560-5; 978-3-98547-058-7},
   MRCLASS = {14A22 (14E30 18G80)},
  MRNUMBER = {4680283},
}

\end{document}